\documentclass[11pt]{amsart}
\usepackage{amssymb,latexsym,amsmath,longtable,mathdots}
\usepackage[mathscr]{eucal}
\usepackage{bbm}
\usepackage{color,lscape}

\numberwithin{equation}{section}

\def\a{\alpha}  
\def\b{\beta}  
\def\c{\gamma}  
\def\d{\delta}

\def\z{\zeta}
\def\m{\mu}

\def\s{\sigma}

\def\w{\omega} 

\def\x{\xi}
 
\def\sA{\mathscr{A}} 
\def\fa{\mathfrak{a}} 
  
\def\tAd{\mathrm{Ad}} \def\tad{\mathrm{ad}}
 \def\tAnn{\mathrm{Ann}}
\def\tAut{\mathrm{Aut}}

\def\fb{\mathfrak{b}} 

\def\bC{\mathbb C} 
 
\def\fc{\mathfrak{c}}  
\def\bc{\mathbf{c}}

\def\cD{\mathcal D}
 
\def\fd{\mathfrak{d}} 
\def\td{\mathrm{d}}

 \def\tdim{\mathrm{dim}}

 \def\fe{\mathfrak{e}}

\def\ff{\mathfrak{f}}

\def\tFlag{\mathrm{Flag}}

\def\tGr{\mathrm{Gr}}
\def\fg{{\mathfrak{g}}}

 \def\ttH{\mathtt{H}}

\def\tHom{\mathrm{Hom}}
\def\fh{\mathfrak{h}} 
\def\bh{\mathbf{h}}
\def\bi{\mathbf{i}} 
 
\def\bI{\mathbf{I}}  \def\cI{\mathcal I} \def\sI{\mathscr{I}}

 \def\tIm{\mathrm{Im}}
 \def\tim{\mathrm{im}}

\def\ttJ{\mathtt{J}}
 
\def\fk{\mathfrak{k}} 

\def\sfk{\mathsf{k}} 
 \def\tker{\mathrm{ker}}
 
 \def\ttL{\mathtt{L}}
\def\fl{\mathfrak{l}} 
\def\tLG{\mathrm{LG}}
  
\def\fm{\mathfrak{m}}
\def\sfm{\mathsf{m}} 
\def\tmax{\mathrm{max}} 
\def\tmod{\mathrm{mod}}

\def\fn{\mathfrak{n}}

\def\bP{\mathbb P}

\def\fp{\mathfrak{p}} 
\def\sfp{\mathsf{p}}

\def\bQ{\mathbb Q} \def\cQ{\mathcal Q}
 
\def\fq{\mathfrak{q}} \def\bq{\mathbf{q}}
\def\sfq{\mathsf{q}} 
\def\bR{\mathbb R}

 \def\cS{\mathcal S}

\def\fs{\mathfrak{s}}
 
\def\tSL{\mathrm{SL}} \def\tSO{\mathrm{SO}}
\def\tSp{\mathrm{Sp}} \def\tSpin{\mathit{Spin}}
 \def\tSing{\mathrm{Sing}}
 \def\tSym{\mathrm{Sym}}

 \def\tspan{\mathrm{span}}
\def\fsl{\mathfrak{sl}} \def\fso{\mathfrak{so}} 
\def\fsp{\mathfrak{sp}} \def\fsu{\mathfrak{su}}
\def\cT{\mathcal T} 
 \def\ttT{\mathtt{T}}
\def\ft{\mathfrak{t}} 
 
\def\bU{\mathbb U}

 \def\cV{\mathcal V}

\def\bx{\mathbf{x}}
 \def\bZ{\mathbb Z}

\def\fz{\mathfrak{z}} 

\def\half{\tfrac{1}{2}}

\def\one{\mathbbm{1}}

\def\dfn{\stackrel{\hbox{\tiny{dfn}}}{=}}
\def\sbullet{{\hbox{\tiny{$\bullet$}}}}

\def\op{\oplus}
\def\ot{\otimes}

\def\wt{\widetilde}

\def\tw{\hbox{\small $\bigwedge$}}

\def\wtL{{\Lambda_\mathrm{wt}}}
\def\rtL{{\Lambda_\mathrm{rt}}}

\newcounter{numcnt}

\newcounter{cnt}

\newcounter{acnt}
\newenvironment{a_list_nem}{ 
  \begin{list}{{(\alph{acnt})}}
   {\usecounter{acnt} \setlength{\itemsep}{3pt}
    \setlength{\leftmargin}{25pt} \setlength{\labelwidth}{20pt} }
   }
   {\end{list}}
\newenvironment{a_list}{ 
  \begin{list}{{\emph{(\alph{acnt})}}}
   {\usecounter{acnt} \setlength{\itemsep}{3pt}
    \setlength{\leftmargin}{25pt} \setlength{\labelwidth}{20pt} }
   }
   {\end{list}}
\newcounter{Acnt}

\newcounter{icnt}
\newenvironment{i_list}{ 
  \begin{list}{{\emph{(\roman{icnt})}}}
   {\usecounter{icnt} \setlength{\itemsep}{3pt}
    \setlength{\leftmargin}{25pt} \setlength{\labelwidth}{20pt} }
   }
   {\end{list}}
\newenvironment{i_list_nem}{ 
  \begin{list}{{(\roman{icnt})}}
   {\usecounter{icnt} \setlength{\itemsep}{3pt}
    \setlength{\leftmargin}{25pt} \setlength{\labelwidth}{20pt} }
   }
   {\end{list}}
\newcounter{Icnt}

\newcounter{exam_cnt}

\newcounter{mccnt}
\newenvironment{circlist}{ 
  \begin{list}{$\circ$}
   {\usecounter{cnt} \setlength{\itemsep}{2pt}
    \setlength{\leftmargin}{15pt} \setlength{\labelwidth}{20pt} }
   }
   {\end{list}}

\newtheorem{corollary}[equation]{Corollary}
\newtheorem*{corollary*}{Corollary}
\newtheorem{lemma}[equation]{Lemma}
\newtheorem*{lemma*}{Lemma}
\newtheorem{proposition}[equation]{Proposition}
\newtheorem*{proposition*}{Proposition}
\newtheorem{theorem}[equation]{Theorem}
\newtheorem*{theorem*}{Theorem}

\theoremstyle{remark}
\newtheorem*{assume*}{Assume}

\newtheorem*{claim*}{Claim}

\newtheorem{definition}[equation]{Definition}
\newtheorem*{definition*}{Definition}
\newtheorem{example}[equation]{Example}
\newtheorem*{example*}{Example}
\newtheorem*{hint*}{Hint}
\newtheorem*{notation*}{Notation}
\newtheorem*{question*}{Question}
\newtheorem*{answer*}{Answer}
\newtheorem{remark}[equation]{Remark}
\newtheorem*{remark*}{Remark}
\newtheorem*{remarks*}{Remarks}
\newtheorem*{fact*}{Fact}

\numberwithin{HWeq}{section}
\theoremstyle{definition}

\def\fm{{\mathfrak{g}}}
\def\mss{\fm_0^\mathrm{ss}}
\def\MR{{G_\bR}} 
\def\mR{{\fm_\bR}} 

\def\hp{{\fh_\varphi}} \def\Hp{{H_\varphi}}
\def\hpp{{\fh_\varphi^\perp}}
\def\bvsigma{{\mathbf{\varsigma}}}

\def\FI{\mathrm{F}\,\mathrm{I}}
\def\FII{\mathrm{F}\,\mathrm{II}}
\begin{document}
\title[Schubert variations of Hodge structure]{Schubert varieties as variations of Hodge structure}
\author[Robles]{C. Robles}
\email{robles@math.tamu.edu}
\address{Mathematics Department, Mail-stop 3368, Texas A\&M University, College Station, TX  77843-3368} 
\thanks{Robles is partially supported by NSF DMS-1006353.}
\date{\today}
\begin{abstract}
We (1) characterize the Schubert varieties that arise as variations of Hodge structure (VHS); (2) show that the isotropy orbits of the infinitesimal Schubert VHS `span' the space of all infinitesimal VHS; and (3) show that the cohomology classes dual to the Schubert VHS form a basis of the invariant characteristic cohomology associated to the infinitesimal period relation (a.k.a. Griffiths transversality).
\end{abstract}
\maketitle
\setcounter{tocdepth}{1}

\section{Introduction} \label{S:intro}

A variation of Hodge structure (VHS) is a complex submanifold of a period domain $\cD$ that satisfies the infinitesimal period relation (IPR), also known as Griffiths' transversality.  The period domain is an open orbit of the compact dual $\check\cD$, a rational homogeneous variety, and the IPR extends to an invariant differential system on $\check\cD$.  The homology classes $[X]$ of the Schubert varieties $X \subset\check\cD$ form an additive basis of the integer homology $H_\sbullet(\check\cD,\bZ)$; so it is natural to ask:  Which Schubert varieties $X \subset \check\cD$ are integrals of the IPR?  That is, \emph{which Schubert varieties arise as variations of Hodge structure?}

Our starting point is the recent Green--Griffiths--Kerr structure theorem, which motivates us to pose the question, not only for period domains $\cD$, but in the more general Hodge domains $D$.  The structure theorem is discussed in greater detail in Section \ref{S:dfns}; roughly speaking, the idea is as follows.  Like period domains, Hodge domains are complex homogeneous manifolds $D = G_\bR/H$ admitting an invariant differential system $\sI$.  The Hodge domains admit various realizations as Mumford--Tate domains.  Mumford--Tate domains are in turn the proper context in which to consider variations of Hodge structure; for (the lift of) any global variation of Hodge structure $\Phi$ is contained in a Mumford--Tate domain $D_\Phi \subset \cD$.  (The Mumford--Tate domain $D_\Phi$ is, in general, a proper submanifold of the period domain $\cD$.)  Moreover, the variation of Hodge structure, a priori an integral of the infinitesimal period relation on $\cD$, is an integral of $\sI$.  And conversely, any integral $Y\subset D_\Phi$ of $\sI$, is a variation of Hodge structure; that is, an integral of the infinitesimal period relation on $\cD$.  This justifies our referring $\sI$ as the `infinitesimal period relation on the Hodge domain $D$,' and to to integrals $Y \subset D$ of $\sI$ as `variations of Hodge structure.'

The compact dual $\check D = G_\bC/P$ of the Hodge domain is a rational homogeneous variety.  As above, the Schubert classes form an additive basis of the integer homology of $\check D$, and we generalize the motivating question to: \emph{which Schubert varieties $X \subset \check D$ are variations of Hodge structure?}  The answer is given by Theorem \ref{T:SchubInt}.  We call these variations of Hodge structure `Schubert VHS.'  Corollaries to the theorem include:
\begin{i_list_nem}
\item Define a `maximal Schubert VHS' to be a maximal element (with respect to containment) of the set of Schubert VHS, and a `maximal VHS' to be a maximal element of the set of VHS.  A priori, a maximal Schubert VHS need not be a maximal VHS.  However, Proposition \ref{P:SchubMax} asserts that this is indeed the case:  any maximal Schubert VHS is a maximal variation of Hodge structure.
\item If the Hodge group $G_\bC$ is either special linear (type $A$) or symplectic (type $C$), then the maximal Schubert VHS are homogeneously embedded Hermitian symmetric spaces (Proposition \ref{P:sm}).  This need not be the case for other simple Hodge groups; see the examples of Appendix \ref{A:egs}.
\end{i_list_nem}

The second main result of the paper is Theorem \ref{T:intelem} which describes the entire space of infinitesimal variations of Hodge structure; roughly speaking, this space is spanned by the isotropy orbits of the infinitesimal Schubert variations of Hodge structure.  In particular, there exists a VHS of dimension $d$ if and only if there exists a Schubert VHS of dimension $d$.  This yields counter-examples to a rigidity theorem of R. Mayer for maximal variations of Hodge structure (Examples \ref{eg:C5P25} and \ref{eg:B6P35}).

The final main result of the paper is a description of the invariant characteristic cohomology (ICC) of the IPR by Theorem \ref{T:icc}.  Loosely speaking, the ICC describes the topological invariants of a global variation of Hodge structure that can be defined independently of the monodromy group.  Roughly speaking, we find that the ICC is spanned by cohomology classes dual to the homology classes of the Schubert VHS.  This answers the question posed by Green--Griffiths--Kerr \cite[p. 301]{MR2666359} to explicitly identify invariant representatives.

\smallskip

Variations of Hodge structure have been of interest since introduced by P. Griffiths in the 1960s; the interested reader might consult the following sample \cite{MR2165541, MR717607, MR720288, MR894385, MR974782, MR1016441, MR2048516, MR720291, FL, MR720289, MR0229641, MR0233825, MR720290}, and the references therein.


\tableofcontents

In addition to the main results outlined above, I wish to point out a few other aspects of the paper.  Grading elements play a salient role in the paper; they are reviewed in Section \ref{S:GE}, and their relationship to Hodge structures is established in Section \ref{S:HS}.  As an illustration of their utility, we use grading elements to give a simple characterization of Hodge representations of Calabi--Yau type (Proposition \ref{P:CY}).  This generalizes \cite[Lemma 2.27]{FL}, which gives a necessary condition for a Hodge representation to be of Calabi--Yau type in the case that the corresponding Hodge domain is Hermitian symmetric.  

Proposition \ref{P:red} is a reduction theorem which asserts that, for the purpose of studying the differential system, we may assume that the IPR is bracket--generating.  This is equivalent to \eqref{E:TI}; when considering results that follow, keep in mind that this assumption will be in effect from that point on. 

Finally, many examples are presented in Appendix \ref{A:egs}.

\subsection*{Acknowledgements}
I benefited from conversations/correspondence with many colleagues.  I would especially like to thank J. Daniel, P. Griffiths, M. Kerr and R. Laza for their time and insight.

\section{Preliminaries} \label{S:prelim}

In Section \ref{S:dfns} we review the relevant definitions and results from Hodge theory, and set notation.  Section \ref{S:GE} presents the relationship between parabolic algebras and grading elements; the latter will play a prominent role in the representation theoretic computations of the paper.  Section \ref{S:HS} establishes the correspondence between grading elements and Hodge structures.

\subsection{Hodge groups and representations} \label{S:dfns}

\noindent This section reviews some definitions and results from \cite{MR2918237}.

Fix $n \in \bZ$.  Let $V$ be a $\bQ$--vector space.  Set $\bi = \sqrt{-1}$, and let $S^1 \subset \bC$ be the unit circle.

\begin{definition} \label{D:hs}
A \emph{Hodge structure of weight $n\in\bZ$} on $V$ is a homomorphism $\varphi : \bU(\bR) \simeq S^1 \to \tSL(V_\bR)$ of $\bR$--algebraic groups such that $\varphi(-1) = (-1)^n \one$.  The associated \emph{Hodge decomposition} 
$V_\bC = \op_{p+q=n}V^{p,q}$, with $\overline{V^{p,q}} = V^{q,p}$, is given by
$$
  V^{p,q} = \{ v \in V_\bC \ | \ \varphi(z) v = z^{p-q} v \,,\ \forall \ z\in S^1 \}\,.
$$
(Any of the integers $p,q,n$ may be negative.)  The corresponding \emph{Hodge filtration} is given by $F^p = \op_{q\ge p}V^{q,n-q}$.  

The \emph{Hodge numbers} $\bh = (h^{p,q})$ and $\mathbf{f} = (f^p)$ are
$$
  h^{p,q} \ = \ \tdim_\bC\,V^{p,q}
  \quad\hbox{and}\quad
  f^p \ = \ \tdim_\bC\,F^p \, .
$$

A Hodge structure of weight $n\ge0$ is \emph{effective} if $V^{p,q} \not= 0$ only when $p,q\ge0$.
\end{definition}

\begin{definition} \label{D:phs}
Let $Q : V \times V \to \bQ$ be a non-degenerate bilinear form satisfying $Q(u,v) = (-1)^nQ(v,u)$ for all $u,v\in V$.  A Hodge structure $(V,\varphi)$ of weight $n$ is \emph{polarized} by $Q$ if the \emph{Hodge--Riemann bilinear relations} hold: 
$$
  Q(F^p , F^{n-p+1}) = 0 \,,\quad\hbox{and}\quad
  Q(v,\varphi(\bi)\bar v) > 0 \ \forall \ 0\not=v\in V_\bC \,. 
$$
\end{definition}

\noindent Throughout we assume that the Hodge structure $\varphi$ is polarized.  

The \emph{Mumford--Tate group} $M_\varphi$ associated to $\varphi$ is the $\bQ$--algebraic closure of $\tIm\,\varphi$ in $\tSL(V_\bR)$. It is the maximal subgroup acting trivially on the Hodge tensors \cite[(I.B.1)]{MR2918237}.  As a consequence, the Mumford--Tate group is reductive \cite[(I.B.6)]{MR2918237}.  Definition \ref{D:phs} implies that 
\begin{equation} \label{E:ctmt}
  \varphi(S^1) \ \subset \ M_\varphi \ \subset \ \tAut(V_\bR,Q) \, .
\end{equation}
Additionally, the stabilizer $H_\varphi \subset M_\varphi(\bR)$ of $\varphi$ is compact ($Q$ is a Hodge tensor), and equal to the centralizer of $\varphi(S^1)$.  

Fix $(V,Q)$.  Given $\bh = (h^{p,q})$, the \emph{period domain} $\cD = \cD_\bh$ is the set of Hodge structures $\varphi$ on $V$ polarized by $Q$ and with $\tdim\,V^{p,q} = h^{p,q}$.  The period domain is a homogeneous manifold with respect to the $\tAut(V,Q)$--action: given $\varphi \in \cD$, we have $\cD = \tAut(V_\bR,Q) \cdot \varphi$.  The orbit $D_\varphi = M_\varphi(\bR)\cdot\varphi \subset \cD$ is a \emph{Mumford--Tate domain}.  The \emph{compact dual} $\check D_\varphi = M_\varphi(\bC)\cdot\varphi$ is a rational homogeneous variety \cite{MR0066396}.  Let $P_\varphi \subset M_\varphi(\bC)$ be the stabilizer of $\varphi$.  Then 
$$
  D_\varphi \ \simeq \ M_\varphi(\bR)/H_\varphi 
  \quad\hbox{ and }\quad
  \check D_\varphi \ \simeq \ M_\varphi(\bC)/P_\varphi \, ,
$$
as complex manifolds.

As the structure theorem below indicates, Mumford--Tate domains are natural objects for the study of VHS.  Let $\Phi : S \to \Gamma\backslash \cD$ be a global variation of Hodge structure.  Modulo isogeny, the image $\Phi(S)$ is contained in $\Gamma \backslash D_\Phi$, where $D_\Phi$ is a Mumford--Tate domain for the Mumford--Tate group $M_\Phi$ of the variation of Hodge structure \cite[(III)]{MR2918237}.  More precisely, let $S$ be smooth and quasi-projective.  We assume $V = V_\bZ \ot_\bZ \bQ$ and $\Gamma \subset \tAut(V,Q)\cap\tAut(V_\bZ)$ is the monodromy group.  Let $\wt S$ be the universal cover of $S$, and $\wt\Phi : \wt S \to \cD$ the lift of $\Phi$.  Fix a generic point $\tilde \eta \in \wt S$, and let $\tilde\varphi = \tilde\Phi(\tilde \eta)$.  The Mumford--Tate group $M_\Phi$ of $\Phi$ is the Mumford--Tate group $M_{\tilde\varphi}$ of $\tilde\varphi$, cf. \cite[(III)]{MR2918237}.  In particular, $M_\Phi$ is reductive and $\Gamma \subset M_\Phi$.  Let $M_\Phi = M_1\times\cdots\times M_t\times A$ be the almost product decomposition into $\bQ$--simple factors $M_s$ and abelian $A$.  Then the orbit $D_\Phi = M_\Phi(\bR)\cdot\tilde\varphi$ is a Mumford--Tate domain for $M_\Phi$, and $D_s = M_s(\bR)\cdot\tilde\eta \subset \cD$ is a Mumford--Tate domain for $M_s$.

\begin{theorem}[Structure Theorem {\cite[(III.A)]{MR2918237}}] \label{T:str}
Let $\Phi : S \to \Gamma\backslash\cD$ be a global variation of Hodge structure.
\begin{i_list}
\item The $D_s$ are complex submanifolds of $D_\Phi$.
\item Up to isogeny, the monodromy group splits $\Gamma = \Gamma_1 \times \cdots \times \Gamma_k$, with $k \le t$, and $\overline\Gamma{}_j^\bQ = M_j$.
\item Modulo isogeny, $\Phi(S) \subset \Gamma \backslash D_\Phi \ \simeq \ \Gamma_1\backslash D_1 \times \cdots \times \Gamma_k \backslash D_k \times (D_{k+1} \times\cdots\times D_t)$.
\end{i_list}
\end{theorem}

Led in part by the structure theorem, Green--Griffiths--Kerr \cite{MR2918237} ask: Which semisimple, $\bQ$--algebraic groups $G$ may be realized as Mumford--Tate groups of polarized Hodge structures?  (They show that the adjoint form of a simple $\bQ$--algebraic group $G$ is a Mumford--Tate group if and only if $G_\bR$ contains a compact maximal torus \cite[(IV.E.1)]{MR2918237}.)   What can be said about the different realizations of $G$ as a Mumford--Tate group?  What is the relationship between the associated Mumford--Tate domains?  To address these questions, Green--Griffiths--Kerr introduce the following notion.

\begin{definition} \label{D:Hrep}
A \emph{Hodge representation} $(G,\rho,\varphi)$ consists of a reductive $\bQ$--algebraic group $G$, a representation $\rho : G \to \tAut(V,Q)$, defined over $\bQ$, and a non-constant homomorphism $\varphi : S^1 \to G_\bR$ such that $(V,Q,\rho\circ\varphi)$ is a polarized Hodge structure.   
\end{definition}

\begin{theorem}[{\cite[(IV.A.2)]{MR2918237}}] \label{T:HR}
Assume $G$ is a semisimple $\bQ$--algebraic group.  If $G$ has a Hodge representation, then $G_\bR$ contains a compact maximal torus $T$ with $\varphi(S^1) \subset T$ and $\tdim(T) = \mathrm{rank}(G)$.  Moreover, the maximal subgroup $H_\varphi \subset G_\bR$ such that $\rho(H_\varphi)$ fixes the polarized Hodge structure $(V,Q,\rho\circ\varphi)$ is compact and is the centralizer of the circle $\varphi(S^1)$ in $G_\bR$.
\end{theorem}

Theorem \ref{T:Ad} describes the relationship between the Mumford--Tate domains associated to two differential realizations of $G$ as a Mumford--Tate group.  Let $B$ denote the \emph{Killing form} on $\fm$, and let $\tAd : G \to \tAut(\fm,B)$ denote the \emph{adjoint representation}.

\begin{theorem}[{\cite[(IV.F.1)]{MR2918237}}] \label{T:Ad}
The triple $(G,\rho,\varphi)$ is a Hodge representation if and only if $(G, \tAd , \varphi)$ is a Hodge representation.  In this case, the stabilizers of the two associated polarized Hodge structures $\rho\circ\varphi$ and $\tAd\circ\varphi$ agree.  Let $H_\varphi$ denote this stabilizer, so that the two Mumford--Tate domains are isomorphic to $G_\bR/H_\varphi$ as complex, homogeneous manifolds.  Then the two infinitesimal period relations\footnote{The infinitesimal period relation is defined in Section \ref{S:IPR}.} agree under this identification.
\end{theorem}

\begin{definition} \label{D:Hgp}
A \emph{Hodge group} is a semisimple $\bQ$--algebraic group $G$ admitting a Hodge representation.   Given a Hodge representation $(G , \tAd , \varphi)$, the associated \emph{Hodge domain} is 
$$
  D \ \dfn \ G(\bR)/H_\varphi \, ,
$$
and the \emph{compact dual} is 
$$
  \check D \ \dfn \ G(\bC) / P_\varphi \,.
$$ 
\end{definition}

\noindent The compact dual $\check D$ carries a canonical, invariant exterior differential system $\sI$ corresponding to the IPR associated with the polarized Hodge structure $(\fg,B,\tAd\circ\varphi)$.  The content of Theorem \ref{T:Ad} is that a Hodge domain may appear in many different ways as a Mumford-Tate domain; however, for each such realization, the IPRs on the various Mumford-Tate domains all coincide under the identification with the Hodge domain.  In particular, \emph{for the purpose of studying variations of Hodge structure (integrals of the IPR), it suffices to consider Hodge domains.}  This is the perspective adopted here.


\subsection{Parabolic subalgebras and grading elements} \label{S:GE}
Given a parabolic subalgebra $\fp \subset \fg$, let $\fg_0 \subset \fp$ denote the reductive subalgebra in the Levi decomposition of $\fp$.   A grading element $\ttT \in \ft_\bC$ is canonically associated with $\fp$.  Because the grading element acts semisimply (that is, diagonalizably) on every $\fg_0$--module, this association is a very useful tool in the study of parabolic algebras and parabolic geometries.  We will see in Section \ref{S:HS} that there is a bijection between Hodge structures $\varphi$ and grading elements.  We briefly review parabolic subalgebras and grading elements; an excellent reference is \cite[\S3.2]{MR2532439}.  

Let $\fg_\bC$ be a semisimple Lie algebra, and $\ft_\bC \subset \fg_\bC$ a Cartan subalgebra.  Let $\Delta \subset \ft_\bC^*$ be the \emph{roots} of $\fg_\bC$.  Given a root $\a\in\Delta$, let $\fm_\a\subset\fm_\bC$ denote the corresponding \emph{root space}.  Fix a Borel subalgebra $\fb \subset \fm_\bC$ containing $\ft_\bC$.  The choice of Borel subalgebra determines the set of \emph{positive roots} $\Delta^+ \subset \Delta$.  Our convention is that $\fb$ is the direct sum of the Cartan subalgebra $\ft_\bC$ and the \emph{positive} root spaces; that is,
\begin{equation}\label{E:b}
  \fb \ = \ \ft_\bC \,\op\, \Big( \bigoplus_{\a\in\Delta^+} \fm_\a \Big) \, .
\end{equation}
Let $\Sigma = \{\sigma_1,\ldots,\s_r\} \subset \Delta^+$ denote the \emph{simple roots}.  Given any subspace $\fs \subset \fm_\bC$, let 
$$
  \Delta(\fs) \ \dfn \ \{ \a\in\Delta \ | \ \fm_\a \subset \fs \}
  \quad\hbox{and}\quad
  \Sigma(\fs) \ \dfn \ \{ \s \in\Sigma \ | \ \fm_\s \subset \fs \} \,.
$$
For example, $\Delta(\fb) = \Delta^+$ is the set of positive roots.

\begin{definition*}
Let $\rtL \subset \wtL \subset \ft_\bC^*$ be the root and weight lattices of $\fg_\bC$.  The \emph{set of grading elements} is the lattice $\tHom(\rtL,\bZ) \subset \bi\ft$.
\end{definition*}

Let $\{\ttT_1,\ldots,\ttT_r\}$ be the basis of $\ft_\bC$ dual to the simple roots, 
\begin{equation} \label{E:Ti}
  \sigma_i(\ttT_j) \ = \ \d_{ij} \,.
\end{equation}
Then the $\ttT_i$ form a $\bZ$--basis of the grading elements,
$$
  \tHom( \rtL , \bZ ) \ = \ \tspan_\bZ\{ \ttT_1,\ldots,\ttT_r\} \, .
$$
Given a grading element $\ttT\in \tHom(\rtL,\bZ)$, every $\fm_\bC$--representation $U$ admits a \emph{$\ttT$--graded decomposition}
\begin{equation} \label{E:grU}
  U \ = \ \bigoplus_{\ell\in\bQ} U_\ell \,,\quad \hbox{with } \ 
  U_\ell \ \dfn \ \{ u \in U \ | \ \ttT(u) = \ell u \} \,,
\end{equation}
into $\ttT$--eigenspaces.  Each $U_\ell$ is the direct sum of weight spaces $U_\lambda\subset U$; here $\lambda\in\wtL$ is a weight of $U$ such that $\lambda(\ttT) = \ell$.  In the case that $U = \fm_\bC$, we have 
\begin{equation}\label{E:grm}
  \fm_\bC \ = \ \fg_\sfk \op \cdots \op \fg_{-\sfk} \,,\quad \hbox{with } \ 
  \fm_\ell \ = \ \{ \z\in\fm_\bC \ | \ [\ttT,\z] = \ell \z \} \,.
\end{equation}

\subsubsection{Graded structures} \label{S:gr}
Observe that $\fm_\ell$ maps $U_m$ into $U_{\ell+m}$,
\begin{equation}\label{E:mU}
  \fm_\ell(U_m) \ \subset \ U_{\ell+m} \,.
\end{equation}
This has two consequences.  First, in the case that $U = \fm_\bC$,
\begin{equation}\label{E:p}
  \fp_\ttT \ \dfn \ \fm_{\ge0} \ = \ \fm_0 \op\fm_1 \op\cdots\op\fg_\sfk
\end{equation}
is a parabolic subalgebra of $\fm_\bC$.  The reductive and nilpotent subalgebras in the Levi decomposition of $\fp_\ttT$ are, respectively, $\fg_0$ and 
\begin{equation} \nonumber 
  \fm_+ \ \dfn \ \fm_1 \op \cdots \op \fg_\sfk \, .
\end{equation}
Moreover, \eqref{E:mU} implies that 
$$
  \fm^\ell \ \dfn \ \fm_\ell \op \fm_{\ell+1} \op \cdots \op \fg_\sfk
$$
is a $\fp_\ttT$--module.  

Modulo the action of $G_\tad(\bC)$, the parabolic $\fp_\ttT$ contains our fixed Borel subalgebra $\fb$; that is $\fb \subset \tAd_g(\fp_\ttT)$ for some $g \in G_\tad(\bC)$.  Assuming $\fb \subset \fp_\ttT$, \eqref{E:b} and \eqref{E:Ti} imply
\begin{equation} \label{E:Tpos}
  \ttT \ = \ n^i \,\ttT_i \,, \quad\hbox{with } \ 0\le n^i \in \bZ \, .
\end{equation}
In this case, 
$$
  \fm_0 \ = \ \fz_\ttT \ \op \ \mss \,,
$$
where 
$$
  \fz_\ttT \ = \ \tspan_\bC\{ \ttT_i \ | \ n^i > 0 \}
$$
is the center of $\fm_0$, and 
$$
  \mss \ = \ [\fm_0 , \fm_0]
$$
is the semisimple subalgebra of $\fm_0$ generated by the simple roots $\Sigma(\fm_0) = \{ \s_i \ | \ n^i = 0 \}$.

The second consequence of \eqref{E:mU} is that each $U_\ell$ is a $\fm_0$--module.  Specializing to the case that $U = \fm_\bC$, the Killing form $B$ on $\fm_\bC$ induces an isomorphism
\begin{equation} \label{E:m*}
  \fm_\ell^* \ \simeq \ \fm_{-\ell}
\end{equation}
of $\fm_0$--modules.

\subsubsection{Grading elements associated to parabolic subalgebras}

We now describe a bijective correspondence between nonempty subsets $I \subset \{1,\ldots,r\}$, and parabolic subalgebras $\fp_I \subset \fm_\bC$ that contain the fixed Borel $\fb$, \cite[\S3.2.1]{MR2532439}.  Given $I$, define a grading element
\begin{equation} \label{E:gr}
  \ttT_I \ \dfn \ \sum_{i\in I} \ttT_i \ \in \ \tHom(\rtL,\bZ) \, .
\end{equation}
Let $\fp_I$ be the parabolic subalgebra \eqref{E:p} determined by $\ttT_I$.  Then $\fb \subset \fp_I$. Conversely, every parabolic $\fp \supset \fb$ is of the form $\fp = \fp_I$ for 
\begin{equation}\label{E:Ip}
  I \ = \ I(\fp) \ \dfn \ \{ i \ | \ \fm_{-\sigma_i} \not\subset \fp \} 
  \ = \ \{ i \ | \ \s_i(\ttT_I) \ge 0 \} \,.
\end{equation}

\begin{example*}
If $I = \{1,\ldots,r\}$, then the parabolic subalgebra $\fp_I$ is the Borel $\fb$.  If $I = \{ i\}$ consists of a single element, then the parabolic $\fp_I$ is a maximal parabolic.
\end{example*}

\begin{remark} \label{R:g1}
Recall the parabolic $\fp_\ttT$ of \eqref{E:p}.  Without loss of generality, $\fb \subset \fp_\ttT$, and $\ttT$ is of the form \eqref{E:Tpos}.  If $I = \{ i \ | \ n^i > 0\}$, then $\fp_\ttT = \fp_I$.  If $\op_\ell \fg_\ell(I)$ and $\op_s\fg_s(\ttT)$ are the $\ttT_I$ and $\ttT$--graded decompositions of $\fg_\bC$, then $\fg_0(\ttT) = \fg_0(I)$ and $\fg_+(\ttT) = \fg_+(I)$.  However, $\fg_\ell(I)=\fg_\ell(\ttT)$, for all $\ell$, if and only if $\ttT = \ttT_I$.  Moreover, $\fg_1(\ttT)$ generates $\fg_+(\ttT)$, as an algebra, if and only if $\ttT = \ttT_I$.
\end{remark}

\begin{definition*}
Given the parabolic $\fp$, we say that $\ttT_I$ is the \emph{grading element associated with the parabolic}.
\end{definition*}

\subsection{Hodge structures and grading elements} \label{S:HS}

In this section we will establish a correspondence between Hodge structures $\varphi$ and grading elements $\ttT$.  Roughly speaking, grading elements are infinitesimal Hodge structures, and we will show the following.
\begin{a_list_nem}
\item To every Hodge group $G$ and polarized Hodge structure $\varphi : S^1 \to T \subset G_\tad $ on $(\fm,B)$ we may canonically associate a grading element $\ttT_\varphi \in \tHom(\rtL,\bZ)$.  See \eqref{E:dfnT} and Lemma \ref{L:HSgr}.
\item Conversely, given a complex semisimple Lie algebra $\fm_\bC$, a Cartan subalgebra $\ft_\bC \subset \fm_\bC$ and a grading element $\ttT \in \tHom(\rtL,\bZ)$, there exists an integral form $\fm$ of $\fm_\bC$ and polarized Hodge structure $\varphi : S^1 \to \tAut(\fm_\bR,B)$ determined by $\ttT$.  The precise statement is given by Proposition \ref{P:HSgr}.
\end{a_list_nem}

Let $G$ be a Hodge group.  The connected, real Lie groups $G_\bR = G_\Lambda$ with Lie algebra $\fg_\bR$ are indexed by sub-lattices $\rtL \subset \Lambda \subset \wtL$.  The fundamental group and center satisfy, $\pi_1(G_\Lambda) = \wtL/\Lambda$ and $Z(G_\Lambda) = \Lambda/\rtL$.  The \emph{adjoint form} is $G_\tad = G_{\rtL}$, and the group $G_{\wtL}$ is simply connected.  Let 
$$
  \Lambda^* \ \dfn \ \tHom(\Lambda,2\pi\bi\bZ) \,.
$$
Then the maximal torus of $G_\Lambda$ is compact by Theorem \ref{T:HR}, and given by
$$
  T \ = \ \ft/\Lambda^* \,.
$$  
In general, a $\fg_\bR$--module $U$ admits the structure of a $G_\Lambda$--module if and only if the weights $\Lambda(U)$ of $U$ are contained in $\Lambda$.  

Suppose that $\varphi : S^1 \to G$ is a homomorphism of real algebraic groups with image $\varphi(S^1) \subset T = \ft_\bR / \Lambda^*$.  Let $\ttL_\varphi\in\Lambda^*$ be the lattice point such that the line $\bR\ttL_\varphi \subset \Lambda^*$ projects to the circle $\varphi(S^1) \subset T$.  That is, $\varphi(e^{2\pi\bi t}) = t \ttL_\varphi \ \tmod \ \Lambda^*$, for $t \in\bR$.  Fix a representation $\rho : G\to \tAut(V,Q)$ defined over $\bQ$.  Suppose that $\rho\circ\varphi$ defines a polarized Hodge structure on $V$.  Then the group element $\rho\circ\varphi(e^{2\pi\bi t}) \in \tAut(V_\bC)$ acts on $V^{p,q}$ by $e^{2\pi\bi(p-q) t} \mathbbm{1}$.  Taking the derivative $\partial/\partial t$ at $t=0$ implies that $\ttL_\varphi$ acts on $V^{p,q}$ by $2\pi\bi(p-q)\mathbbm{1}$. 
Define $\ttT_\varphi\in\tHom(\Lambda , \half\bZ) \subset \bi \ft\subset \ft_\bC$ by 
\begin{equation} \label{E:dfnT}
  4\pi\bi \ttT_\varphi \ \dfn \ \ttL_\varphi \,.
\end{equation}
Then $\ttT_\varphi$ acts on $V^{p,q}$ by the scalar $(p-q)/2 = p - n/2$.  Therefore, the $\ttT_\varphi$--graded decomposition of $V_\bC$ is
\begin{equation} \label{E:Vdecomp}
\renewcommand{\arraystretch}{1.3}
\begin{array}{rcl}
  V_\bC & = & V_{m/2} \op V_{m/2-1} \op\cdots\op V_{1-m/2} \op V_{-m/2} \,,
  \hbox{ where}\\
  V_\ell & = & \{ v \in V_\bC \ | \ \ttT_\varphi(v) = \ell v \}
  \quad\hbox{and}\quad V^{p,q} \ = \ V_{(p-q)/2} \ = \ V_{p-n/2}\,;
\end{array}
\end{equation}
above, we choose $0 \le m$ to be the smallest integer such that $V_{m/2} \not= 0$.  The Hodge filtration is 
\begin{equation} \label{E:F}
  F^p_\varphi \ = \ V_{p-n/2} \,\op\, V_{p+1-n/2}\,\op\cdots\op\,V_{m/2} \,.
\end{equation}
Since weights $\mu \in \wtL$ are real--valued on $\ttT_\varphi \in \bi \fm_\bR$, we have 
\begin{equation} \label{E:Vconj}
  \overline{V_\ell} \ \simeq \ V_{-\ell} \,,
\end{equation}
as $\fm_0$--modules.  Moreover, the $G$--invariant polarization $Q$ yields the $\fm_0$--module identification
\begin{equation} \label{E:Vdual}
  V_\ell^* \ \simeq \ V_{-\ell} \, .
\end{equation}

\begin{remark}\label{R:phii}
Since $\ttT_\varphi$ acts on $V_\ell$ by the scalar $\ell \in \half\bZ$, it follows that $\ttL_\varphi$ acts on $V_\ell$ by $4\pi\bi\ell \mathbbm{1}$.  Therefore, if $z\in S^1$, then $\varphi(z)$ acts on $V_\ell$ by $z^{2\ell}\mathbbm{1}$.  In particular, $\varphi(\bi)$ acts on $V_\ell$ by the scalar $\bi^{2\ell}$.  Note that, $m$ is even (respectively, odd) if and only if $2\ell$ is even (respectively, odd) for each $\ttT_\varphi$--eigenvalue $\ell$.  Therefore, $\varphi(\bi)$ acts on $V_\ell$ by $\pm\one$ (respectively, $\pm\bi\one$) if and only if $m$ is even (respectively, odd). 
\end{remark} 

\begin{remark}[Hodge tensors] \label{R:HT}
Let $T(V) = \op_{k,\ell} V^{\ot k} \ot (V^*)^{\ot\ell}$ denote the tensor algebra of $V$.  The Lie algebra action of $\fg_\bC$ on $V_\bC$ induces a Lie algebra action on the complex tensor algebra $T(V_\bC) = T(V) \ot_\bQ \bC$.  Let $\op_{q\in\bQ} T_q$ be the $\ttT_\varphi$--graded decomposition \eqref{E:grU} of $T(V_\bC)$ into $\ttT_\varphi$--eigenspaces.  The Hodge tensors (or Hodge classes)  
$$
  \mathrm{Hg}_\varphi \ = \ T_0 \,\cap\, T(V)
$$
are the $\bQ$--tensors in the kernel $T_0$ of $\ttT_\varphi$.
\end{remark}

By definition \eqref{E:dfnT}, $\ttT_\varphi \in \tHom(\Lambda,\half\bZ)$.  

\begin{lemma} \label{L:HSgr}
Let $(V,\rho,\varphi)$ be a Hodge representation of a semisimple, $\bQ$--algebraic Lie group $G_\Lambda$.  Then $\ttT_\varphi$ is a grading element.  That is, $\ttT_\varphi \in \tHom(\rtL,\bZ)\,\cap\,\tHom(\Lambda,\half\bZ)$.
\end{lemma}

\begin{proof}
It is a consequence of \eqref{E:ctmt} that $\varphi$ induces a Hodge structure of weight zero on the Lie algebra $\fg$, cf. \cite[(I.B.2)]{MR2918237}.  Explicitly, 
\begin{equation} \label{E:Hm}
  \fg_\bC \ = \ 
  \op_{\ell\in\bZ} \, \fg^{-\ell,\ell} \,,
  \quad\hbox{where } 
  \fg^{-\ell,\ell} \ = \ 
  \{ \z\in\fg_\bC \ | \ \z(V^{p,q}) \subset V^{p-\ell,q+\ell} \ \forall \ p,q \} \, .
\end{equation}
By \eqref{E:Vdecomp}, each $\fm^{-\ell,\ell}$ is the $\ttT_\varphi$--eigenspace of eigenvalue $-\ell$.  Since $\ttT_\varphi \in \ft_\bC$, this implies that each root space $\fm_\a \subset \fm_\bC$ is contained in one of the $\fm^{-\ell,\ell}$, and $\a(\ttT_\varphi) = -\ell \in \bZ$.
\end{proof}

Let $\fm_\bC = \op \fm_\ell$ be the $\ttT_\varphi$--graded decomposition \eqref{E:grm} of $\fm_\bC$, and let $\fp_\varphi\subset \fm_\bC$ be the associated parabolic subalgebra \eqref{E:p}.  Then $\fp_\varphi$ is the stabilizer in $\fm_\bC$ of the filtration \eqref{E:F}.  Since $\fp_\varphi$ is parabolic, we may choose the Borel subalgebra $\fb$ so that $\fb \subset \fp_\varphi$.  As observed in \eqref{E:Tpos}, 
\begin{equation} \label{E:preT}
  \ttT_\varphi \ = \ \sum n_i \ttT_i \,,\quad\hbox{with}\quad
  0 \le n_i \in \bZ\,.
\end{equation}
(We will see in Section \ref{S:TI} that, for the purposes of considering variations of Hodge structure, we may restrict to the case that $n_i = 0,1$.)

Let $\fh_\varphi \subset \fm_\bR$ be the Lie algebra of $H_\varphi$.  As noted in Theorem \ref{T:HR}, $H_\varphi$ is the centralizer of $\varphi(S^1)$.  It follows from \eqref{E:grm} that 
\begin{equation} \label{E:hpC}
  \fh_\varphi(\bC) \ \dfn \ \fh_\varphi \ot_\bR\bC = \fm_0 \, .
\end{equation}
The real form $\fm_\bR$ defines complex conjugation on $\fm_\bC = \fm_\bR \ot_\bR \bC$.  Since the roots are real--valued on $\ttT_\varphi \in \bi\ft \subset \bi\fm_\bR$, we have 
\begin{equation} \label{E:mellconj}
  \overline{\fm_\ell} \ \simeq \ \fm_{-\ell} \ \simeq \ \fm_\ell^*
\end{equation}
as $\fg_0$--modules.  (The first identification is a special case of \eqref{E:Vconj}, and the second identification is \eqref{E:m*}.)  

Since $T$ is compact, the algebra $\fm_\bR$ admits a Cartan decomposition $\fm_\bR = \fk_\bR \op \fq_\bR$ with $\ft \subset \fk_\bR$, cf. \cite[Proposition 6.59]{MR1920389}.  Moreover, the roots $\Delta = \Delta(\fm_\bC , \ft_\bC)$ take purely imaginary values on $\ft$; as a consequence, either $\fm_\a\subset \fk_\bC$ or $\fm_\a \subset \fp_\bC$, cf.  \cite[p.390]{MR1920389}.  Recall that the Killing form $B$ is negative-definite on $\fk_\bR$ and positive-definite on $\fq_\bR$.  Taken with Remark \ref{R:phii}, these observations yield the following.

\begin{proposition}[{\cite[(IV.B.3)]{MR2918237}}] \label{P:mB}
Let $\Lambda = \rtL$, so that $G_\Lambda = G_\tad$ is the adjoint form of the group.  Given a group homomorphism $\varphi:S^1 \to T = \ft_\bC/\rtL$, let $\ttT_\varphi$ be the associated element of $\tHom(\rtL,\half\bZ)$, cf. \eqref{E:dfnT}.  Let $\fm_\bC = \op_{2\ell\in\bZ}\,\fm_\ell$ be the $\ttT_\varphi$--graded decomposition \eqref{E:grm}.  Then $\varphi$ gives a polarized Hodge structure on $(\fm,B)$ if and only if $\ttT_\varphi$ is a grading element
(that is, $\ttT_\varphi \in \tHom(\rtL,\bZ)$ and 
\begin{equation}\label{E:kqroots}
  \fk_\bC \ = \ \fm_\mathrm{even} \ \dfn \ \op_{\ell\in\bZ} \,\fm_{2\ell}
  \quad\hbox{and}\quad
  \fq_\bC \ = \ \fm_\mathrm{odd} \ \dfn \ \op_{\ell\in\bZ} \,\fm_{2\ell+1} \, .
\end{equation}
\end{proposition}

\noindent
The set of \emph{compact roots} is $\Delta(\fk_\bC)$; the set of \emph{noncompact roots} is $\Delta(\fq_\bC)$.  Together \eqref{E:preT} and \eqref{E:kqroots} are equivalent to 
\begin{equation} \label{E:Tphi}
   \ttT_\varphi \ = \ \sum n_j \ttT_j \,, \quad\hbox{with}\quad 
   0 \le n_j \ \equiv \ \left\{ \begin{array}{ll}
  0 \ \tmod \ 2\,, & \hbox{ if } \s_j \in \Delta(\fk_\bC) \\
  1 \ \tmod \ 2\,, & \hbox{ if } \s_j \in \Delta(\fq_\bC) \,.
  \end{array} \right.
\end{equation}

Proposition \ref{P:HSgr} gives a converse to the composite Lemma \ref{L:HSgr} and Proposition \ref{P:mB}: given $\fm_\bC$ and a grading element $\ttT\in\tHom(\rtL,\bZ)$, there exists a integral form $\fm$ of $\fm_\bC$ such that the circle $\varphi$ associated to $\ttT$ by \eqref{E:dfnT} gives a Hodge structure on $(\fm,B)$.

\begin{proposition} \label{P:HSgr}
Fix a complex semisimple Lie algebra $\fm_\bC$ and grading element $\ttT\in\tHom(\rtL,\bZ)$.  Let $\fg_\bC = \op_\ell\fg_\ell$ be the $\ttT$--graded decomposition \eqref{E:grm}.  Then there exists a real form $\fm_\bR$ of $\fm_\bC$, defined over $\bZ$, a Cartan subalgebra $\ft \subset \fm_\bR$, and a Cartan decomposition $\fm_\bR = \fk_\bR \op \fq_\bR$ such that $\ft \subset \fk_\bR$ (the torus $T\subset G_\bR$ is compact), $\fk_\bC = \fm_\mathrm{even}$ and $\fq_\bC = \fm_\mathrm{odd}$.  In particular, the circle $\varphi : S^1 \to G_\tad$ determined by \eqref{E:dfnT} defines a Hodge structure $(\fm,B,\varphi)$.
\end{proposition}

\begin{proof}
Let $\{ x_\a \ | \ \a\in\Delta\} \,\cup\, \{ h_i \ | \ 1\le i \le r\}$ be a Chevalley basis of $\fm_\bC$, \cite[\S25.2]{MR499562}.  That is, $x_\a\in\fm_\a$ and $h_i\in\ft_\bC$, and 
\begin{circlist}
\item $[x_{\s_i} , x_{-\s_i}] = h_i$ and $[h_i , x_\a ] \in \bZ x_\a$, for all $1\le i\le r$ and $\a\in\Delta$;
\item if $\a,\b,\a+\b\in\Delta$ and $[x_\a,x_\b] = c_{\a,\b}\,x_{\a+\b}$, then $c_{-\a,-\b} = -c_{\a,\b} \in \bZ$.
\end{circlist}
Set $\z_\a = x_\a - x_{-\a}$ and $\xi_\a = x_\a + x_{-\a}$.  Then $\z_{-\a} = -\z_\a$ and $\x_{-\a} = \x_\a$.  If $\b \not=\pm\a$, then 
\begin{eqnarray*}
  [\z_\a , \z_\b] & = & c_{\a,\b}\,\z_{\a+\b} \,+\, c_{-\a,\b}\,\z_{\a-\b} \,,\\
  {{[}\z_\a , \x_\b{]}} & = & c_{\a,\b}\,\x_{\a+\b} \,+\, c_{\a,-\b}\,\x_{\a-\b} \,,\\
  {{[}\x_\a , \x_\b{]}} & = & c_{\a,\b}\,\z_{\a+\b} \,+\, c_{\a,-\b}\,\z_{\a-\b} \,,
\end{eqnarray*}
while $[\z_\a,\x_\a] = 2\,[x_\a,x_{-\a}]\in\tspan_\bZ\{h_i \ | \ 1\le i \le r \}$, and $[h_i,\x_\a] \in \bZ \z_\a$ and $[h_i , \z_\a] \in\bZ \x_\a$.  

Define $\fm = \fk \op \fq$ by
\begin{eqnarray*}
  \fk & = & \tspan\{ \bi h_i \ | \ 1\le i \le r \} \ \op \ 
  \tspan\{ \z_\a \,,\, \bi\xi_\a \ | \ \a(\ttT) \hbox{ even}\} \,,\\
  \fq& = & \tspan\{ \bi \z_\a \,,\, \x_\a \ | \ \a(\ttT) \hbox{ odd}\} \,.
\end{eqnarray*}
Then the algebra $\fm$ is defined over $\bZ$.  It follows that the Killing form $B : \fm \times \fm \to \bZ$ is defined over $\bZ$.  It is  straightforward to confirm that $\fm_\bR$ is a real form of $\fm_\bC$ with Cartan subalgebra $\ft_\bR = \tspan_\bR\{\bi h_i \ | \ 1 \le i \le r\}$, and Cartan decomposition $\fm_\bR = \fk_\bR \op \fq_\bR$.
\end{proof}

We close this section with a classification of Hodge representations (Theorem \ref{T:HdgRep}).  First, we review dual representations and real/complex/quaternionic representations in the two remarks below.

\begin{remark} \label{R:mu*}
Let $w_0$ be the longest element in the Weyl group $W = W(\fm_\bC,\ft_\bC)$.  If $U$ is the irreducible $\fm_\bC$--module of highest weight $\m$, then 
\begin{equation} \label{E:mu*}
  \m^* \ = \ -w_0(\m)
\end{equation}
is the highest weight of the dual $U^*$, and $-\m^*$ is the lowest weight of $U$.

The permutation $-w_0 : \Delta \to \Delta$ preserves the simple roots $\Sigma$, and so may be represented by an involution of the Dynkin diagram.  The following is Theorem 2.13 of \cite[3.2.6]{MR1349140}:  For $\fg = \fsl_n\bC$ ($n\ge3$), $\fso_{4n+2}(\bC)$ ($n\ge1$), and $\fe_6$, the canonical involution $-w_0$ of the simple roots coincides with the only nontrivial automorphism of the Dynkin diagram.  For all other simple Lie algebras $\fg$, the involution is the identity transformation.
\end{remark}

\begin{remark} \label{R:RCQ}
Let $\rho:G \to \tAut(V)$ be a representation defined over $\bQ$.  Assume $V_\bR$ is irreducible.  Then there exists an irreducible complex $G_\bC$--representation $U$, of highest weight $\m \in \Lambda$, such that one of the following holds:
\begin{circlist}
  \item $V_\bC = U$.  In this case $U \simeq U^*$ and $U$ is \emph{real}.
  \item $V_\bC = U \op U^*$ and $U \not\simeq U^*$.  We say $U$ is \emph{complex}.
  \item $V_\bC = U \op U^*$ and $U \simeq U^*$.  We say $U$ is \emph{quaternionic}.
\end{circlist}
By \eqref{E:Vdecomp}, 
\begin{subequations} \label{SE:wt}
\begin{equation} \label{E:wt1}
  \begin{array}{l}
  \hbox{\emph{if $U$ is real or quaternionic, then $m = 2\mu(\ttT_\varphi)$.}}
\end{array}
\end{equation}
Assume $U \simeq U^*$, and define 
$$
  \ttH_\mathrm{cpt} \ = \ 
  2 \sum_{\s_i\in\Delta(\fk_\bC)} \ttT_i \ = \ 2 \sum_{n_i\in2\bZ} \ttT_i \,.
$$
Then $U$ is real if and only if $\m(\ttH_\mathrm{cpt})$ is even (and $U$ is quaternionic if and only if $\m(\ttH_\mathrm{cpt})$ is odd), cf. \cite[(IV.E.4)]{MR2918237}.

By Remark \ref{R:mu*}, the highest weight of $U^*$ is $\mu^* = -\w_0(\m)$.  The representation $U$ is complex if and only if $\m \not= \m^*$.  By \eqref{E:Vdecomp}, 
\begin{equation} \label{E:wt2}
  \begin{array}{l}
  \hbox{\emph{if $U$ is complex, then 
        $m = 2 \, \tmax \{ \mu(\ttT_\varphi) , \mu^*(\ttT_\varphi)\}$.}}
\end{array}
\end{equation}
\end{subequations}
\end{remark}

\begin{theorem}[{\cite[(IV.B.6)]{MR2918237}}] \label{T:HdgRep}
Fix a semisimple $\bQ$--algebraic Lie algebra $\fg$ and a grading element $\ttT\in\tHom(\rtL,\bZ)$.\footnote{Without loss of generality, $\ttT$ is of the form \eqref{E:Tpos}.}  Fix a lattice $\Lambda$ and let $G_\bR = G_\Lambda$.  Suppose that $\ttT\in\tHom(\Lambda,\half\bZ)$, so that \eqref{E:dfnT} determines a homomorphism $\varphi : S^1 \to T = \ft_\bR / \Lambda^*$ of $\bR$--algebraic groups.  Let $\rho:G \to \tAut(V)$ be a representation defined over $\bQ$.  Assume that $V_\bR$ is irreducible, and let $U$ be the associated irreducible representation of highest weight $\m \in \Lambda$, cf. Remark \ref{R:RCQ}.  Then $(V,\rho\circ\varphi)$ gives a polarized Hodge structure (and $(G,\rho,\varphi)$ is a Hodge representation), for an appropriate choice of invariant $Q$, if and only if \eqref{E:kqroots} holds.
\end{theorem}



\section{Variations of Hodge structure}

\subsection{Tangent spaces} \label{S:ts}

Let $D = G_\bR/H_\varphi$ be a Hodge domain, and $\check D = G_\bC/P_\varphi$ the compact dual.  Set
$$
  o \ \dfn \ H_\varphi \ \in \ D \ = \ G_\bR/H_\varphi \, .
$$
As noted in Section \ref{S:GE}, $\fg_0$ is reductive; from \eqref{E:hpC} we see that $\fh_\varphi$ is reductive.  So there exists $\fh_\varphi$--module decomposition 
\begin{equation} \label{E:hphiperp}
  \fm_\bR \ = \ \fh_\varphi \,\op\, \fh_\varphi^\perp \,.
\end{equation}
Explicitly, from \eqref{E:grm}, \eqref{E:hpC} and \eqref{E:mellconj}, we see that
\begin{equation} \label{E:hpp}
  \fh^\perp_\varphi \ = \ \fm_\bR \,\cap \, \left(\fm_{\pm1} \op \cdots \op \fm_{\pm k}\right) \, .
\end{equation}

The tangent space to $D$ at $o$ is naturally identified with $\fm_\bR/\fh_\varphi$ as an $H_\varphi$--module; and by \eqref{E:hphiperp}, $\fm_\bR/\fh_\varphi \simeq \fh_\varphi^\perp$, again as $H_\varphi$--modules.  In particular, the (real) \emph{tangent bundle} is naturally identified with a homogeneous vector bundle
$$
  TD \ \dfn \ G_\bR \times_{H_\varphi} (\fm_\bR/\fh_\varphi) 
  \ \simeq \ G_\bR \times_{H_\varphi} \fh_\varphi^\perp \, .
$$  
The real form $\fm_\bR$ defines complex conjugation on $\fm_\bC = \fm_\bR \ot_\bR \bC$.  By \eqref{E:mellconj} and \eqref{E:hpp}, there exists a $\fh_\varphi$--module decomposition
$$
  \fh_\varphi^\perp \ = \ \fh^\perp_1 \op \cdots \op \fh^\perp_k \,,
  \quad\hbox{where}\quad
  \fh^\perp_\ell  \ \dfn \ (\fm_\ell \op \fm_{-\ell}) \,\cap\, \fg_\bR \, .
$$
By \eqref{E:mellconj},
\begin{equation} \label{E:hppC}
  \fh^\perp_\ell\ot_\bR \bC \ = \ \fg_{\ell} \,\op\,\fg_{-\ell} \, .
\end{equation}
Define a complex structure $\ttJ_\varphi$ on $\fh^\perp_\varphi$ by specifying 
$$
  \left.\ttJ_{\varphi}\right|_{\fh^\perp_\ell} \ \dfn \ 
  -\tfrac{\bi}{\ell}\,\tad\, \ttT_\varphi : \fh^\perp_\ell \to \fh^\perp_\ell \, .
$$
It is a consequence of $\bi \ttT_\varphi \in \ft$, that $\ttJ_\varphi$ is real.    Moreover, the fact that $\left.\tad\,\ttT_\varphi\right|_{\fm_\ell} = \ell \one$, yields both $\ttJ_\varphi^2 = -\one$, and 
$$
  \left.\ttJ_\varphi\right|_{\fm_\ell} \ = \ -\bi \one
  \quad\hbox{and}\quad
  \left.\ttJ_\varphi\right|_{\fm_{-\ell}} \ = \ \bi \one \,.
$$
Finally, since $\ttT_\varphi$ acts trivially on $\fm_0 = \fh_\varphi \ot_\bR \bC$ (indeed, $\fm_0$ is the centralizer of $\ttT_\varphi$ in $\fm_\bC$), we see that $\ttJ_\varphi$ is $H_\varphi$--invariant, and therefore defines a homogeneous complex structure, also denoted $\ttJ_\varphi$, on $TD$.  Moreover, the \emph{holomorphic tangent bundle} is naturally identified with the homogeneous vector bundle
$$
  \cT D \ \dfn \ G_\bR \times_{H_\varphi} \fm_- \, , 
$$
where 
\begin{equation}\label{E:m-}
  \fm_- \ \dfn \ \fm_{-1} \,\op\cdots\op\, \fg_{-\sfk} 
  \ = \ \bigoplus_{\a(\ttT_\varphi)>0} \fm_{-\a} \, .
\end{equation}

\subsection{The infinitesimal period relation} \label{S:IPR}

Since $\fh^\perp_1$ is an $H_\varphi$--module, it follows that $\fh^\perp_1/\fh_\varphi \subset \fm_\bR/\fh_\varphi$ is $H_\varphi$--invariant, and so defines a homogeneous subbundle.
 
\begin{definition} \label{D:IPR}
The \emph{infinitesimal period relation} (IPR) is the homogeneous subbundle 
\begin{equation} \label{E:T_1}
  T_1 \ \dfn \ G_\bR \times_{H_\varphi} (\fh^\perp_1/\fh_\varphi) \ \subset \ T D \,.
\end{equation}
A connected submanifold $Z \subset D$ is an \emph{variation of Hodge structure} (VHS), or \emph{integral manifold of the IPR}, if $T Z \subset \left.T_1{}\right|_Z$.  An irreducible variety $X \subset D$ is an \emph{algebraic VHS}, or \emph{integral variety of the IPR}, if the smooth locus $X^\circ$ is a VHS.  

Let $\sA = \sA^\sbullet$ denote the graded ring of smooth, real--valued, differential forms on $D$.  Let $\sI^\sbullet = \sI \subset \sA$ be the graded differential ideal generated by smooth sections (and their exterior derivatives) of the subbundle $\tAnn(T_1) \subset T^* D$.  Let $T_x D$ denote the tangent space at $x \in D$.  An \emph{integral element}, or \emph{infinitesimal variation of Hodge structure} (IVHS), is any subspace $E \subset T_x D$ such that $\left.\sI\right|_{E} = 0$.  The infinitesimal variations of Hodge structure $E \subset T_oD$ may be identified with subspaces $E \subset \fh^\perp_\varphi$ such that $[E,E] \subset \fh_\varphi$.  Note that, a connected submanifold $Z \subset D$ is a VHS if and only if $\left.\sI\right|_{Z} \equiv 0$; that is, each $T_z Z$ is an IVHS.  

In a mild abuse of nomenclature, we will sometimes refer to $\sI$ as the IPR.\end{definition}

Any element $g \in G_\bR$ may be viewed as a biholomorphism of $D$.  Let $g_* : T_x D \to T_{g\cdot x} D$ denote the associated push-forward map.  The following is an immediate consequence of the homogeneity of $T_1$.

\begin{lemma} \label{L:hom}
If $g \in G_\bR$, then $g^*\sI = \sI$.  In particular:
\begin{a_list}
\item
If $x = g^{-1}\cdot o$, then $E \subset T_x D$ is an IVHS if and only if $g_*E \subset T_o D$ is an IVHS.
\item
A submanifold $Z \subset D$ is a VHS if and only if $g\cdot Z$ is a VHS.
\end{a_list}
\end{lemma}

The IPR extends to the compact dual as follows.  The holomorphic tangent bundle of the compact dual is the homogeneous bundle
$$
  \cT \check D \ = \ G_\bC \times_{P_\varphi} (\fm_\bC/\fp_\varphi) \, .
$$
As noted in Section \ref{S:GE}, $\fm^{-1} = \fm_{\ge-1}$ is a $\fp_\varphi$--module.  Therefore, $\fm^{-1}/\fp_\varphi$ is an $\fp_\varphi$--module, and we may define a homogeneous subbundle
$$
  \cT_1 \ = \ G_\bC \times_{P_\varphi} (\fm^{-1}/\fp_\varphi)
  \ \subset \ \cT \check D \, .
$$
By \eqref{E:hppC}, $\fh_\varphi^\perp\ot_\bR\bC = \fg_{\not=0}$; and this yields $(\cT_1 \op \overline \cT_1)_{|D} = T_1 \ot \bC$.  

Let $\check\sA_\bC = \check\sA\ot_\bR \bC$ be the graded ring of smooth, complex--valued, differential forms on $\check D$, and let $\cI \subset \check\sA_\bC$ be the graded, differential ideal generated by smooth sections of $\tAnn(\cT_1) \subset \cT ^*\check D$.  Then 
$$
  \left(\cI \op \overline \cI\right)_{|D} \ = \ \sI \ot \bC \,.
$$
The subalgebras $\fe\subset\fm_{-1}\simeq\fm^{-1}/\fp_\varphi$ are precisely the integral elements of $\cI$ in $\cT_o\check D$.  Note that $[\fe,\fe] \subset [\fg_{-1},\fg_{-1}] \subset \fg_{-2}$ and $[\fe,\fe] \subset \fe \subset \fg_{-1}$ force $\fe$ to be abelian.  We have the following well-known lemma.  

\begin{lemma} \label{L:abel}
There is a bijective correspondence between infinitesimal variations of Hodge structure $E \subset T_o D$ and subspaces $\fe\subset\fm_{-1}$, such that $[\fe,\fe]=0$, that identifies $E\ot_\bR\bC$ with $\fe\op\overline\fe$.
\end{lemma}

\noindent In a mild abuse of terminology I will also refer to the abelian subalgebras $\fe \subset \fg_{-1}$ as IVHS.  

\subsection{Reduction to \boldmath $\ttT_\varphi = \ttT_I$ \unboldmath } \label{S:TI}

Given $\ttT_\varphi$ as in \eqref{E:Tphi}, set
$$
  I \ \dfn \ \{ i \ | \ n_i = 1 \} \quad\hbox{and}\quad
  J \ \dfn \ \{ j \ | \ n_j > 1 \}\,.
$$
Together, $\Delta^+ \subset \{ n^i\s_i \ | \ 0 \le n^i \in \bZ \}$\footnote{This is a standard result in representation theory, cf. \cite{MR1920389}.} and \eqref{E:grm} yield
\begin{equation} \label{E:preDm1}
  \Delta(\fm_1) \ = \ 
  \bigcup_{i\in I} \,\big\{ \a\in\Delta \ | \ \a(\ttT_i) = 1 \hbox{ and } \a(\ttT_j) = 0
  \ \forall \ j \in (I\backslash\{i\}) \cup J \, \big\} \,.
\end{equation}
If $I = \emptyset$, then \eqref{E:preDm1} implies $\fm_{\pm1} = 0$.  In this case, all variations of Hodge structure are trivial; they consist of a single point.  So I will 
\begin{center}
\emph{assume from this point forward that $I \not= \emptyset$.}
\end{center}

The proposition below asserts that, for the purpose of studying the IPR, there is no loss of generality in assuming $J=\emptyset$.  Equivalently, $\ttT_\varphi = \ttT_I$.

\begin{proposition} \label{P:red}
There exists a $\bQ$--algebraic semisimple subgroup $F \subset G$ with the following properties.
\begin{a_list}
\item The $\bQ$--subalgebra $\ff \subset \fg$ is a sub-Hodge structure.
\item Every integral manifold $Z \subset D$ of the infinitesimal period relation $T_1$ on $D$ that passes through $o \in D$ is contained in the orbit 
$$C \ \dfn \ F_\bR \cdot o \,.$$  
\item If $S_1 \subset T C$ is the IPR \eqref{E:T_1} on $C$, then $S_1 = \left.T_1\right|_C$; that is, the restriction of the IPR on $D$ to $C$ agrees with the IPR on $C$.
\item The distribution $S_1 \subset T C$ is bracket generating: if we inductively define $S_{a+1} = S_a + [S_1 , S_a]$, then $T C = S_b$ for some $1 \le b\in\bZ$.
\item When restricted to $\ff_\bC$, the grading element $\ttT_J$ vanishes and $\ttT_\varphi = \ttT_I$.
\end{a_list}
\end{proposition}

\begin{remark*}
Another result, similar in spirit to Proposition \ref{P:red}, is \cite[Theorem 4.6]{MR1624194}.
\end{remark*}

\begin{remark*} 
It follows from Lemma \ref{L:hom}(b) and Proposition \ref{P:red}(b) that, modulo the action of $G_\bR$, every VHS $Z \subset D$ is contained in $C$.
\end{remark*}

\begin{proof}
Let $\ff_\bC \subset \fm_\bC$ be the semisimple subalgebra generated by the simple roots $\{\s_i \ | \ i \not \in J \}$, and let $\ff$ be the $\bQ$--algebraic semisimple subalgebra $\ff_\bC \cap \fm$.  It is clear from the construction of Proposition \ref{P:HSgr} that $\ff \ot_\bQ \bC = \ff_\bC$.  Proposition \ref{P:red}(e) is immediate.

Let $\ff_\bC = \op \ff_s$ be the $\ttT_I$--graded decomposition \eqref{E:grm} of $\ff_\bC$.  Then Proposition \ref{P:red}(e) implies
\begin{equation} \label{E:fing}
  \ff_s \ \subset \ \fm_s \,,
\end{equation}
yielding Proposition \ref{P:red}(a).  I claim that $\fg_{\pm1} \subset \ff_\bC$, whence
\begin{equation} \label{E:fing1}
  \fg_{\pm1} \ = \ \ff_{\pm1} \,.
\end{equation}  
To see this, it suffices to observe that
\begin{equation} \label{E:fing0}
  \fg_{\s_i} \ \subset \ \left\{
  \begin{array}{ll}
    \fg_1 \ & \hbox{ if } i \in I \,,\\
    \fg_{>1} \ & \hbox{ if } i \in J \,,\\
    \fg_0 \ & \hbox{ if } i \not\in I \cup J \, .
  \end{array} \right.
\end{equation}
(Indeed, $\ff_\bC \cap \fm_-$ is the smallest subalgebra of $\fm_\bC$ that contains $\fm_{\pm1}$.)  It also follows from \S\ref{S:gr} and \eqref{E:fing0} that 
$$
  \fg_0 \ = \ \ff_0 \,\op\, \tspan_\bC\{\ttT_j \ | \ j \in J \} \, .
$$
Since $[\ft_\bC , \ff_s] \subset \ff_s$, we have $[\fg_0 , \ff_s] \subset \ff_s$.  Thus $\ff_\bR$ is a $H_\varphi$--module.  Therefore, we may define a homogeneous subbundle
$$
  S \ \dfn \ G_\bR \times_\Hp (\ff_\bR/\fh_\varphi)
$$
of the tangent bundle $T D$.  Since $\ff$ is a subalgebra of $\fm$, it follows that $S$ is an involutive distribution on $D$.  (In fact, $S$ is the smallest involutive distribution containing $T_1$.)  Therefore, $D$ is foliated by maximal integral manifolds of $S$.  Let $F_\bR \subset G_\bR$ be the closed Lie group with Lie algebra $\ff_\bR$.  It is immediate from the definition of $S$ that the homogeneous submanifold $C = F_\bR \cdot o$ is the maximal integral manifold through $o \in D$.  Indeed, if $L \subset F_\bR$ is the closed, connected Lie subgroup with Lie algebra $\fl = \fh_\varphi \cap \ff_\bR$, then the tangent bundle of $C$ is $TC = F_\bR \times_L (\ff_\bR/\fl) = \left. S\right|_C$.  Moreover, \eqref{E:T_1} and \eqref{E:fing1} imply
$$
  T_1 \ \subset \ S \,.
$$
So any integral of $T_1$ (equivalently, any VHS on $D$) is necessarily an integral manifold of $S$.  Proposition \ref{P:red}(b) now follows.

Set $\fl^\perp_1 = (\ff_1\op\ff_{-1}) \cap \ff_\bR$.  Then the IPR on $C$ is $S_1 = F_\bR \times_L (\fl^\perp_1/\fl)$.  By \eqref{E:fing1}, $\fl^\perp_1 = \fh^\perp_1$.  Proposition \ref{P:red}(c) follows.

Finally to establish Proposition \ref{P:red}(d), it suffices to show that $\ff_1$ generates $\ff_+$.  (Equivalently, $\ff_{-1}$ generates $\ff_-$.)  This is Theorem 3.2.1(1) (and Definition 3.1.2) of \cite{MR2532439}.
\end{proof}

\begin{equation} \label{E:TI}
  \begin{array}{l}
  \hbox{\emph{Assume for the remainder of the notes that $G = F$.} } \\
  \hbox{\emph{Equivalently, } } \textstyle
  \ttT_\varphi \ = \ \ttT_I \ = \ \sum_{i\in I} \ttT_i \,.
  \end{array}
\end{equation}
In this case, \eqref{E:preDm1} yields
\begin{equation} \label{E:Dm1}
\renewcommand{\arraystretch}{1.3}\begin{array}{rcl}
  \Delta(\fm_1) & = & \{ \a\in\Delta \ | \ \a(\ttT_\varphi) = 1 \} \\
  & = & \bigcup_{i\in I}\,\big\{ \a\in\Delta \ | \ \a(\ttT_i) = 1 \hbox{ and } 
  \a(\ttT_j) = 0 \ \forall \ j \in I\backslash\{i\} \,\big\} \,.
\end{array}
\end{equation}

\subsection{Schubert varieties}

Schubert varieties are distinguished by the property that their homology classes form an additive basis of the integer homology $H_\sbullet(\check D, \bZ)$, cf. \cite{MR0077878, MR0142697}.  This section is a brief review of Schubert varieties.  There are many excellent references; see, for example, \cite{MR0429933}.

Let $B \subset G_\bC$ denote the Borel subgroup with Lie algebra $\fb \subset \fm_\bC$.  Let $W$ denote the Weyl group of $\Delta(\fm_\bC,\ft_\bC)$.  Given $w \in W$ the $B$--orbit 
$$
  C_w \ \dfn \ B w^{-1} \cdot o \ \subset \ \check D
$$
is a \emph{Schubert cell}.  The compact dual is a disjoint union
$$
  \check D \ = \ \bigsqcup_{w\in W^\varphi} C_w
$$
over the set 
$$
  W^\varphi \ \dfn \ 
  \{ w \in W \, | \, w(\lambda) \hbox{ is $\fm_0$--dominant } \forall 
           \hbox{ $\fm_\bC$--dominant weights } \lambda \} \,.
$$
If $w \in W^\varphi$, then the dimension of $C_w$ is the length $|w|$ of $w$ (cf. Remark \ref{R:Dw}(a)).  The Zariski closure
$$
  X_w \ \dfn \ \overline{C}_w
$$
is a \emph{Schubert variety}.  

\begin{remark} \label{R:Wp}
The set $W^\varphi$ admits several characterizations; here are two.  (a)  Let 
$$
  \Delta(w) \ \dfn \ \Delta^+ \cap w \Delta^- \,.
$$
Then 
$$
  W^\varphi \ = \ \{ w \in W \, | \, \Delta^+(\fm_0) \subset w ( \Delta^+) \} 
  \ = \ \{ w \in W \, | \, \Delta(w) \subset \Delta(\fm_+) \} \, .
$$
See \cite[p. 360, Remark 5.13]{MR0142696}.\footnote{Beware: there are two Remarks 5.13 in \cite{MR0142696}, the second is on p. 361--362.}

(b)  The Weyl group of the semisimple component $\mss$ of $\fm_0$ is naturally identified with the subgroup $W_0 \subset W$ generated by the simple reflections $\{ r_j \ | \ \s_j \in \Delta(\fm_0) \}$.  Each coset $W_0\backslash W$ admits a unique representative of minimal length, and $W^\varphi$ is the set of these minimal representatives \cite[Proposition 5.13]{MR0142696}.
\end{remark}

\begin{remark} \label{R:Dw}
The set $\Delta(w)$ has several interesting properties; three of which will be useful to us.  (a)  Each simple root $\s_i \in \Sigma$ determines a reflection $r_i \in W$, and these simple reflections generate the Weyl group.  In particular, any element $w\in W$ may be expressed as $w = r_{i_1}\cdots r_{i_\ell}$.  When $\ell$ is minimal, we call this expression \emph{reduced}, and $|w| = \ell$ is the \emph{length} of $w$.  The number of elements in $\Delta(w)$ is equal to the length of $w$, \cite[Proposition 3.2.14(3)]{MR2532439}.  That is,
\begin{equation} \label{E:|Dw|}
  |\Delta(w)| = |w| \,. 
\end{equation}

(b)  Let 
$$
  \varrho \ \dfn \ \half \sum_{\a\in\Delta^+} \a \ = \ \sum_i \w_i \, .
$$
If $w \in W^\varphi$, then
\begin{equation} \label{E:rhow}
  \varrho_w \ \dfn \ \varrho \,-\, w(\varrho)
   \ = \ \sum_{\a\in\Delta(w)} \a \, .
\end{equation}
Cf. \cite[(5.10.1)]{MR0142696}.  

(c) A set $\Phi \subset \Delta$ is \emph{closed} if given any two $\b,\c \in \Phi$ such that $\b+\c \in \Delta$, it is the case that $\b+\c \in \Phi$.  The mapping $w \mapsto \Delta(w)$ is a bijection of $W^\varphi$ onto the family of all subsets $\Phi \subset \Delta(\fg_+)$ with the property that both $\Phi$ and $\Delta^+\backslash\Phi$ are closed \cite[Proposition 5.10]{MR0142696}.
\end{remark}

Remark \ref{R:Dw}(c) implies that 
\begin{equation} \label{E:nw}
  \fn_w \ = \ \bigoplus_{\a\in\Delta(w)} \fm_{-\a} \ \subset \ \fm_-
\end{equation}
is a nilpotent subalgebra of $\fm_\bC$.  Let $N_w \subset G_\bC$ be the associated unipotent Lie subgroup.  Then
\begin{equation} \nonumber 
  w C_w \ = \ N_w \cdot o \ \stackrel{(\ast)}{\simeq} \ \bC^{|w|}\,,
\end{equation}
where $(\ast)$ is due to Remark \ref{R:Dw}(a) and the fact that $N_w$ is nilpotent.  In particular, 
\begin{equation} \label{E:dimXw}
  \tdim\,X_w \ = \ \tdim\,C_w \ = \ |w| \, .
\end{equation}
Finally, note that $X_w = \overline C_w$ yields
\begin{equation} \label{E:wXw}
  w X_w \ = \ \overline{ N_w \cdot o } \, .
\end{equation}

\subsection{Schubert variations of Hodge structure}  \label{S:SVHS}

\begin{definition*}
A \emph{Schubert variation of Hodge structure}, or \emph{Schubert integral} of the IPR, is a Schubert variety $X_w$, $w \in W^\varphi$, that is an algebraic VHS.
\end{definition*}

By Remark \ref{R:Wp}(a), $\Delta(w) \subset \Delta(\fm_+)$.  From Remark \ref{R:Dw}(a) and \eqref{E:rhow} it follows that $\varrho_w(\ttT_\varphi) \ge |w|$.  Equations \eqref{E:Dm1} and \eqref{E:rhow} imply that 
\begin{equation}\label{E:rhowTp}
  \Delta(w) \ \subset \ \Delta(\fm_1)
  \quad\hbox{if and only if}\quad
  \varrho_w(\ttT_\varphi) \,=\, |w| \,.
\end{equation}
Set
\begin{equation} \label{E:Wint}
  W^\varphi_\sI \ \dfn \ \{ w \in W^\varphi \ | \ \Delta(w) \subset \Delta(\fm_1) \} 
  \ = \ 
  \{ w \in W^\varphi \ | \ \varrho_w(\ttT_\varphi) = |w| \} \, .
\end{equation}

\begin{theorem} \label{T:SchubInt}
Let $w \in W^\varphi$.  The Schubert variety $X_w$ is a variation of Hodge structure if and only if $w\in W^\varphi_\sI$.
\end{theorem}

\begin{lemma} \label{L:DwDm1}
Let $w \in W^\varphi$.  The subspace $\fn_w = T_o (wC_w)$ is an IVHS if and only if $\Delta(w) \subset \Delta(\fm_1)$.
\end{lemma}

\begin{proof}
By definition, $\fn_w \subset \fm_-$.  By Lemma \ref{L:abel}, $\fn_w$ is an IVHS if and only if $\fn_w \subset \fm_{-1}$ and $[\fn_w,\fn_w] = 0$.  The former is equivalent to $\Delta(w) \subset \Delta(\fm_1)$.  The latter then follows from the fact that $\fn_w$ is an algebra and $[\fn_w,\fn_w] \subset [\fm_{-1},\fm_{-1}] \subset \fm_{-2}$.
\end{proof}

\begin{proof}[Proof of Theorem \ref{T:SchubInt}]
By Lemma \ref{L:hom}(b), $X_w = \overline{C_w}$ is an integral variety if and only if $w X_w = \overline{N_w\cdot o}$ is an integral variety.  Since $N_w \subset G_\bC$ is an Lie subgroup with Lie algebra $\fn_w$, Lemma \ref{L:hom} implies $w X_w$ is an integral variety if and only if $\fn_w$ is an IVHS.  The proof now follows from Lemma \ref{L:DwDm1}.
\end{proof}

A variation of Hodge structure is \emph{maximal} if it not contained in any strictly larger VHS.  The \emph{Bruhat order} is the partial order on $W^\varphi$ defined by $w \le w'$ when $X_w \subset X_{w'}$.   Let
\begin{equation} \label{E:Wmaxint}
  W^\varphi_{\sI,\tmax} \ \dfn \ \{ w \in W^\varphi_\sI \ | \ w 
  \hbox{ is maximal in the Bruhat order} \} \, .
\end{equation}

\begin{proposition} \label{P:SchubMax}
Let $w \in W^\varphi_\sI$.  Then the Schubert VHS $X_w$ is a maximal variation of Hodge structure if and only if $w \in W^\varphi_{\sI,\tmax}$.  (That is, if $X_w$ is maximal among all Schubert VHS, then $X_w$ is maximal among all VHS.)
\end{proposition}

\noindent Proposition \ref{P:SchubMax} is proved in Appendix \ref{S:bo}.  Given the proposition, we will refer to the Schubert varieties $X_w$ indexed by $w \in W^\varphi_{\sI,\tmax}$ as the \emph{maximal Schubert VHS}.

\begin{proposition} \label{P:sm}
If $G_\bC = \tSL_{r+1}\bC$ or $G_\bC = \tSp_{2r}\bC$, then the maximal Schubert VHS $X_w \subset \check D = G_\bC/P_\varphi$ are homogeneously embedded Hermitian symmetric spaces of the form
$$ 
  Y_1 \ = \ \tGr(a_1,n_1) \ \times \cdots \times \ \tGr(a_s,n_s) 
  \quad\hbox{or}\quad
  Y_2 \ = \ Y_1 \ \times \ \tLG(n,2n) \, .
$$ 
\end{proposition}

\noindent Here, $\tGr(a,n)$ is the Grassmannian of $a$-planes in $\bC^n$, and $\tLG(n,2n)\subset\tGr(n,2n)$ is the Lagrangian Grassmannian.

While maximal Schubert VHS of both forms $Y_1$ and $Y_2$ may occur in the symplectic case, only $Y_1$ arises in the case that $G_\bC = \tSL_{r+1}\bC$.  Proposition \ref{P:sm} is proved in Appendix \ref{A:p}, and illustrated by the examples of Appendix \ref{A:egA} and \ref{A:egC}.  As seen from the examples in Appendix \ref{A:egBD}, Proposition \ref{P:sm} fails for $\fm_\bC = \mathfrak{spin}_m\bC$.

\subsection{Infinitesimal variations of Hodge structure}\label{S:intelem}

Let $\tGr(\ell,\cT D) \to D$ be the Grassmann bundle of holomorphic tangent subspaces of dimension $\ell$.  Let $\cV_\ell(\cI) \subset \tGr(\ell,\cT D)$ denote the subbundle of IVHS of dimension $\ell$, cf. Section \ref{S:IPR}.  By Lemma \ref{L:hom}(a), $\cV_\ell(\cI)$ is a homogeneous subbundle of $\tGr(\ell,\cT D)$.  In this section we will identify the fibre $\cV_\ell(\cI)_o$ over $o \in D$.  Roughly speaking, Theorem \ref{T:intelem} states that $\cV_\ell(\cI)_o$ is a subvariety of $\tGr(\ell,\cT D)$ that is `spanned' by the orbits of the Schubert IVHS under the isotropy representation.

Define
\begin{equation} \label{E:WpJell}
  W^\varphi_\sI(\ell)  \ \dfn \ \{ w \in W^\varphi_\sI \ | \ |w|=\ell \} \, .
\end{equation}
Let $G_0 = \{ x \in G_\bC \ | \ \tAd_x \fm_j \subset \fm_j \ \forall \, j\}$ be the closed reductive subgroup of $G_\bC$ with Lie algebra $\fm_0$.  Given $w \in W^\varphi_\sI(\ell)$, view $\fn_w \in \tGr(\ell,\fm_{-1})$ as an element of $\bP \tw^\ell\fm_{-1}$, via the Pl\"ucker embedding.  Let $\hat \fn_w \subset \tw^\ell\fm_{-1}$ denote the corresponding line and let
\begin{subequations} \label{SE:I}
\begin{equation}
  \bI_w \ \dfn \ \tspan_\bC\,\{ G_0 \cdot \hat\fn_w\} \ \subset \ \tw^\ell\fm_{-1}
\end{equation}
be the linear span of the $G_0$--orbit of $\hat\fn_w$.  Set
\begin{equation}
  \bI_\ell \ \dfn \ \bigoplus_{w\in W^\varphi_\sI(\ell)} \bI_w \,.
\end{equation}
\end{subequations}

\begin{theorem} \label{T:intelem}
\begin{a_list}
\item The line $\hat\fn_w \subset \tw^\ell\fm_{-1}$ is a $G_0$--highest weight line.  In particular, the subspace $\bI_w \subset \tw^\ell\fm_{-1}$ is an irreducible $G_0$--module of highest weight $-\varrho_w$.  
\item The $\ell$--dimensional IVHS at $o \in D$ are given by 
$$
  \cV_\ell(\cI)_o \ = \ \tGr(\ell,\fm_{-1}) \ \cap \ 
  \bP\,\bI_\ell\, .
$$
\end{a_list}
\end{theorem}

\begin{corollary} \label{C:intelem}
There exists a VHS of dimension $\ell$ if and only if there exists a Schubert VHS of dimension $\ell$.
\end{corollary}

\noindent We will see that the theorem is a straight forward consequence of Kostant's beautiful description \cite{MR0142696} of the Lie algebra homology $H_\sbullet(\fm_-)$.

\begin{proof}
Define $\d : \tw^\ell \fm_- \to \tw^{\ell-1} \fm_-$ by 
\begin{equation} \label{E:homd}
  \d(x_1\wedge\cdots\wedge x_\ell) \ \dfn \ 
  \sum_{j < k} (-1)^{j+k+1} [x_j,x_k] 
  \wedge x_1 \wedge \cdots \hat x_j \cdots \hat x_k \cdots \wedge x_\ell \,.
\end{equation}
Note that the intersection $\tGr(\ell,\fm_{-1}) \,\cap\,\tker\,\d$ is precisely the set of abelian subalgebras of $\fm_{-1}$.  By Lemma \ref{L:abel}, the fibre admits the identification $\cV_\ell(\cI)_o \simeq \{ \fe\subset \fm_{-1} \ | \ [\fe,\fe]=0\}$.  Thus,
\begin{equation}\label{E:V(I)}
  \cV_\ell(\cI)_o \ = \ \tGr(\ell,\fm_{-1}) \ \cap \ 
  \bP \,\tker \{ \d : \tw^\ell\fm_- \ \to \ \tw^{\ell-1}\fm_- \} \,.
\end{equation} 
(Here, we view $\tw^\ell\fm_{-1}$ as a subspace of $\tw^\ell\fm_-$, and identify $\tGr(\ell,\fm_{-1})$ with the subset $\tGr(\ell,\fm_-) \,\cap\, \bP \tw^\ell\fm_{-1}$.)  

The map $\d$ satisfies $\d^2=0$, and so defines Lie algebra homology groups
\begin{equation} \label{E:homD}
  H_\ell(\fm_-) \ \dfn \ 
  \frac{\tker\{ \d : \tw^\ell \fm_- \ \to \ \tw^{\ell-1} \fm_-\}}
  {\tim\{ \d : \tw^{\ell+1} \fm_- \ \to \ \tw^{\ell} \fm_-\}} \,.
\end{equation}
Moreover, $\d$ is a $\fm_0$--module map; therefore, $H_\ell(\fm_-)$ admits the structure of a $\fm_0$--module.  Let 
$$
  W^\varphi(\ell) \ \dfn \ \{ w \in W^\varphi \ | \ |w|=\ell \} \, .
$$
Kostant \cite[Corollary 8.1]{MR0142696} proved that the $\fm_0$--module decomposition is
\begin{equation} \label{E:homK}
  H_\ell(\fm_-) \ = \ \bigoplus_{w\in W^\varphi(\ell)} \bI_w \,,
\end{equation}
where $\bI_w$ is an irreducible $\fm_0$--module of highest weight $-\varrho_w$, and highest weight line $[\hat\fn_w]\subset H_\ell(\fm_-)$.

Note that, since $\bI_w$ is irreducible and $\ttT_\varphi$ lies in the center of $\fm_0$, the grading element $\ttT_\varphi$ acts on $\bI_w$ by the scalar $-\varrho_w(\ttT_\varphi)$.  It follows from \eqref{E:Wint} that the $-\ell$ eigenspace of $\ttT_\varphi$ in $H_\ell(\fm_-)$ is $\bI_\ell$.  On the other hand, since the eigenvalues of $\ttT_\varphi$ on $\fm_-$ are the negative integers $-1,\ldots,-\sfk$, it follows that the eigenvalues of $\ttT_\varphi$ on $\tw^\ell\fm_-$ lie in $\{ -\ell,\ldots,-\ell\sfk \}$.  Moreover, the $-\ell$ eigenspace in $\tw^\ell\fm_{-}$ is $\tw^\ell\fm_{-1}$.  Since the $-\ell$ eigenspace in $\tw^{\ell+1}\fm_-$ is trivial, it follows from \eqref{E:homD} and \eqref{E:homK} that 
$$
  \bI_\ell \ = \ 
  \tker\left\{ \d : {\tw^\ell\fm_{-1}} \ \to \ 
  \fg_{-2} \ot \tw^{\ell-1}\fg_{-1} \right\}
  \ \subset \ \tw^\ell\fm_{-1} \, .
$$
The theorem now follows from \eqref{E:V(I)}.
\end{proof}

\section{Characteristic cohomology of the infinitesimal period relation}\label{S:CC}

In Section \ref{S:coh} the three pertinent cohomologies (de Rham, characteristic and Lie algebra) are reviewed.  The main result, Theorem \ref{T:icc}, which describes the invariant characteristic cohomology in terms of dual Schubert classes, is stated and proven in Section \ref{S:Picc}.

\subsection{Cohomologies} \label{S:coh}

\subsubsection{Characteristic cohomology}\label{S:icc}

Let $\sA = \sA^\sbullet$ denote the smooth, real--valued, forms on $D$, and let $(\sA,\td)$ be the de Rham complex on $D$.  Recall (Definition \ref{D:IPR}) the graded differential ideal $\sI \subset \sA$ generated by sections of $\tAnn(T_1) \subset T^*D$.  The property $\td \sI \subset \sI$ implies that the exterior derivative descends to the quotient $\sA/\sI$; let $(\sA/\sI,\td)$ denote the associated quotient complex.  The corresponding cohomology $H^\sbullet_\sI(D)$ is the \emph{characteristic cohomology} of the IPR. Note that, given a variation of Hodge structure $Z \subset D$, we have a map $H^\sbullet_\sI(D) \to H^\sbullet(Z)$; the characteristic cohomology may be viewed as the cohomology inducing ordinary cohomology on VHS virtue of their being integrals of the IPR \cite[(III.A)]{MR2666359}.  

The exterior derivative also preserves the subspace $\sA^\MR \subset \sA$ of $\MR$--invariant forms, yielding a subcomplex $(\sA^\MR,\td)$; let $H^\sbullet(D)^\MR$ denote the corresponding cohomology.  Moreover, the homogeneity of $T_1$ (Lemma \ref{L:hom}) implies that the differential ideal $\sI$ is preserved by the (pull-back) action of $\MR$.  Therefore, the subspace $(\sA/\sI)^\MR \subset \sA/\sI$ of $\MR$--invariant elements is well-defined, and preserved by $\td$; the associated cohomology 
\begin{equation} \label{E:icc0}
  H^\ell_\sI(D)^\MR \ = \ 
  \frac{\tker \{ \td\,:\; (\sA/\sI)^\ell \ \to \ (\sA/\sI)^{\ell+1} \}^\MR}
       {\tim \{ \td\,:\; (\sA/\sI)^{\ell-1} \ \to \ (\sA/\sI)^\ell \}^\MR}
\end{equation}
is the \emph{invariant characteristic cohomology}.  As observed in \cite[(III.B)]{MR2666359},  given a global variation of Hodge structure $\Phi : S \to \Gamma\backslash D$, there is an induced map $\Phi^* : H^\sbullet_\sI(D)^\MR \to H^\sbullet(S)$.  The image may be viewed as the topological invariants of global variations of Hodge structure that may be defined universally (that is, independently of the monodromy groups $\Gamma$). 

The projection $\sA^\MR \to (\sA/\sI)^\MR$ commutes with exterior differentiation, and so induces a map $(\sA^\MR,\td) \to ( (\sA/\sI)^\MR,\td)$ of complexes.  Let 
\begin{equation}\label{E:piI}
  p_\sI \,:\, H^\sbullet(D)^\MR \ \to \ H^\sbullet_\sI(D)^\MR
\end{equation}
denote the induced ring homomorphism of cohomologies.

\subsubsection{Lie algebra cohomology} \label{S:lac}

Let $\fa$ be a Lie algebra.  Given $\theta \in \tw^\ell\fa^*$ and $x_0,\ldots,x_\ell \in \fa$, define a linear map 
$$
  \d \,:\, \tw^\ell\fa^* \ \to \ \tw^{\ell+1}\fa^*
$$
by 
\begin{equation}\label{E:d}
  (\d\theta)(x_0,\ldots,x_\ell) \ \dfn \ \sum_{i<j} (-1)^{i+j}
  \theta\left( [ x_i,x_j] , x_0 , \ldots, \hat x_i , \ldots ,\hat x_j , \ldots,
  x_\ell \right) \, .
\end{equation}
It is straightforward to check that $\d^2=0$.  In particular, $(\tw^\sbullet\fa^*,\d)$ is a complex.  The associated \emph{Lie algebra cohomology groups} are defined by  
$$
  H^\ell(\fa) \ \dfn \ 
  \frac{ \tker\{ \d \,:\, \tw^\ell\fa^* \ \to \ \tw^{\ell+1} \fa^* \} }
  {\tim\{ \d \,:\, \tw^{\ell-1}\fa^* \ \to \ \tw^\ell \fa^* \} } \,.
$$

\subsubsection{Relative Lie algebra cohomology}\label{S:rlac}

Given a subalgebra $\fb \subset \fa$, consider $\tAnn(\fb) \subset \fa^*$.  Equation \eqref{E:d} defines a map $\d : \tw\tAnn(\fb) \to \tw\tAnn(\fb)$ also satisfying $\d^2=0$.  The adjoint action of $\fb$ on $\fa$ naturally induces an $\fb$--module structure on $\tAnn(\fb)$.  Moreover, $\d$ preserves the subspace $\tw\tAnn(\fb)^\fb \subset (\tw\tAnn(\fb))$ of $\fb$--invariants, cf. \cite[\S22]{MR0024908}.   The subcomplex $((\tw\tAnn(\fb))^\fb , \d)$ defines the \emph{relative Lie algebra cohomology} 
$$
  H^\ell(\fa,\fb ) \ = \ 
  \frac{ \tker\{ \d \,:\, 
         \tw^\ell\tAnn(\fb) \ \to \ \tw^{\ell+1}\tAnn(\fb) \}^\fb }
  {\tim\{ \d \,:\, \tw^{\ell-1}\tAnn(\fb) \ \to \ \tw^\ell\tAnn(\fb) \}^\fb } \,.
$$

In the case $\fa = \fg_\bR$ and $\fb = \fh_\varphi$, standard arguments yield 
\begin{equation} \label{E:idco0'}
  H^\sbullet(\mR,\hp) \ \simeq \ H^\sbullet(D)^\MR \,, 
\end{equation}
cf. \cite{MR0024908, MR928600}.  Moreover, I claim that complexifying yields
\begin{equation} \label{E:idco0}
  H^\sbullet(D,\bC)^\MR \ \simeq \ H^\sbullet(\fm_\bC,\fm_0) 
  \ \simeq \ H^\sbullet(\check D,\bC) \,.
\end{equation}
To see why this is the case, follow the notation in the proof of Proposition \ref{P:HSgr} to define 
$$
  \check \fg_\bR \ \dfn \ \tspan_\bR\{ \bi h_i \ | \ 1 \le i \le r \} \ \op \ 
  \tspan_\bR\{ \z_\a \,,\, \bi\x_\a \ | \ \a \in \Delta^+ \} \, ,
$$
a compact real form of $\fg_\bC$.  Let $\check G_\bR \subset G_\bC$ denote the corresponding compact Lie subgroup.  Note that $\fh_\varphi = \check \fg_\bR \cap \fp_\varphi$.  Thus,
$$
  \check D  \ = \ \check G_\bR / H_\varphi \, ,
$$
and the analog of \eqref{E:idco0'} for $\check D$ is $H^\sbullet(\check\fg_\bR,\hp) = H^\sbullet(\check D)^{\check G_\bR}$.  The two relative Lie algebra cohomologies $H^\sbullet(\fg,\hp)$ and $H^\sbullet(\check\fg,\hp)$ are both real forms of $H^\sbullet(\fg_\bC,\fg_0)$.  Thus, $H^\sbullet(D,\bC)^\MR\simeq  H^\sbullet(\fm_\bC,\fm_0) \simeq H^\sbullet(\check D,\bC)^{\check G_\bR}$.  So, to establish \eqref{E:idco0}, it remains to observe that $H^\sbullet(\check D) = H^\sbullet(\check D)^{\check G_\bR}$; this is well-known, cf. \cite[Theorem 2.3]{MR0024908}.

\subsubsection{Kostant's theorem}\label{S:kostant}

As noted in Section \ref{S:gr}, the subalgebra $\fm_-$ is a $\fm_0$--module.  It is straightforward to confirm that $\d \,:\, \tw^\ell(\fm_-)^* \ \to \ \tw^{\ell+1} (\fm_-)^*$ is a $\fm_0$--module map.  As a consequence, $H^\ell(\fm_-)$ naturally admits the structure of an $\fm_0$--module.  Kostant \cite{MR0142697} showed that $H^\sbullet(\fm_\bC,\fm_0) = (H^\bullet(\fm_-) \ot H^\bullet(\fm_-)^*)^{\fm_0}$.  Taken with \eqref{E:idco0}, this yields
\begin{equation} \label{E:idco1}
  H^\sbullet(D,\bC)^{G_\bR} \ = \ 
  (H^\bullet(\fm_-) \ot H^\bullet(\fm_-)^*)^{\fm_0} \,.
\end{equation}

Kostant described the $\fm_0$--module structure of $H^\sbullet(\fm_-)$ as follows.  Let $H_{\varrho_w}$ denote the irreducible $\fm_0$--module of \emph{lowest} weight $\varrho_w$, cf. \eqref{E:rhow}.  By \cite[Theorem 5.14]{MR0142696}, 
\begin{equation} \label{E:Hell}
  H^\ell(\fm_-) \ = \ 
  \bigoplus_{w\in W^\varphi(\ell)} H_{\varrho_w} \,,
\end{equation} 
is a $\fg_0$--module decomposition.  Moreover, the modules are inequivalent; that is, given $w,w'\in W^\varphi$, we have $H_{\varrho w}\simeq H_{\varrho w'}$ if and only if $w = w'$.  It is then a consequence of Schur's lemma that
\begin{equation} \label{E:idco1'}
  (H^\bullet(\fm_-) \ot H^\bullet(\fm_-)^*)^{\fm_0} \ = \ 
  \bigoplus_{w\in W^\varphi} \left( H_{\varrho_w} \,\ot\, H_{\varrho_w}^*
  \right)^{\fm_0} \,,
\end{equation}
where $H_{\varrho_w}^*$ is the $\fm_0$--module dual to $H_{\varrho_w}$.  Let
\begin{equation} \label{E:Hw}
  \mathscr{H}_w \ \dfn \ 
  \left( H_{\varrho_w} \,\ot\, H_{\varrho_w}^*
  \right)^{\fm_0} \, .
\end{equation}
Then \eqref{E:idco1}, \eqref{E:idco1'} and \eqref{E:Hw} yield
\begin{equation} \label{E:idco2}
  H^\sbullet(D,\bC)^{G_\bR} \ = \ 
   \bigoplus_{w\in W^\varphi} \mathscr{H}_w \, .
\end{equation}

\subsubsection{Schubert classes}\label{S:schubcl}

The Schubert classes $\{ [X_w] \ | \ w \in W^\varphi \}$ form a free basis of the integer homology $H_\bullet(\check D , \bZ)$, cf. \cite{MR0077878}.  The integral homology groups have no torsion \cite{MR0077878, MR0062750}; thus, the Schubert classes also form a basis of $H_\bullet(\check D , \bC)$.  Let $\{ \bx_w \ | \ w \in W^\varphi \}$ denote the dual basis of $H^\sbullet(\check D,\bC)$; here our degree convention is $\bx_w \in H^{2|w|}(\check D,\bC)$.  Making use of \eqref{E:idco0}, we also regard these classes as a basis of $H^\sbullet(D,\bC)^\MR$. 

By Schur's lemma (see also \cite[Theorem 5.2]{MR0142697}), $\tdim\,\mathscr{H}_w = 1$.  Under the identification \eqref{E:idco2}, Kostant \cite{MR0142697} showed that 
\begin{equation} \label{E:Hwxw}
  \mathscr{H}_w \ = \ \tspan_\bC\{ \bx_w \} \, .
\end{equation}
Thus, \eqref{E:idco2} and \eqref{E:Hwxw} yield
\begin{equation} \label{E:idco3}
  H^\sbullet( D , \bC )^{G_\bR}  \ = \ 
  \tspan_\bC\{ \bx_w \ |  w \in W^\varphi \} \, .
\end{equation}
Since the classes $\bx_w$ are represented by real forms, we have $H^\sbullet(D)^{G_\bR} = \tspan_\bR\{\bx_w \ | \ w \in W^\varphi \}$.

\subsection{Invariant characteristic cohomology of the IPR} \label{S:Picc}

\begin{theorem} \label{T:icc}
The ring homomorphism $p_\sI : H^\sbullet(D)^{G_\bR} \to H^\sbullet_\sI(D)^{G_\bR}$ of \eqref{E:piI} is surjective with kernel 
$$
  \tker\,p_\sI \ = \ \bigoplus_{w\in W^\varphi\backslash W^\varphi_\sI} 
  \mathscr{H}_w \,.
$$
In particular, the map $p_\sI$ is given by
$$
  \bc = \sum_{w\in W^\varphi} c_w \bx_w \quad \mapsto \quad 
  \bc_\sI \equiv \sum_{w\in W^\varphi_\sI} c_w \bx_w \,.
$$
Thus, the invariant characteristic cohomology satisfies $H^\sbullet_\sI(D)^{G_\bR} \equiv \tspan_\bR\{ \bx_w \ | \ w \in W^\varphi_\sI \}$.
\end{theorem}

\noindent Above, we use $\equiv$ (in place of $=$) to emphasize that $\bc_\sI \in H^\sbullet(D)^{G_\bR}/\tker\,p_\sI$.  The corollary below, first proved in \cite[(III.B.1)]{MR2666359}, follows immediately.

\begin{corollary*}
The invariant characteristic cohomology vanishes in odd degree.  In even degree $2\ell$, the invariant characteristic cohomology is of Hodge type $(\ell,\ell)$.
\end{corollary*}

The remainder of this section is given to the proof of the theorem.  Use \eqref{E:hphiperp} to identify $(\hpp)^*$ with $\tAnn(\hp) \subset \fg_\bR^*$.  Then \eqref{E:T_1} yields $\mathfrak{i} \dfn \tAnn(\fh^\perp_1) \subset (\hpp)^* \simeq \tAnn(\hp)$, and let $\mathfrak{I} \subset \tw\tAnn(\hp)$ be the graded ideal generated by $\mathfrak{i} \op \d(\mathfrak{i}) \subset \tAnn(\hp) \op \tw^2 \tAnn(\hp)$.  Then $\d$ preserves $\mathfrak{I}$, and this induces a complex $( \tAnn(\hp)/\mathfrak{I} \,,\, \d_\mathfrak{I} )$.  Moreover, since $\fh^\perp_1$ is an $\fh_\varphi$--module, it follows that $\mathfrak{I}$ is also an $\fh_\varphi$--module.  Thus, $( \tAnn(\hp)/\mathfrak{I} \,,\, \d_\mathfrak{I} )$ is a complex of $\fh_\varphi$--modules.  Therefore, the space $(\tw\tAnn(\hp)/\mathfrak{I})^{\fh_\varphi}$ of $\fh_\varphi$--invariants defines a sub-complex $( (\tw\tAnn(\hp)/\mathfrak{I})^{\fh_\varphi} \,,\, \d_\mathfrak{I} )$.  As shown in \cite[(III.B.1)]{MR2666359}, standard arguments (analogous to those establishing \eqref{E:idco0'}) yield 
\begin{equation} \label{E:icc1}
  H^\ell_\sI(D)^{G_\bR} \ = \ 
  \frac{ \tker\{ \d_\mathfrak{I} \,:\, 
                 \tw^\ell\tAnn(\hp)/\mathfrak{I}^\ell 
         \ \to \ \tw^{\ell+1} \tAnn(\hp)/\mathfrak{I}^{\ell+1} 
         \}^{\fh_\varphi} }
  {\tim\{ \d_\mathfrak{I} \,:\, 
                 \tw^{\ell-1}\tAnn(\hp)/\mathfrak{I}^{\ell-1} 
         \ \to \ \tw^\ell \tAnn(\hp)/\mathfrak{I}^\ell 
         \}^{\fh_\varphi} } \,.
\end{equation}

We now proceed to analyze the complexification $H^\sbullet_\sI(D,\bC)^{G_\bR}$.  We have noted \eqref{E:hpp} that $\fh_\varphi^\perp \ot_\bR \bC = \fm_{\not=0}$.  Set
\begin{equation} \label{E:j}
  \mathfrak{j} \ \dfn \ \mathfrak{i} \ot_\bR \bC \ = \ 
  \tAnn( \fm_{\pm 1} ) \ \subset \ (\fm_{\not=0})^* \, . 
\end{equation}
As an $\fm_0$--module, $\mathfrak{j} \simeq ( \fm_{\pm2} \op \cdots \op \fm_{\pm k})^*$.  Let $\d(\mathfrak{j})$ denote the image of $\mathfrak{j}$ under the map $\d : (\fm_{\not=0})^* \to \tw^2 (\fm_{\not=0})^*$ defined by \eqref{E:d}.  Let $\mathfrak{J} \subset \tw (\fm_{\not=0})^*$ denote the graded ideal generated by $\mathfrak{j}\op\d(\mathfrak{j}) \subset (\fm_{\not=0})^* \op \tw^2 (\fm_{\not=0})^*$.  Then $\mathfrak{J} = \mathfrak{I}\ot_\bR \bC$, and 
\begin{equation} \label{E:c1}
  ( \tw\tAnn(\hp)/\mathfrak{I} )\ot_\bR \bC
  \ = \ \tw(\fm_{\not=0})^*/\mathfrak{J} \,.
\end{equation}
By \eqref{E:j}, we have 
\begin{equation} \label{E:c2}
  \tw(\fm_{\not=0})^*/\mathfrak{J} \ = \ 
  \tw(\fm_{\pm1})^*/\mathfrak{J} \, .
\end{equation}
Note that the map $\d : (\fm_{\pm2})^* \to \tw^2 (\fm_{\pm1})^*$, defined by \eqref{E:d}, takes values in $\tw^2 (\fm_{-1})^* \,\op\, \tw^2 (\fm_1)^*$.  Moreover, under the identification $\tw^2 (\fm_1)^* = \tw^2 (\overline\fm_{-1})^*$ given by \eqref{E:mellconj}, the map may be expressed as $\d = \d_2\op\bar\d_2$, with $\d_2 : (\fm_{-2})^* \to \tw^2 (\fm_{-1})^*$.  Let $\mathfrak{J}_2 \subset \tw(\fm_{-1})^*$ and $\bar{\mathfrak{J}}_2 = \tw(\overline\fm_{-1})^*$ be the graded ideals generated by $\tim\,\d_2$ and $ \tim\,\bar\d_2$, respectively.  By \eqref{E:c1} and \eqref{E:c2}, we have 
\begin{equation} \label{E:icc2}
  ( \tw\tAnn(\hp)/\mathfrak{I} )\ot_\bR \bC
  \ = \ \left( \tw(\fm_{-1})^* / \mathfrak{J}_2 \right) 
  \,\ot \, 
  \left( \tw(\fm_1)^* / \bar{\mathfrak{J}}_2 \right) \,.
\end{equation}
Utilizing \eqref{E:mellconj}, it follows that $\tw(\fm_1)^* / \bar{\mathfrak{J}}_2 = ( \tw(\fm_{-1})^* / \mathfrak{J}_2)^*$, so that \eqref{E:icc2} yields 
\begin{equation} \label{E:icc3}
  \Big( \tw\tAnn(\hp)/\mathfrak{I} 
  \Big)^{\fh_\varphi} \ot_\bR \bC \ = \ 
  \Big( \left( \tw(\fm_{-1})^* / \mathfrak{J}_2 \right) 
  \,\ot \, 
  \left( \tw(\fm_{-1})^* / \mathfrak{J}_2 \right)^*
  \Big)^{\fm_0} \, .
\end{equation}
By definition of $\fg_{-1}$, the grading element $\ttT_\varphi$ acts on $\fm_{-1}$ with eigenvalue $-1$, and therefore on $(\fm_{-1})^*$ with eigenvalue $1$.  Therefore, $\ttT_\varphi$ acts on $\tw^\ell(\fm_{-1})^* / \mathfrak{J}^\ell_2$ by the eigenvalue $\ell$, and on its dual by $-\ell$.  Since $\ttT_\varphi \in\fm_0$, this forces 
$$
 \Big( \left( \tw(\fm_{-1})^* / \mathfrak{J}_2 \right) 
  \,\ot \, 
  \left( \tw(\fm_{-1})^* / \mathfrak{J}_2 \right)^*
  \Big)^{\fm_0} \ = \ 
  \bigoplus_\ell 
  \Big( \left( \tw^\ell(\fm_{-1})^* / \mathfrak{J}^\ell_2 \right) 
  \,\ot \, 
  \left( \tw^\ell(\fm_{-1})^* / \mathfrak{J}^\ell_2 \right)^*
  \Big)^{\fm_0} \,.
$$
This, taken with \eqref{E:icc3}, implies $(\tw^{2\ell+1}\tAnn(\hp)/\mathfrak{I}^{2\ell+1})^{\fh_\varphi}=0$.  Whence, \eqref{E:icc1} yields $H^{2\ell+1}_\sI(D)^{G_\bR} = 0$; the ICC vanishes in odd degree.  Moreover,
\begin{equation} \label{E:icc4}
  H^{2\ell}_\sI(D,\bC)^{G_\bR} \ = \  
  \Big( \left( \tw^\ell(\fm_{-1})^* / \mathfrak{J}^\ell_2 \right) 
  \,\ot \, 
  \left( \tw^\ell(\fm_{-1})^* / \mathfrak{J}^\ell_2 \right)^*
  \Big)^{\fm_0}\,.
\end{equation}

I claim that $\tw^\ell(\fm_{-1})^* / \mathfrak{J}^\ell_2$ is a subspace of $H^\ell(\fm_-)$.  To see this, recall that $H^\ell(\fm_-)$ is an $\fm_0$--module.  Therefore, $H^\ell(\fm_-)$ decomposes into a direct sum of eigenspaces of $\ttT_\varphi \in\fm_0$.  Let $H^\ell_m \subset H^\ell(\fm_-)$ denote the eigenspace for the  eigenvalue $m$.  Then
\begin{equation} \label{E:evH}
  H^\ell(\fm_-) \ = \ \bigoplus_{m\ge \ell} H^\ell_m \,,
\end{equation}
and
\begin{equation} \nonumber 
  H^\ell_\ell \ = \ 
  \frac{ \tker\{ \d \,:\, \tw^\ell(\fm_{-1})^* \ \to \ 0 \} }
  {\tim\{ \d \,:\, \fm_{-2}\ot\tw^{\ell-2}(\fm_{-1})^* 
                 \ \to \ \tw^\ell (\fm_{-1})^* \} } \,.
\end{equation}
The right-hand side is precisely $\tw^\ell(\fm_{-1})^* / \mathfrak{J}^\ell_2$, by the definition of $\mathfrak{J}$.  So, by \eqref{E:icc4}, 
\begin{equation} \label{E:icc5}
  H^{2\ell}_\sI(D,\bC)^{G_\bR} \ = \ 
  \big( H^\ell_\ell \,\ot \, ( H^\ell_\ell )^* \big)^{\fm_0} \, .
\end{equation}
Since $H_{\varrho_w}$ is an irreducible $\fm_0$--module, the grading element $\ttT_\varphi$ acts on $H_{\varrho_w}$ by the scalar $\varrho_w(\ttT_\varphi)$.\footnote{As an element of the center $\fz\subset\fg_0$, $\ttT_\varphi$ necessarily acts on an irreducible $\fg_0$--module by a scalar.}  In particular, \eqref{E:Wint}, \eqref{E:WpJell} and \eqref{E:Hell} yield
\begin{equation} \label{E:icc6}
  H^\ell_\ell \ = \ \bigoplus_{w \in W^\varphi_\sI(\ell)} H_{\varrho_w} \,.
\end{equation}
Together, \eqref{E:Hw}, \eqref{E:icc5} and \eqref{E:icc6} yield 
$$
  H^{2\ell}_\sI(D,\bC)^{G_\bR} \ = \
  \bigoplus_{w \in W^\varphi_\sI(\ell)} \mathscr{H}_{w} \,.
$$
The theorem now follows from \eqref{E:Hwxw}. \hfill\qed

\appendix
\section{Examples} \label{A:egs}

\subsection{Special linear Hodge groups \boldmath $G = A_r$ \unboldmath} \label{A:egA}

Throughout we fix a complete flag $\bC^\sbullet = \{ 0 \subset \bC^1 \subset \bC^2 \subset \cdots \subset \bC^{r+1} \}$.

\begin{example*} 
Consider $\fm_\bC = \fsl_6\bC$ and $\ttT_\varphi = \ttT_2+\ttT_4$ (equivalently, $I = \{2,4\}$).  Then $\check D$ may be identified with the partial flag variety 
$$
  \tFlag_{(2,4)}\bC^6 \ = \ \{ (E^2,E^4) \in \tGr(2,6) 
  \times \tGr(4,6) \ | \ E^2 \subset E^4 \} \,.
$$
The underlying real form $\fm_\bR$ of Proposition \ref{P:HSgr} is $\fsu(2,4)$.  To see this, note that the noncompact simple roots are $\Sigma(\fm_\mathrm{odd}) = \{ \s_2 \,,\, \s_4 \}$.  The Weyl group element $w = (2312)$ maps 
$$
  w\s_1 \, = \, \s_3 \,,\ \ 
  w\s_2 \, = \, -(\s_1+\s_2+\s_3) \,,\ \
  w\s_3 \, = \, \s_1 \,,\ \
  w\s_4 \, = \, \s_2+\s_3+\s_4 \,,\ \
  w\s_5 \, = \, \s_5 \, .
$$
That is, $w\Sigma$ is a simple system with a single noncompact root, $w\s_2$.  It now follows from the Vogan diagram classification \cite[VI]{MR1920389} that $\fm_\bR = \fsu(2,4)$.  The basis $\{\ttT^w_j\}$ dual to $\w\Sigma$ is $\ttT^w_1 = \ttT_3-\ttT_2$, $\ttT^w_2 = \ttT_4-\ttT_2$, $\ttT^w_3 = \ttT_1-\ttT_2+\ttT_4$, $\ttT^w_4 = \ttT_4$ and $\ttT^w_5 = \ttT_5$; and 
$$
  \ttT_\varphi \ = \ \ttT_2+\ttT_4 \ = \ 2\,\ttT^w_4 - \ttT^w_2 \,.
$$

We have 
$$
  \mss \ \simeq \fsl_2\bC \,\times\,\fsl_2\bC \,\times\,\fsl_2\bC \, ,
$$
and $\Delta(\fm_1) = \{ \a\in\Delta \ | \ \a(\ttT_2+\ttT_4)=1 \}$.  The maximal Schubert integrals are indexed by $W^\varphi_{\sI,\tmax} = \{ (2312) \,,\, (4521) \,,\, (4534)\}$, cf. Section \ref{S:SVHS}; the corresponding root sets are 
\begin{eqnarray*}
  \Delta(4534) & = & \{ \s_j + \cdots + \s_k \ | \ 3 \le j \le 4 \le k \} 
  \ = \ \{ \a\in\Delta(\fm_1) \ | \ \a(\ttT_2) = 0 \} \,,\\
  \Delta(2312) & = & \{ \s_j + \cdots + \s_k \ | \ j \le 2 \le k \le 3 \} 
  \ = \ \{ \a\in\Delta(\fm_1) \ | \ \a(\ttT_4)=0 \} \,,\\
  \Delta(4521) & = & \{ \s_2 \,,\, \s_1+\s_2 \}\,\cup\,\{ \s_4 \,,\, \s_4+\s_5 \} 
  \ = \ \{ \a\in\Delta(\fm_1) \ | \ \a(\ttT_3)=0 \} \,.
\end{eqnarray*}
The three maximal Schubert VHS, each of dimension four, are 
\begin{eqnarray*}
  X_{(4534)} & = & \{ E^\sbullet \ | \ E^2 = \bC^2 \} 
  \ = \ \tGr( 2 , \bC^6/\bC^2) \ \simeq \ \tGr(2,\bC^4) \,,\\
  X_{(2312)} & = & \{ E^\sbullet \ | \ E^4 = \bC^4 \} 
  \ = \ \tGr( 2 , 4)\,,\\
  X_{(4521)} & = & \{ E^\sbullet \ | \ E^2 \subset \bC^3 \subset E^4 \} 
  \ = \ \bP(\bC^3)^* \times \bP(\bC^6/\bC^3)  \ \simeq \ \bP^2 \times \bP^2 \, .
\end{eqnarray*}
\end{example*}

\begin{example*} 
Consider $\fm_\bC = \fsl_9\bC$ and $\ttT_\varphi = \ttT_2+\ttT_4+\ttT_7$ (equivalently, $I = \{2,4,7\}$).  If $w = (456345)$, then $w\Sigma$ is a simple system with a single noncompact root, $w \s_5$.  By the Vogan diagram classification \cite[VI]{MR1920389} of real, semisimple Lie algebras, the underlying real form $\fm_\bR$ of Proposition \ref{P:HSgr} is $\fm_\bR = \fsu(5,4)$.

The compact dual $\check D$ may be identified with the partial flag variety 
$$
  \tFlag_{(2,4,7)}\bC^9 \ = \ \{ (E^2,E^4,E^7) \in \tGr(2,9) 
  \times \tGr(4,9) \times \tGr(7,9) \ | \ E^2 \subset E^4\subset E^7 \} \,.
$$
The maximal Schubert VHS are 
$$
\renewcommand{\arraystretch}{1.3} \begin{array}{c|c|c}
 w & X_w & \Delta(w) \\ \hline
 (45621) & \bP^2 \times \bP^3 & \a(\ttT_3+\ttT_7)=0 \\
 (784521) & \bP^2 \times \bP^2 \times \bP^2 & \a(\ttT_3 + \ttT_6)=0 \\
 (456345) & \tGr(2,\bC^5) & \a(\ttT_2+\ttT_7) = 0 \\
 (784534) & \tGr(2,\bC^4) \times \bP^2 & \a(\ttT_2+\ttT_6)=0 \\
 (786743) & \bP^2 \times \tGr(2,\bC^4) & \a(\ttT_2+\ttT_5)=0 \\
 (7867421)& \bP^2 \times \bP^1 \times \tGr(2,\bC^4) & \a(\ttT_3+\ttT_5)=0 \\
 (7867562312) & \tGr(2,\bC^4) \times \tGr(3,\bC^5) & \a(\ttT_4)=0
\end{array}
$$
The second column describes the maximal integral $X_w$ as a homogeneously embedded Hermitian symmetric space; the third column gives the additional condition necessary to distinguish $\Delta(w)$ as a subset of $\Delta(\fm_1) = \{ \a\in\Delta \ | \ \a(\ttT_2+\ttT_4+\ttT_7)=1 \}$.  
\end{example*}

\subsection{Symplectic Hodge groups \boldmath $G = C_r$ \unboldmath} \label{A:egC}

Fix a nondegenerate, skew-symmetric bilinear form $\bvsigma$ on $\bC^{2r}$.  Throughout, 
$$
  \tLG(d,2r)  \ = \ \{ E \in \tGr(d,{2r}) \ | \ 
  \bvsigma_{|E} = 0\}
$$ 
will denote the Lagrangian grassmannian of $\varsigma$--isotropic $d$--planes $E \subset \bC^{2r}$.

\begin{example} \label{eg:C5P25}
This is a counter-example to \cite[Theorem 1.1.1(a)]{MR1624194}.  

Let $\fm_\bC = \fsp_{10}\bC$ and $\ttT_\varphi = \ttT_2+\ttT_5$.  The underlying real form of Proposition \ref{P:HSgr} is $\fm_\bR = \fsp(5,\bR)$.  To see this, let $w = (543545) \in W$.  Then 
$$
  w\s_1 = \s_1 \,,\ w \s_2 = \s_2+\cdots+\s_5 \,,\ 
  w\s_3 = \s_4 \,,\ w\s_4 = \s_3 \,,\ w\s_5 = -2(\s_3+\s_4)-\s_5 \, .
$$
In particular, $w\Sigma$ is a simple system with $w\s_5(\ttT_\varphi) = -1$ odd, and all other $w\s_j(\ttT_\varphi)$ even; that is, $w\s_5$ is the unique noncompact root of $w \Sigma$.  The claim now follows from the Vogan diagram classification of real Lie algebras \cite[VI]{MR1920389}.

Let $U_{\w_1} = \bC^{10}$ be the irreducible $\fm_\bC$--module of highest weight $\w_1$.  Then $U_{\w_1}$ is self-dual.  Moreover, $\ttH_\mathrm{cpt} = 2(\ttT_1+\ttT_3+\ttT_4)$ and $\w_1 = \s_1+\cdots+\s_4+\half\s_5$.  So $\w_1(\ttH_\mathrm{cpt}) = 6$ is even.  It follows from Remark \ref{R:RCQ} that $U_{\w_1}$ is real.  Set $V_\bC = U_{\w_1}$.  The Hodge decomposition \eqref{E:Vdecomp} is $V_\bC = V_{3/2} \op V_{1/2} \op V_{-1/2} \op V_{-3/2}$.  As $\mss = \fsl_2\bC \op \fsl_3\bC$--modules,
$$
  V_{3/2} \ \simeq \ \bC^2 \ot \bC \,,\quad
  V_{1/2} \ \simeq \ \bC \ot \tw^2\bC^3
$$
and $V_{-p} \simeq V_p^*$.  In particular, the Hodge numbers are $\bh = (2,3,3,2)$.  The Hodge domain $D = \tSp(5,\bR)/\Hp$ is the period domain for these Hodge numbers.  Note that, this example satisfies the hypotheses of \cite[Theorem 1.1.1(a)]{MR1624194}.

The maximal Schubert varieties are $$
\renewcommand{\arraystretch}{1.3} \begin{array}{c|c|c}
 w & X_w & \Delta(w) \\ \hline
 (52312) & \tGr(2,\bC^4) \times \bP^1 & \a(\ttT_4)=0 \\
 (54521) & \bP^2 \times \tLG(2,\bC^4) & \a(\ttT_3) = 0 \\
 (545345) & \tLG(3,\bC^6) & \a(\ttT_2)=0 \\
 (234123) & \tGr(2,5) & \a(\ttT_5)=0 \\
\end{array}
$$
The second column describes the maximal integral $X_w$ as a homogeneously embedded Hermitian symmetric space; the third column gives the additional condition necessary to distinguish $\Delta(w)$ as a subset of $\Delta(\fm_1) = \{ \a\in\Delta \ | \ \a(\ttT_2+\ttT_5)=1 \}$.  From \eqref{E:dimXw} (or a direct dimension count), we see that the maximal Schubert VHS are of dimensions five and six; the two of dimension six are indexed by $w_1 = (545345)$ and $w_2 = (234123)$.  Corollary \ref{C:intelem} assures us that every IVHS is of dimension at most six.  By Lemma \ref{L:hom}(b), both $w_1 X_{w_1}$ and $w_2 X_{w_2}$ are integral varieties.  Moreover, from \eqref{E:wXw}, we see that $o$ is a smooth point of both $w_1 X_{w_1}$ and $w_2 X_{w_2}$.  Therefore, there exist two distinct integrals of maximal dimension through $o \in D$, contradicting \cite[Theorem 1.1.1(a)]{MR1624194}.  (Indeed, this example is also a counter-example to \cite[Theorems 5.1.1 \& 5.3.1(a)]{MR1624194}.)

As noted in Theorem \ref{T:intelem}, the line $\hat\fn_w \subset \tw^{|w|}\fm_{-1}$ is a $G_0$--highest weight line of weight $-\varrho_w$.  For the two Schubert varieties above, $-\varrho_{w_1} = -4\w_2 + 4\w_5$ and $-\varrho_{w_2} = 5\w_2-2\w_5$.  In particular, both $\bI_{w_1}$ and $\bI_{w_2}$ are trivial $\mss$--modules; equivalently, $\bI_{w_1} = \hat\fn_{w_1}$ and $\bI_{w_2} = \hat\fn_{w_2}$.  Therefore, $\bI_6 = \hat\fn_{w_1} \op \hat\fn_{w_2}$.
\end{example}

\begin{example*} 
If $\fm_\bC = \fsp_{10}\bC$ and $\ttT_\varphi = \ttT_2$, then $\check D \simeq \tLG(2,\bC^{10})$.  The unique maximal element of $W^\varphi_\sI$ is $(234123)$, for which $\Delta(234123) = \{ \a\in\Delta^+ \ | \ \a(\ttT_2)=1 \,,\ \a(\ttT_5) = 0 \}$.  The unique maximal Schubert VHS is $X_{(234123)} = \tGr(2,\bC^5)$.
\end{example*}

\subsection{Spin Hodge groups \boldmath $G = B_r$ and $D_r$ \unboldmath} \label{A:egBD}

Fix a nondegenerate, symmetric bilinear form $\bq$ on $\bC^m$, for $m=2r,2r+1$.  Throughout, 
$$
  \tGr_\bq(d,m)  \ = \ \{ E \in \tGr(d,m) \ | \ 
  \bq_{|E} = 0\}
$$ 
will denote the $\bq$--isotropic $d$--planes $E \subset \bC^{m}$.  In particular, $\tGr_\bq(1,m)$ is the smooth quadric hypersurface
$$
  \cQ^{m-2} \ \subset \ \bP^{m-1} \,.
$$
We also fix a complete isotropic flag $\bC^\sbullet = \{ 0 \subset \bC^1 \subset \cdots \subset \bC^m\}$; that is, $$\bq(\bC^a , \bC^{m-a}) \ = \ 0\,.$$

\begin{example*} 
Consider $\fm_\bC = \fso_{11}\bC = \fb_5$ and $\ttT_\varphi = \ttT_3$ (equivalently, $I = \{3\}$).  Then $\check D\simeq\tGr_\bq(3,{11})$.  By the Vogan diagram classification \cite[VI]{MR1920389} of real, semisimple Lie algebras, the underlying real form $\fm_\bR$ of Proposition \ref{P:HSgr} is $\fso(6,5)$.  We have $\Delta(\fm_1) = \{ \a\in\Delta \ | \ \a(\ttT_3)=1 \}$.  The maximal Schubert VHS are given by:
\begin{a_list_nem}
\item  $w = (34543)$ with
  $\Delta(w) =  \{ \a \in \Delta(\fm_1) \ | \ \a(\ttT_2)=0  \}$ and 
  $$X_w \ = \ \{ E \in \tGr_\bq(3,{11}) \ | \ 
    \bC^2 \subset E \subset \bC^9 \}
    \ \simeq \ \cQ^5 \subset \bP^6 \,.$$
\item $w = (345421)$ with
  $\Delta(w) = \{ \a\in\Delta(\fm_1) \ | \ \a(\ttT_2+\ttT_4) \le 1 \}$ and 
  \begin{eqnarray*}
    X_w & = & \{ E \in \tGr_\bq(3,{11}) \ | \ \tdim\,(E\cap\bC^3) \ge 2 
    \,,\ E \subset \bC^7 \} \,,\\
    \tSing\,X_w & = & \{\bC^3\} \,, \hbox{(the point $o\in\check D$).}
  \end{eqnarray*}
\item $w = (3452312)$ with 
  $\Delta(w) = \{ \a\in\Delta(\fm_1) \ | \ \a(\ttT_2+\ttT_5) \le1 \}$ and 
  \begin{eqnarray*}
    X_w & = & \{ E \in \tGr_\bq(3,{11}) \ | \ \tdim\,(E\cap\bC^4) \ge 2
    \,,\ E \subset\bC^6\} \,,\\
    \tSing\,X_w & = & \{ E \in \tGr_\bq(3,\bC^{11}) \ | \ E \subset\bC^4\} 
    \ \simeq \ \bP^3 \,.
  \end{eqnarray*}
\end{a_list_nem}
\end{example*}

\begin{example} \label{eg:B6P35}
The following is a counter-example to \cite[Theorem 1.1.1(b)]{MR1624194}.  

Let $\fm_\bC = \fso_{13}=\fb_6$ and $\ttT_\varphi = \ttT_3+\ttT_5$.   The compact dual $\check D$ is the partial flag $\tFlag_{3,5}^\bq(\bC^{13})$ of $\bq$--isotropic planes.  The underlying real form of Proposition \ref{P:HSgr} is $\fm_\bR = \fso(4,9)$.  To see this, let $w = (342312) \in W$.  Then
$$\renewcommand{\arraystretch}{1.3}\begin{array}{lll}
  w\s_1 = \s_4\,,\ & w\s_2 = -(\s_1 + \cdots + \s_4)\,,\ &
  w\s_3 = \s_1\,,\\
  w\s_4 = \s_2\,,\ & w\s_5 = \s_3 + \s_4 + \s_5\,,\ &
  w\s_6 = \s_6 \, .
\end{array}$$
Note that $w\s_2(\ttT_\varphi) = -1$, and all other $w\s_j(\ttT_\varphi)$ are even.  That is, $w\s_2$ is the unique noncompact root of the simple system $w\Sigma$.  It follows from the Vogan diagram classification \cite[VI]{MR1920389} that $\fm_\bR = \fso(4,9)$.

Let $U_{\w_1} = \bC^{13}$ be the irreducible $\fm_\bC$--module of highest weight $\w_1 = \s_1+\cdots+\s_6$.  Then $U_{\w_1}$ is self-dual.  Additionally, $\ttH_\mathrm{cpt} = 2(\ttT_1+\ttT_2+\ttT_4+\ttT_6)$ and $\w_1$.  So $\w_1(\ttH_\mathrm{cpt}) = 8$ is even, and Remark \ref{R:RCQ} implies that $U_{\w_1}$ is real.  Set $V_\bC = U_{\w_1}$.  The decomposition of \eqref{E:Vdecomp} is $V_2\op V_1 \op V_0 \op V_{-1} \op V_{-2}$.  As $\mss = \fsl_3\bC\op\fsl_2\bC\op\fsl_2\bC$--modules,
$$
  V_2 \ \simeq \ (\tw^2\bC^3)\ot\bC\ot\bC \,,\quad
  V_1 \ \simeq \ \bC\ot\bC^2\ot\bC \,,\quad
  V_0 \ \simeq \ \bC\ot\bC\ot\tSym^2\bC \,,
$$
and $V_{-p} \simeq V_p^*$.  In particular, the Hodge numbers are $\bh = (3,2,3,2,3)$.  The Hodge domain $D = \tSO(4,9)/H_\varphi$ is the period domain for these Hodge numbers.  

The maximal Schubert varieties are 
\begin{a_list_nem}
\item $w = (564)$ with 
  $\Delta(w) = \{ \a\in\Delta(\fm_1) \ | \ \a(\ttT_3)=0\,,\ 
    \a(\ttT_4+\ttT_6)\le 1 \}$ and 
  $$X_w \ = \ \{ (E^3\subset E^5) \in \check D \ | \ 
    E^3 = \bC^3 \,,\ \tdim(E^5\cap\bC^5) \ge 4 \,,\ E^5 \subset \bC^7 \}\,.$$
  Note that $\tSing\,X_w = \{ \bC^3 \subset \bC^5 \} = o \in \check D$.
\item $w = (565321)$ with
  $\Delta(w) = \{ \a\in\Delta(\fm_1) \ | \ \a(\ttT_4)=0 \}$, and 
  $X_w = \bP^3\times\cQ^3$ is a homogeneously embedded Hermitian symmetric space.
\item $w = (342312)$ with 
  $\Delta(w) = \{\a\in\Delta(\fm_1) \ | \ \a(\ttT_5)=0 \}$, and 
  $X_w = \tGr(3,5)$ is a homogeneously embedded Hermitian symmetric space.
\end{a_list_nem}
From \eqref{E:dimXw} (or a direct dimension count), we see that the maximal Schubert VHS are of dimensions three and six; the two of dimension six are indexed by $w_1 = (565321)$ and $w_2 = (342312)$.  Corollary \ref{C:intelem} assures us that every IVHS is of dimension at most six.  From \cite[\S4.3]{MR1624194}, we have $q_1^\mathit{even} = h^{3,1}\,h^{4,0} = 2\cdot 3 = 6$; thus, the hypotheses of \cite[Theorem 1.1.1(b)]{MR1624194} are satisfied.  By Lemma \ref{L:hom}(b), both $w_1 X_{w_1}$ and $w_2 X_{w_2}$ are integral varieties.  Moreover, from \eqref{E:wXw}, we see that $o$ is a smooth point of both $w_1 X_{w_1}$ and $w_2 X_{w_2}$.  Therefore, there exist two distinct integrals of maximal dimension through $o \in D$, contradicting \cite[Theorem 1.1.1(b)]{MR1624194}.  (Indeed, this example is also a counter-example to \cite[Theorems 5.1.2(a) \& 5.3.1(b)]{MR1624194}.)

As noted in Theorem \ref{T:intelem}, the line $\hat\fn_w \subset \tw^{|w|}\fm_{-1}$ is a $G_0$--highest weight line of weight $-\varrho_w$.  For the two Schubert varieties above, $-\varrho_{w_1} = -4\w_3+6\w_4-3\w_5$ and $-\varrho_{w_2} = -5\w_3+3\w_5$.  In particular, as $\mss$--modules, $\bI_{w_1}\simeq\bC\ot(\tSym^6\bC^2)\ot\bC$ and $\bI_{w_2}$ is trivial.  In particular, the $G_0$--orbit of $\fn_{w_1} \in \bP\bI_{w_1}$ is naturally identified with the Veronese embedding $v_6(\bP^1)$, while $\bP\bI_{w_2} = \{\fn_{w_2}\}$.  The orbit $v_6(\bP^1)$ indexes the set $\{ g w_1 X_{w_1} \ | \ g \in G_0 \}$; each variety $g w_1 X_{w_1} = \overline{(g N_{w_1} g^{-1})\cdot o}$ is a distinct algebraic VHS, of maximal dimension, containing $o \in D$ as a smooth point.
\end{example}

\begin{example*} 
Consider $\fm_\bC = \fso_{10} = \fd_5$ and $\ttT_\varphi = \ttT_2$.  The compact dual $\check D = D_5/P_2$ may be identified with the $\bq$--isotropic 2--planes $\tGr_\bq(2,\bC^{10}) = \{ E \in \tGr(2,\bC^{10}) \ | \ \bq_{|E}=0\}$.  We have $\Delta(\fm_1) = \{ \a\in\Delta \ | \ \a(\ttT_2)=1 \}$.  The maximal Schubert VHS are given by:
\begin{a_list_nem}
\item $w = (235432)$ with 
  $\Delta(w) = \{ \a \in\Delta(\fm_1) \ | \ \a(\ttT_1)=0 \}$ and 
  $$X_w \ = \ \{ E \in \tGr_\bq(2,{10}) \ | \ \bC^1 \subset E  \subset \bC^8 \} 
  \ \simeq \ \cQ^6 \,.$$
\item $w = (235123)$ with 
  $\Delta(w) = \{ \a \in\Delta(\fm_1) \ | \ \a(\ttT_5)=0 \}$ and 
  $$X_w \ = \ \{ E \in \tGr_\bq(2,{10}) \ | \ E \subset \bC^5 \} 
    \ \simeq \ \tGr(2,\bC^5) \,.$$ 
\item $w = (234123)$ with 
  $\Delta(w) = \{ \a \in\Delta(\fm_1) \ | \ \a(\ttT_4)=0 \}$ and 
  $$X_w \ = \ \{ E \in \tGr_\bq(2,{10}) \ | \ E  \subset \tilde\bC^5 \} 
  \ \simeq \ \tGr(2,\bC^5) \,.$$
  Here, if $\{e_1,\ldots,e_{10}\}$ is a basis of $\bC^{10}$ adapted to 
  the filtration $\bC^\sbullet$, then 
  $\tilde\bC^5 = \tspan_\bC\{ e_1,\ldots,e_4,e_6\}$.
\item $w = (235431)$ with 
  $\Delta(w) = \{ \a \in\Delta(\fm_1) \ | \ \a(\ttT_1+\ttT_3)\le 1 \}$ and 
  \begin{eqnarray*}
    X_w & = & \{ E \in \tGr_\bq(2,{10}) \ | \ \tdim\,(E\cap\bC^2) \ge 1 \,,\ 
    E \subset \bC^{7} \} \,,\\
    \tSing\, X_w & = & \{ \bC^2 \} \ = \ o \ \in \ \check D \, .
  \end{eqnarray*}
\item $w = (235412)$ with 
  $\Delta(w) = \{ \a \in\Delta(\fm_1) \ | \ \a(\ttT_1+\ttT_4+\ttT_5) \le 1 \}$ and 
  \begin{eqnarray*}
    X_w & = & \{ E \in \tGr_\bq(2,{10}) \ | \ \tdim\,(E\cap\bC^3) \ge 1 \,, 
    \ E \subset \bC^6 \} \,,\\
    \tSing\,X_w & = & \{ E \subset \bC^3 \} \ \simeq \ \bP^2 \, .
  \end{eqnarray*}
\end{a_list_nem}
\end{example*}





\subsection{\boldmath The exceptional Hodge group $G = E_6$ \unboldmath} \label{A:egE}

\begin{example*}
In the case $\fm_\bC = \fe_6$ and $\ttT_\varphi = \ttT_2$, we have $\tdim_\bC D = 21$.  There are six maximal Schubert VHS, and they are each of dimension ten.  The maximal Schubert VHS are given by 
$$
\renewcommand{\arraystretch}{1.3} \begin{array}{r|c|c}
 \multicolumn{1}{c|}{w} & \Delta(w) & X_w \\ \hline
 (2456345243) & \a(\ttT_1) = 0 & \cS_5 \\
 (2456345134) & \a(\ttT_4) \le 1 &  \\
 (2456345241) & \a(\ttT_1+\ttT_5)\le1 &  \\
 (2456345132) & \a(\ttT_1+\ttT_4+\ttT_6)\le2 &  \\
 (2456341324) & \a(\ttT_3+\ttT_6)\le 1 &  \\
 (2453413245) & \a(\ttT_6) = 0 & \cS_5 
\end{array}
$$
The second column gives the additional condition necessary to define $\Delta(w)$ as a subset of $\Delta(\fm_1) = \{\a\in\Delta \ | \ \a(\ttT_2)=1\}$; the third column describes the integrals that are homogeneously embedded Hermitian symmetric spaces.  Above, $\cS_5$ denotes the \emph{Spinor variety} 
$$
  \cS_5 \ \dfn \ D_5/P_5 \ = \  \tSpin({10},\bC)/P_5 \,.
$$
\end{example*}

\begin{example*}
In the case $\fm_\bC = \fe_6$ and $\ttT_\varphi = \ttT_3$, we have $\tdim_\bC D = 25$. The maximal Schubert VHS are given by 
$$
\renewcommand{\arraystretch}{1.3} \begin{array}{r|c|c}
 \multicolumn{1}{c|}{w} & \Delta(w) & X_w \\ \hline
 (341324) & \a(\ttT_5)=0 & \tGr(2,\bC^5) \\
 (3456132) & \a(\ttT_1+\ttT_2+\ttT_5)\le1 & \\
 (34562451) & \a(\ttT_1+\ttT_4)\le1 & \\
 (34561345) & \a(\ttT_2)=0 & \tGr(2,\bC^6) \\
 (3456245342) & \a(\ttT_1)=0 & \cS_5 = D_5/P_5
\end{array}
$$
The second column gives the additional condition necessary to define $\Delta(w)$ as a subset of $\Delta(\fm_1) = \{\a\in\Delta \ | \ \a(\ttT_3)=1\}$; the third column describes the integrals that are homogeneously embedded Hermitian symmetric spaces.
\end{example*}

\begin{example*}
In the case $\fm_\bC = \fe_6$ and $\ttT_\varphi = \ttT_4$, we have $\tdim_\bC D = 29$. The maximal Schubert VHS are given by 
$$
\renewcommand{\arraystretch}{1.3} \begin{array}{r|c|c}
 \multicolumn{1}{c|}{w} & \Delta(w) & X_w \\ \hline
 (432413) & \a(\ttT_5)=0 & \tGr(2,\bC^5) \\
 (456321) & \a(\ttT_2+\ttT_3+\ttT_5)\le1 &  \\
 (456245) & \a(\ttT_3)=0 & \tGr(2,\bC^5) \\
 (456345134) & \a(\ttT_2)=0 & \tGr(3,\bC^6) 
\end{array}
$$
The second column gives the additional condition necessary to define $\Delta(w)$ as a subset of $\Delta(\fm_1) = \{\a\in\Delta \ | \ \a(\ttT_4)=1\}$; the third column describes the integrals that are homogeneously embedded Hermitian symmetric spaces.
\end{example*}

\begin{example*}
In the case $\fm_\bC = \fe_6$ and $\ttT_\varphi = \ttT_5$, we have $\tdim_\bC D = 25$. The maximal Schubert VHS are given by 
$$
\renewcommand{\arraystretch}{1.3} \begin{array}{r|c|c}
 \multicolumn{1}{c|}{w} & \Delta(w) & X_w \\ \hline
  (564524) & \a(\ttT_3)=0 & \tGr(2,\bC^5) \\
  (5645321) & \a(\ttT_2+\ttT_3+\ttT_6)\le1 &  \\
  (56432413) & \a(\ttT_4+\ttT_6)\le1 &  \\
  (56453413) & \a(\ttT_2)=0 & \tGr(2,\bC^6) \\
  (5432451342) & \a(\ttT_6)=0 &  \cS_5
\end{array}
$$
The second column gives the additional condition necessary to define $\Delta(w)$ as a subset of $\Delta(\fm_1) = \{\a\in\Delta \ | \ \a(\ttT_5)=1\}$; the third column describes the integrals that are homogeneously embedded Hermitian symmetric spaces.
\end{example*}

\begin{example*}
Consider the exceptional $\fm_\bC = \fe_6$ and $\ttT_\varphi = \ttT_2 + \ttT_5$, we have $\tdim_\bC D = 25$.  The underlying real form $\fm_\bR$ of Proposition \ref{P:HSgr} satisfies $\fk_\bR = \fso(10) \op \bR$.  To see this, note that the noncompact simple roots are $\Sigma(\fm_\mathrm{odd}) = \{ \s_2 \,,\, \s_5 \}$.  The Weyl group element $w = (2431)$ maps 
\begin{eqnarray*}
  w\s_1 & = &  -(\s_1+\cdots+\s_4) \,,\quad 
  w\s_2 \, = \, \s_4 \,,\quad
  w\s_3 \, = \, \s_1 \,,\\
  w\s_4 & = & \s_3 \,, \quad
  w\s_5 \, = \, \s_2+\s_4+\s_5 \,,\quad
  w\s_6 \, = \, \s_6 \, .
\end{eqnarray*}
That is, $w\Sigma$ is a simple system with a single noncompact root, $w\s_1$.  The claim now follows from the Vogan diagram classification \cite[VI]{MR1920389}.  The basis $\{\ttT^w_j\}$ dual to $\w\Sigma$ is $\ttT^w_1 = \ttT_5-\ttT_2$, $\ttT^w_2 = \ttT_4-\ttT_2$, $\ttT^w_3 = \ttT_1-\ttT_2+\ttT_5$, $\ttT^w_4 = \ttT_3-\ttT_2+\ttT_5$, $\ttT^w_5 = \ttT_5$ and $\ttT^w_6 = \ttT_6$; and 
$$
  \ttT_\varphi \ = \ \ttT_2+\ttT_5 \ = \ 2\,\ttT^w_5 - \ttT^w_1 \,.
$$
The maximal Schubert VHS are given by 
$$
\renewcommand{\arraystretch}{1.3} \begin{array}{r|c|c}
 \multicolumn{1}{c|}{w} & \Delta(w) & X_w \\ \hline
 (562) & \a(\ttT_2)=0 & \bP^1\times\bP^2 \\
 (2431) & \a(\ttT_5)=0 &  \bP^4 \\
 (56453413) & \a(\ttT_4)=0 &  \tGr(2,6)
\end{array}
$$
The second column gives the additional condition necessary to define $\Delta(w)$ as a subset of $\Delta(\fm_1) = \{\a\in\Delta \ | \ \a(\ttT_2+\ttT_5)=1\}$; the third column describes the integrals that are homogeneously embedded Hermitian symmetric spaces.
\end{example*}

\subsection{\boldmath The exceptional Hodge group $G = F_4$ \unboldmath} \label{A:egF}

In the case that $\fg$ is the exceptional Lie algebra $\ff_4$, there are 15 grading elements $\ttT_\varphi$ of the form \eqref{E:TI}. Each is considered in the examples below; these examples, along with Proposition \ref{P:red}, yield

\begin{corollary*} 
Variations of Hodge structure in $\check D = \mathrm{F}_4(\bC)/P_\varphi$ are of dimension at most 7.
\end{corollary*}

For the exceptional $\ff_4$, then there are two possibilities for the underlying real form $\fm_\bR = \fk_\bR \op \fq_\bR$ of Proposition \ref{P:HSgr}, cf. \cite[p.416]{MR1920389}.  The first, denoted $\FI$, has maximal compact subalgebra $\fk_\bR = \fsp(3)\op\fsu(2)$.  The simple roots may be selected so that the unique noncompact simple root is $\s_1$.  The second, denoted $\FII$, has maximal compact subalgebra $\fk_\bR = \fso(9)$.  The simple roots may be selected so that the unique noncompact simple root is $\s_4$.\footnote{The ordering on the simple roots is as in \cite{Bourbaki}, $\s_1$ is a long root and $\s_4$ is a short root.}

\begin{example*}
Consider the case $\ttT_\varphi = \ttT_1$.  We have $\tdim_\bC D = 15$.  The underlying real form $\fm_\bR$ of Proposition \ref{P:HSgr} is $\FI$.  The two maximal Schubert VHS are given by 
\begin{eqnarray*}
  w = (1234232) & \hbox{and} & 
  \Delta(w) = \{ \a\in\Delta(\fm_1) \ | \ \a(\ttT_2)\le1 \} \,;\\
  w = (1234231) & \hbox{and} & 
  \Delta(w) = \{ \a\in\Delta(\fm_1) \ | \ \a(\ttT_2+\ttT_4)\le2 \} \,.
\end{eqnarray*}
\end{example*}

\begin{example*}
Consider the case $\ttT_\varphi = \ttT_2$.  We have $\tdim_\bC D = 20$.  If $w = (2342321)$, then $w\s_1$ is the single noncompact root of $w\Sigma$.  Therefore, the underlying real form $\fm_\bR$ of Proposition \ref{P:HSgr} is $\FI$.  The two maximal Schubert VHS are given by 
\begin{eqnarray*}
  w = (2341) & \hbox{and} & 
  \Delta(w) = \{ \a\in\Delta(\fm_1) \ | \ \a(\ttT_1+\ttT_3)\le1 \} \,;\\
  w = (234232) & \hbox{and} & 
  \Delta(w) = \{ \a\in\Delta(\fm_1) \ | \ \a(\ttT_1)=0 \} \,.
\end{eqnarray*}
The maximal $X_{(234232)}$ is a homogeneously embedded $\tLG(3,\bC^6)$.
\end{example*}

\begin{example*}
Consider the case $\ttT_\varphi = \ttT_3$.  We have $\tdim_\bC D = 20$.  If $w = (34)$, then $w\s_4$ is the single noncompact root of $w\Sigma$.  Therefore, the underlying real form $\fm_\bR$ of Proposition \ref{P:HSgr} is $\FII$.  The unique maximal Schubert VHS is a homogeneously embedded $\bP^2$ given by $w = (34)$ and $\Delta(w) = \{ \a\in\Delta(\fm_1) \ | \ \a(\ttT_2)=0\}$.
\end{example*}

\begin{example*}
Consider the case $\ttT_\varphi = \ttT_4$.  We have $\tdim_\bC D = 15$.  The underlying real form $\fm_\bR$ of Proposition \ref{P:HSgr} is $\FII$.  The unique maximal Schubert VHS is a homogeneously embedded $\bP^2$ given by $w = (43)$ and $\Delta(w) = \{ \a\in\Delta(\fm_1) \ | \ \a(\ttT_2)=0\}$.
\end{example*}

\begin{example*}
Let $\ttT_\varphi = \ttT_1+\ttT_2$.  We have $\tdim_\bC D = 21$.  To see that the underlying real form $\fm_\bR$ of Proposition \ref{P:HSgr} is $\FI$, let $w = (1) \in W$.  Then $w \s_1 = -\s_1$, $w\s_2 = \s_1+\s_2$ and $w \s_j = \s_j$, $j=3,4$.   So $w\s_1$ is the unique noncompact root of $w \Sigma$.  The Vogan diagram classification \cite[p.416]{MR1920389} yields $\fm_\bR = \FI$.  The two maximal Schubert VHS are homogeneously embedded Hermitian symmetric spaces given by 
$$ \renewcommand{\arraystretch}{1.3}
\begin{array}{lll}
  w = (1) \,,\ &  
  \Delta(w) = \{ \a\in\Delta(\fm_1) \ | \ \a(\ttT_2)=0 \} \,,\ 
  & X_w = \bP^1 \,;\\
  w = (234232) \,,\ & 
  \Delta(w) = \{ \a\in\Delta(\fm_1) \ | \ \a(\ttT_1)=0 \} \,, 
  & X_w = \tLG(3,\bC^6) \,.
\end{array}$$
\end{example*}

\begin{example*}
Let $\ttT_\varphi = \ttT_1+\ttT_3$.  We have $\tdim_\bC D = 22$.  In this case the underlying real form $\fm_\bR$ of Proposition \ref{P:HSgr} is $\FI$.  To see this, let $w = (3412321) \in W$.  Then $w \s_1 = -\s_1-2\s_2-4\s_3-2\s_4$, $w\s_2 = \s_2+2\s_3$, $w \s_3 = \s_4$ and $w \s_4 = \s_1+\s_2+\s_3$.   So $w\s_1$ is the unique noncompact root of $w \Sigma$.  The Vogan diagram classification \cite[p.416]{MR1920389} yields $\fm_\bR = \FI$.  The two maximal Schubert VHS are homogeneously embedded Hermitian symmetric spaces given by 
$$ \renewcommand{\arraystretch}{1.3}
\begin{array}{lll}
  w = (12) \,,\ &  
  \Delta(w) = \{ \a\in\Delta(\fm_1) \ | \ \a(\ttT_3)=0 \} \,,\ 
  & X_w = \bP^2 \,;\\
  w = (341) \,,\ & 
  \Delta(w) = \{ \a\in\Delta(\fm_1) \ | \ \a(\ttT_2)=0 \} \,, 
  & X_w = \bP^1\times\bP^2 \,.
\end{array}$$
\end{example*}

\begin{example*}
Let $\ttT_\varphi = \ttT_1+\ttT_4$.  We have $\tdim_\bC D = 20$.  To see that the underlying real form $\fm_\bR$ of Proposition \ref{P:HSgr} is $\FI$, let $w = (12321) \in W$.  We leave it to the reader to confirm that $w\s_1$ is the unique noncompact root of $w\Sigma$; the claim then follows from the Vogan diagram classification \cite[p.416]{MR1920389}.  The three maximal Schubert VHS are homogeneously embedded Hermitian symmetric spaces given by 
$$ \renewcommand{\arraystretch}{1.3}
\begin{array}{lll}
  w = (412) \,,\ &  
  \Delta(w) = \{ \a\in\Delta(\fm_1) \ | \ \a(\ttT_3)=0 \} \,,\ 
  & X_w = \bP^2\times\bP^1 \,;\\
  w = (431) \,,\ &  
  \Delta(w) = \{ \a\in\Delta(\fm_1) \ | \ \a(\ttT_2)=0 \} \,,\ 
  & X_w = \bP^1\times\bP^2 \,;\\
  w = (12321) \,,\ & 
  \Delta(w) = \{ \a\in\Delta(\fm_1) \ | \ \a(\ttT_4)=0 \} \,, 
  & X_w = \cQ^5 \ \subset \ \bP^6 \,.
\end{array}$$
\end{example*}

\begin{example*}
Let $\ttT_\varphi = \ttT_2+\ttT_3$.  We have $\tdim_\bC D = 22$.  To see that the underlying real form $\fm_\bR$ of Proposition \ref{P:HSgr} is $\FI$, let $w = (21) \in W$.  We leave it to the reader to confirm that $w\s_1$ is the unique noncompact root of $w\Sigma$; the claim then follows from the Vogan diagram classification \cite[p.416]{MR1920389}.  The two maximal Schubert VHS are homogeneously embedded Hermitian symmetric spaces given by 
$$ \renewcommand{\arraystretch}{1.3}
\begin{array}{lll}
  w = (21) \,,\ &  
  \Delta(w) = \{ \a\in\Delta(\fm_1) \ | \ \a(\ttT_3)=0 \} \,,\ 
  & X_w = \bP^2 \,;\\
  w = (34) \,,\ & 
  \Delta(w) = \{ \a\in\Delta(\fm_1) \ | \ \a(\ttT_2)=0 \} \,, 
  & X_w = \bP^2 \,.
\end{array}$$
\end{example*}

\begin{example*}
Let $\ttT_\varphi = \ttT_2+\ttT_4$.  We have $\tdim_\bC D = 22$.  To see that the underlying real form $\fm_\bR$ of Proposition \ref{P:HSgr} is $\FI$, let $w = (4321) \in W$.  We leave it to the reader to confirm that $w\s_1$ is the unique noncompact root of $w\Sigma$; the claim then follows from the Vogan diagram classification \cite[p.416]{MR1920389}.  The four maximal Schubert VHS, three of them homogeneously embedded Hermitian symmetric spaces, are given by 
$$ \renewcommand{\arraystretch}{1.3}
\begin{array}{lll}
  w = (43) \,,\ &  
  \Delta(w) = \{ \a\in\Delta(\fm_1) \ | \ \a(\ttT_2)=0 \} \,,\ 
  & X_w = \bP^2 \,;\\
  w = (231) \,,\ &  \multicolumn{2}{l}{
  \Delta(w) = \{ \a\in\Delta(\fm_1) \ | \ 
                 \a(\ttT_1+\ttT_3) \le 1 \,,\ \a(\ttT_4)=0 \} \,,\ } \\
  w = (421) \,,\ &  
  \Delta(w) = \{ \a\in\Delta(\fm_1) \ | \ \a(\ttT_3)=0 \} \,,\ 
  & X_w = \bP^2\times\bP^1 \,;\\
  w = (232) \,,\ & 
  \Delta(w) = \{ \a\in\Delta(\fm_1) \ | \ \a(\ttT_1+\ttT_4)=0 \} \,, 
  & X_w = \cQ^3 \ \subset\ \bP^4 \,.
\end{array}$$
\end{example*}

\begin{example*}
Let $\ttT_\varphi = \ttT_3+\ttT_4$.  We have $\tdim_\bC D = 21$.  To see that the underlying real form $\fm_\bR$ of Proposition \ref{P:HSgr} is $\FII$, let $w = (4) \in W$.  We leave it to the reader to confirm that $w\s_4$ is the unique noncompact root of $w\Sigma$; the claim then follows from the Vogan diagram classification \cite[p.416]{MR1920389}.  The two maximal Schubert VHS, both homogeneously embedded $\bP^1$s, are given by 
$$ \renewcommand{\arraystretch}{1.3}
\begin{array}{ll}
  w = (3) \,,\ &  
  \Delta(w) = \{ \a\in\Delta(\fm_1) \ | \ \a(\ttT_2+\ttT_4)=0 \} \,;\\
  w = (4) \,,\ & 
  \Delta(w) = \{ \a\in\Delta(\fm_1) \ | \ \a(\ttT_3)=0 \} \,.
\end{array}$$
\end{example*}

\begin{example*}
Let $\ttT_\varphi = \ttT_1+\ttT_2+\ttT_3$.  We have $\tdim_\bC D = 23$.  To see that the underlying real form $\fm_\bR$ of Proposition \ref{P:HSgr} is $\FI$, let $w = (342321) \in W$.  We leave it to the reader to confirm that $w\s_1$ is the unique noncompact root of $w\Sigma$; the claim then follows from the Vogan diagram classification \cite[p.416]{MR1920389}.  The two maximal Schubert VHS are given by 
$$ \renewcommand{\arraystretch}{1.3}
\begin{array}{lll}
  w = (2) \,,\ &  
  \Delta(w) = \{ \a\in\Delta(\fm_1) \ | \ \a(\ttT_1+\ttT_3)=0 \} \,,\
  & X_w = \bP^1 \,;\\
  w = (341) \,,\ & 
  \Delta(w) = \{ \a\in\Delta(\fm_1) \ | \ \a(\ttT_2)=0 \} \,,\
  & X_w = \bP^1 \times\bP^2\,.
\end{array}$$
\end{example*}

\begin{example*}
Let $\ttT_\varphi = \ttT_1+\ttT_2+\ttT_4$.  We have $\tdim_\bC D = 23$.  To see that the underlying real form $\fm_\bR$ of Proposition \ref{P:HSgr} is $\FI$, let $w = (2321) \in W$.  We leave it to the reader to confirm that $w\s_1$ is the unique noncompact root of $w\Sigma$; the claim then follows from the Vogan diagram classification \cite[p.416]{MR1920389}.  The three maximal Schubert VHS are given by 
$$ \renewcommand{\arraystretch}{1.3}
\begin{array}{lll}
  w = (42) \,,\ &  
  \Delta(w) = \{ \a\in\Delta(\fm_1) \ | \ \a(\ttT_1+\ttT_3)=0 \} \,,\
  & X_w = \bP^1 \times \bP^1 \,;\\
  w = (431) \,,\ & 
  \Delta(w) = \{ \a\in\Delta(\fm_1) \ | \ \a(\ttT_2)=0 \} \,,\
  & X_w = \bP^1 \times\bP^2\,;\\
  w = (232) \,,\ & 
  \Delta(w) = \{ \a\in\Delta(\fm_1) \ | \ \a(\ttT_1+\ttT_4)=0 \} \,,\
  & X_w = \cQ^3 \ \subset \ \bP^4\,.
\end{array}$$
\end{example*}

\begin{example*}
Let $\ttT_\varphi = \ttT_1+\ttT_3+\ttT_4$.  We have $\tdim_\bC D = 23$.  To see that the underlying real form $\fm_\bR$ of Proposition \ref{P:HSgr} is $\FI$, let $w = (412321) \in W$.  We leave it to the reader to confirm that $w\s_1$ is the unique noncompact root of $w\Sigma$; the claim then follows from the Vogan diagram classification \cite[p.416]{MR1920389}.  The two maximal Schubert VHS are given by 
$$ \renewcommand{\arraystretch}{1.3}
\begin{array}{lll}
  w = (31) \,,\ & 
  \Delta(w) = \{ \a\in\Delta(\fm_1) \ | \ \a(\ttT_2+\ttT_4)=0 \} \,,\
  & X_w = \bP^1 \times\bP^1\,;\\
  w = (412) \,,\ & 
  \Delta(w) = \{ \a\in\Delta(\fm_1) \ | \ \a(\ttT_3)=0 \} \,,\
  & X_w = \bP^2\times\bP^1\,.
\end{array}$$
\end{example*}

\begin{example*}
Let $\ttT_\varphi = \ttT_2+\ttT_3+\ttT_4$.  We have $\tdim_\bC D = 23$.  To see that the underlying real form $\fm_\bR$ of Proposition \ref{P:HSgr} is $\FI$, let $w = (321) \in W$.  We leave it to the reader to confirm that $w\s_1$ is the unique noncompact root of $w\Sigma$; the claim then follows from the Vogan diagram classification \cite[p.416]{MR1920389}.  The two maximal Schubert VHS are given by 
$$ \renewcommand{\arraystretch}{1.3}
\begin{array}{lll}
  w = (3) \,,\ & 
  \Delta(w) = \{ \a\in\Delta(\fm_1) \ | \ \a(\ttT_2+\ttT_4)=0 \} \,,\
  & X_w = \bP^1\,;\\
  w = (421) \,,\ & 
  \Delta(w) = \{ \a\in\Delta(\fm_1) \ | \ \a(\ttT_3)=0 \} \,,\
  & X_w = \bP^2\times\bP^1\,.
\end{array}$$
\end{example*}

\begin{example*}
Let $\ttT_\varphi = \ttT_1+\ttT_2+\ttT_3+\ttT_4$.  We have $\tdim_\bC D = 24$.  To see that the underlying real form $\fm_\bR$ of Proposition \ref{P:HSgr} is $\FI$, let $w = (42321) \in W$.  We leave it to the reader to confirm that $w\s_1$ is the unique noncompact root of $w\Sigma$; the claim then follows from the Vogan diagram classification \cite[p.416]{MR1920389}.  The three maximal Schubert VHS, all homogeneously embedded $\bP^1\times\bP^1$s, are given by 
$$ \renewcommand{\arraystretch}{1.3}
\begin{array}{ll}
  w = (31) \,,\ & 
  \Delta(w) = \{ \a\in\Delta(\fm_1) \ | \ \a(\ttT_2+\ttT_4)=0 \} \,;\\
  w = (41) \,,\ & 
  \Delta(w) = \{ \a\in\Delta(\fm_1) \ | \ \a(\ttT_2+\ttT_3)=0 \} \,;\\
  w =(42) \,,\ & 
  \Delta(w) = \{ \a\in\Delta(\fm_1) \ | \ \a(\ttT_1+\ttT_3)=0 \} \,.
\end{array}$$
\end{example*}

\subsection{\boldmath The exceptional Hodge group $G = G_2$ \unboldmath} \label{A:egG}

In the case that $\fg$ is the exceptional Lie algebra $\fg_2$, there are 3 grading elements $\ttT_\varphi$ of the form \eqref{E:TI}. Each is considered in the examples below; these examples, along with Proposition \ref{P:red}, yield

\begin{corollary*} 
Variations of Hodge structure in $\check D = \mathrm{G}_2(\bC)/P_\varphi$ are of dimension at most 2.
\end{corollary*}

For the exceptional Lie algebra $\fg_2$, Vogan diagram classification \cite[p.416]{MR1920389} implies that the underlying real form $\fm_\bR = \fk_\bR \op \fq_\bR$ of Proposition \ref{P:HSgr} is the split real form with maximal compact subalgebra $\fk_\bR = \fsu(2) \op \fsu(2)$.  

\begin{example*}
Suppose $\fm_\bC = \fg_2$ and $\ttT_\varphi = \ttT_1$.  We have $\tdim_\bC D = 5$.  The unique maximal integral $X_{(1)}$ is a homogeneously embedded $\bP^1$ given by $\Delta(1) = \{ \a\in\Delta \ | \ \a(\ttT_1)=1\,,\ \a(\ttT_2)=0 \}$.  
\end{example*}

\begin{example*}
Suppose $\fm_\bC = \fg_2$ and $\ttT_\varphi = \ttT_2$.  We have $\tdim_\bC D = 5$.  There is a unique maximal Schubert VHS, $X_{(21)}$, with $\Delta(21) = \{ \a\in\Delta \ | \ \a(\ttT_2) = 1 \,,\ \a(\ttT_1)\le1\}$. 
\end{example*}

\begin{example*}
Suppose $\fm_\bC = \fg_2$ and $\ttT_\varphi = \ttT_1 + \ttT_2$.  We have $\tdim_\bC D = 5$.  There are two maximal Schubert VHS, $X_{(1)}$ and $X_{(2)}$, both are homogeneously embedded $\bP^1$'s given by $\Delta(1) =  \{ \a\in\Delta \ | \ \a(\ttT_1)=1\,,\ \a(\ttT_2)=0 \}$ and $\Delta(2) =  \{ \a\in\Delta \ | \ \a(\ttT_2)=1\,,\ \a(\ttT_1)=0 \}$.
\end{example*}

\section{Hodge representations of Calabi--Yau type}\label{A:CY}

\begin{definition*}
A Hodge structure $\varphi$ of weight $n$ on $V$ is of \emph{Calabi--Yau type} if $\tdim_\bC V^{m/2} = 1$ in the decomposition \eqref{E:Vdecomp}.\footnote{In the case that $\varphi$ is an effective Hodge structure of weight $n\ge0$, this is equivalent to $h^{n,0} = 1$.}  A Hodge representation, given by $\rho : G \to \tAut(V,Q)$ and $\varphi : S^1 \to G_\bR$, is of \emph{Calabi--Yau type} if the Hodge structure $\rho\circ\varphi$ on $V$ is of Calabi--Yau type.
\end{definition*}

\begin{proposition} \label{P:CY}
Let $G$ be a Hodge group.  Let $\rho : G \to \tAut(V,Q)$ and $\varphi : S^1 \to G_\bR$ define a Hodge representation.  Without loss of generality, the grading element $\ttT_\varphi = \sum n_i \ttT_i$ associated to the circle $\varphi$ by \eqref{E:dfnT} is of the form \eqref{E:Tphi}.\footnote{As noted in \eqref{E:Tpos}, the Borel subalgebra $\fb$ (equivalently, the simple roots $\Sigma \subset \Delta$) may be chosen so that the $n_i$ are non-negative.  Note that the reduction to \eqref{E:TI} is \emph{not} imposed.}  Assume $V_\bR$ is irreducible, and let $U$ be the associated irreducible, complex $G$--representation of highest weight $\m = \m^i\w_i$ (cf. Remark \ref{R:RCQ}).  Then, the Hodge representation is of Calabi--Yau type if and only if $\{ i \ | \ n_i=0\} \subset \{ i \ | \ \mu^i =0\}$, and one of the following holds:
\begin{a_list}
  \item The representation $U$ is real.  Equivalently, $\mu = \mu^*$, and $\sum_{n_i \in 2\bZ} \mu(\ttT_i)$ is an integer.
  \item The inequality $\m(\ttT_\varphi) \not= \m^*(\ttT_\varphi)$ holds.  (In this case, $U$ is necessarily complex.)
\end{a_list}
\end{proposition}

\begin{remark*}
Assume that $D = G_\bR/H_\varphi$ is Hermitian symmetric.  In \cite[Lemma 2.27]{FL}, Friedman--Laza identify $\{ i \ | \ n_i=0\} \subset \{ i \ | \ \mu^i =0\}$ as a necessary condition for the Hodge structure $\rho\circ\varphi$ to be of Calabi--Yau type, and show that $U$ is either real or complex.
\end{remark*}

\begin{example*}
Let $\fg$ be one of the exceptional Lie algebras $\fe_8$, $\ff_4$ or $\fg_2$.  Fix a lattice $\rtL \subset \Lambda \subset \wtL$ and Lie group $G = G_\Lambda$ with Lie algebra $\fg$ (cf. Section \ref{S:HS}).  By Remark \ref{R:mu*}, all representations $U$ of $\fg_\bC$ are self-dual.  Therefore, by Remark \ref{R:RCQ}, all irreducible representations $U$ of $\fg_\bC$ are either real or quaternionic.  Moreover, $\w_i(\ttT_j) \in \bZ$ for all $i,j$, cf. \cite[Appendix C.2]{MR1920389}.  That is, for any grading element $\ttT \in \tHom(\rtL,\bZ)$, we have $\ttT \in \tHom(\Lambda,\bZ)$.  This has two consequences:
\begin{circlist}
\item Given any irreducible $G_\bC$--representation $U$ of highest weight $\m\in\Lambda$ and any grading element $\ttT$, it follows from Remark \ref{R:RCQ} that $U$ is real.  That is, $\mu(\ttH_\mathrm{cpt})$ is always even, independent of our choice of $\varphi$.   Thus, $U = V_\bC$, for some real vector space $V$.
\item Let $\varphi : S^1 \to T = \ft_\bR/\Lambda^*$ be the homomorphism of $\bR$--algebraic groups determined by $\ttT$, cf. \S\ref{S:HS}.  Then, Theorem \ref{T:HdgRep} implies that $(V,\varphi)$ admits the structure of a Hodge representation.  Without loss of generality $\ttT = \sum n_i\ttT_i$ with $0 \le n_i\in\bZ$.  Whence, Proposition \ref{P:CY} implies that the Hodge representation $(V,\varphi)$ is of Calabi--Yau type if and only if $\m^i = 0$ for all $i$ such that $n_i=0$.
\end{circlist} 
\end{example*}

Set 
$$
  I^\m \ = \  \{ i \ | \ \m^i > 0 \} \ = \ \{ i \ | \ (\s_i,\m) \not= 0 \} 
  \quad\hbox{and}\quad 
  I^\varphi \ = \  \{ i \ | \ n_i > 0 \} \,.
$$
Note that
\begin{equation}\label{E:CY0}
  \{ i \ | \ n_i=0\} \,\subset\, \{ i \ | \ \mu^i =0\} 
  \quad\hbox{if and only if}\quad I^\m \subset I^\varphi \, .
\end{equation}
Let $\ttT_\m$ be the grading element \eqref{E:gr} associated to $I^\m$.  Let 
$$ U \ = \ \op\,U_k^\m\quad\hbox{ and }\quad U \ = \ \op\,U_\ell$$ 
respectively be the $\ttT_\m$ and $\ttT_\varphi$ graded decompositions \eqref{E:grU}.  Set 
$$\sfm \ = \ \m(\ttT_\m)\quad\hbox{ and } \quad\sfp \ = \ \m(\ttT_\varphi)\,.$$
Then, the highest weight line of $U$ is contained in $U^\m_\sfm \,\cap\, U_\sfp$.   

\begin{lemma} \label{L:CY}
The highest weight line of $U$ is $U^\m_\sfm$.  Thus $U^\m_\sfm \subset U_\sfp$.  Equality holds if and only if $I^\m \subset I^\varphi$.
\end{lemma}

\begin{proof}[Proof of lemma]
The lemma is a consequence of two facts.  First, note that every weight $\lambda$ of $U$ is of the form $\lambda = \m -(\s_{j_1} + \cdots + \s_{j_s})$ for some sequence $\{\sigma_{j_\ell}\}$ of simple roots with the property that each $\m -(\s_{j_1} + \cdots + \s_{j_t})$ is also a weight for $1 \le t \le s$.  (This is an elementary result from representation theory; see, for example, \cite{MR1920389}.)  

Second, let $\fp_\m$ be the parabolic subalgebra \eqref{E:p} determined by $\ttT_\m$.  Then $\fp_\m$ is the stabilizer of the highest weight line, cf. \cite[Proposition 3.2.5 or Theorem 3.2.8]{MR2532439}.  Consequently, given a simple root $\s_i \in \Sigma$, $\lambda=\m-\sigma_i$ is a weight of $U$ if and only if $\fg_{-\s_i} \not\subset \fp_\m$; equivalently, $i \in  I^\m$.   

These two facts imply that any weight $\lambda \not= \m$ of $U$ is necessarily of the form $\lambda = \m - \s_i - (\s_{j_2} + \cdots + \s_{j_s})$ for some $i \in I^\m$.  Therefore, $\lambda(\ttT_\m) = \m(\ttT_\m) - 1 - (\s_{j_2}+\cdots \s_{j_s})(\ttT_\m) \le \sfm - 1$.  Therefore, $U^\m_\sfm$ is necessarily the highest weight line.

The discussion preceding the statement of the lemma yields $U^\m_\sfm \subset U_\sfp$.  Moreover, $\lambda(\ttT_\varphi) = \sfp - (\s_i+\s_{j_2}+\cdots+\s_{j_s})(\ttT_\varphi)$.  So, $U^\m_\sfm = U_\sfp$ holds if and only if $I^\m \subset I^\varphi$. 
\end{proof}

Set $-\sfq = -\m^*(\ttT_\varphi)$.  By Remark \ref{R:mu*}, the lowest weight line of $U$ is contained in $U_{-\sfq}$.  In particular, the $\ttT_\varphi$--graded decompositions \eqref{E:grU} of $U$ and $U^*$ are 
\begin{eqnarray*}
  U & = &  
  U_{\sfp} \op U_{\sfp-1} \op \cdots \op U_{1-\sfq} \op U_{-\sfq} \,,\\
  U^* & = &  
  U^*_{\sfq} \op U^*_{\sfq-1} \op \cdots \op U^*_{1-\sfp} \op U^*_{-\sfp} \,.
\end{eqnarray*}
The proof of Proposition \ref{P:CY} breaks into three cases. 

\begin{proof}[Proof of Proposition \ref{P:CY} when $U$ is real]
First, as noted in Remark \ref{R:RCQ}, $U$ is real if and only if $\mu = \mu^*$ and $\mu(\ttH_\mathrm{cpt})$ is even.  Since $\ttH_\mathrm{cpt} = 2 \sum_{n_i \in2\bZ} \ttT_i$, the latter is equivalent to $\sum_{n_i \in2\bZ} \mu(\ttT_i) \in \bZ$.

Given $U$ real, we have $V_\bC=U$.  So $V_{m/2} = U_\sfp$.  By definition, the Hodge representation is of Calabi--Yau type if and only if $\tdim_\bC U_\sfp=1$.  By Lemma \ref{L:CY}, $\tdim_\bC U_\sfp = 1$ if and only if $I^\m \subset I^\varphi$.   Thus, $(V,\rho\circ\varphi)$ is a Hodge structure of Calabi--Yau type if and only if $I^\m \subset I^\varphi$.  The proposition now follows from \eqref{E:CY0}.
\end{proof}

\begin{proof}[Proof of Proposition \ref{P:CY} when $U$ is quaternionic]
As noted in Remark \ref{R:RCQ}, $U$ is quaternionic if and only if $V_\bC = U \op U^*$ and $U \simeq U^*$.  Then $\sfp = \sfq$, so that $V_{m/2} = U_\sfp\op U_\sfp$ is of dimension at least two.  Therefore, $(V,\rho\circ\varphi)$ is not a Hodge structure of Calabi--Yau type.
\end{proof}

\begin{proof}[Proof of Proposition \ref{P:CY} when $U$ is complex ]
As noted in Remark \ref{R:RCQ}, $U$ is complex if and only if $V_\bC = U \op U^*$ and $U \not\simeq U^*$.  Swapping $U$ and $U^*$ if necessary, we may assume that $\sfp \ge \sfq$.  Then, $V_{m/2} = U_\sfp \op U^*_\sfp$, and $U^*_\sfp=0$ if and only if $\sfp > \sfq$.  Therefore, if $\sfp=\sfq$, then $\tdim_\bC V_{m/2} >1$ and $(V,\rho\circ\varphi)$ is not of Calabi--Yau type.   If $\sfp>\sfq$, then Lemma \ref{L:CY} implies that $(V,\rho\circ\varphi)$ gives a Hodge structure of Calabi--Yau type if and only if $I^\m \subset I^\varphi$.  The proposition now follows from \eqref{E:CY0}.
\end{proof}

\noindent
This completes the proof of Proposition \ref{P:CY}.  

\section{Proof of Proposition \ref{P:sm}} \label{A:p}

\subsection{Preliminaries} \label{A:p0}
Let $I \subset \{ 1,\ldots,r\}$ be the index set associated to $\ttT_\varphi$, cf. \eqref{E:TI}, so that 
$$
  \ttT_\varphi \ = \ \sum_{i \in I} \ttT_i \,.
$$
Let $X_w$ be a maximal Schubert VHS.  By Corollary \ref{C:ww'}, this is equivalent to: $\Delta(w)$ is maximal, with respect to containment, among the 
  $\{ \Delta(w') \ | \ w' \in W^\varphi_\sI\}$.  By Lemma \ref{L:abel}, the IVHS $\fn_w$, defined by \eqref{E:nw}, is abelian; equivalently, 
\begin{equation} \label{E:abrt}
\hbox{if $\a,\b\in\Delta(w)$, then $\a+\b$ is not a root.}
\end{equation}
By \eqref{E:Dm1}, we may decompose $\Delta(w)$ as the disjoint union $\sqcup_{i\in I} \Delta_i(w)$ where 
$$
  \Delta_i(w) \ = \ \{ \a \in \Delta(w) \ | \ \a(\ttT_i) = 1 \} \, .
$$
In our descriptions of the set $\Delta_i(w)$ below, it is essential to keep in mind that $\a(\ttT_j) = 0$ for all $\a\in\Delta_i(w)$ and $i\not= j \in I$.  (This is due to the fact that the positive roots are of the form $m^a\s_a$ with $0 \le m^a \in \bZ$, cf. \cite{MR1920389}.)  Write 
$$
  I \ = \ \{ i_1 < i_2 < \cdots < i_t \} \,. 
$$
It will be convenient to set 
$$
  i_0 \,= \, 0 \quad\hbox{and}\quad i_{t+1} \,=\, r+1 \,.
$$

\subsection{The case 
\boldmath $G_\bC = \tSL_{r+1}\bC$ \unboldmath } \label{A:psl}
We will use throughout the proof the standard representation theoretic result that the positive roots of $\fsl_{r+1}\bC$ are of the form
\begin{equation} \label{E:Art}
  \Delta^+ \ = \ \{ \s_j + \cdots + \s_k \ | \ 1 \le j \le k \le r \} \, ,
\end{equation}
cf. \cite{MR1920389}.  In particular, 
\begin{equation} \label{E:Apf3}
  \Delta_{i_s}(w) \ \subset \ \{ \s_j + \cdots + \s_k  \ | \ 
  i_{s-1} \,<\, j \,\le\, i_s \,\le\, k \,<\, i_{s+1} \} \, .
\end{equation}

\subsubsection*{Step 1:}  
Suppose that $\Delta_i(w)$ is empty for some $i\in I$.  Then $\Delta(w) \subset \Delta(\fa_{i-1} \times \fa_{r-i}) = \Delta(\fa_{i-1}) \cup \Delta(\fa_{r-i})$, where $\fa_{i-1} \times \fa_{r-i} = \fsl_i\bC \times \fsl_{r+1-i}\bC$ is the semisimple subalgebra of $\fm_\bC$ generated by the simple roots $\Sigma \backslash\{ \s_i\}$.  Therefore, $X_w = X_{w^1} \times X_{w^2} \subset (\tSL_i\bC/P_{I^1}) \times (\tSL_{r+1-i}/P_{I^2})$, where $\Delta(w^1) = \Delta(w) \cap \Delta(\fa_{i-1})$ and $\Delta(w^2) = \Delta(w) \cap \Delta(\fa_{r-i})$, and $I^1 = \{ j \in I \ | \ j < i \}$ and $I^2 = \{ j \in I \ | \ i < j \}$.  The Schubert variety $X_w$ is Hermitian symmetric if and only if $X_{w^1}$ and $X_{w^2}$ are Hermitian symmetric.  So, without loss of generality, \emph{we may restrict to the case that $\Delta_i(w)$ is nonempty for all $i \in I$.}

\subsubsection*{Step 2:}  
Suppose that $i_s+1 = i_{s+1}$ for some $1 \le s \le t-1$.  By \eqref{E:Apf3}, any pair of roots $\a \in \Delta_{i_s}(w)$ and $\b \in \Delta_{i_{s+1}}(w)$ is of the form 
\begin{eqnarray*}
  \a & = & \s_j + \cdots + \s_{i_s} \,,
  \quad\hbox{ for some } i_{s-1} < j \le i_s \,,\\
  \b & = & \s_{i_{s+1}} + \cdots + \s_k \,,
  \quad\hbox{ for some } i_{s+1} \le k < i_{s+2} \, .
\end{eqnarray*}
Therefore, $\a+\b$ is a root, contradicting \eqref{E:abrt}.  Therefore, 
$$
  i_s + 1 \ < \ i_{s+1} \,.
$$

\subsubsection*{Step 3:}  
Suppose that $i_{s-1}<j < i_s\le k < i_{s+1}$, and $\b = \s_j + \cdots + \s_k \in \Delta(w)$.  I claim that $\b-\s_j \in \Delta(w)$ as well.  To see this, suppose the converse.  Then $\b-\s_j \in \Delta^+\backslash\Delta(w)$.  Also, $\s_j \in \Delta^+(\fg_0) \subset \Delta^+\backslash\Delta(w)$.  Since $\Delta^+\backslash\Delta(w)$ is closed, cf. Remark \ref{R:Dw}(c), this implies $\b \not\in\Delta(w)$, a contradiction.  Therefore, $\b-\s_j\in\Delta(w)$.  Similarly, if $i_s < k$, then $\b-\s_k = \s_j + \cdots + \s_{k-1} \in \Delta(w)$.  It follows by induction that, if $\s_j + \cdots + \s_k \in \Delta(w)$, then $\s_{j'} + \cdots + \s_{k'} \in \Delta(w)$ for all $j \le j' \le i_s \le k' \le k$.

\subsubsection*{Step 4:}  
With respect to the expression \eqref{E:Apf3}, let $j_s$ be the minimal $j$, and $k_s$ the maximal $k$, as $\a$ ranges over $\Delta_{i_s}(w)$.  Suppose that $k_s + 1 \ge j_{s+1}$.  Let $\a \in \Delta_{i_s}(w)$ be a root realizing the maximum value $k_s$, and $\b\in\Delta_{i_{s+1}}(w)$ be a root realizing the minimum value $j_{s+1}$.  Then $\a = \s_a + \cdots + \s_{k_s}$ and $\b = \s_{j_{s+1}} + \cdots + \s_b$ for some $i_{s-1} < a \le i_s$ and $i_{s+1} \le b < i_{s+2}$.  By Step 3, $\c = \s_{k_s+1} + \cdots +\s_b \in \Delta(w)$.  However, $\a+\c$ is a root, contradicting \eqref{E:abrt}.  We conclude that 
\begin{equation} \label{E:Apf4}
  k_s + 1 \ < \ j_{s+1} \,.
\end{equation}

\subsubsection*{Step 5:}  
Set $\Phi_s = \{ \s_j + \cdots + \s_k \ | \ j_s \le j \le i_s \le k \le k_s \}$.  Then $\Delta_{i_s}(w) \subset \Phi_s$ and $\Delta(w) \subset \Phi = \cup_{1\le s \le t} \Phi_s \subset \Delta(\fm_1)$.  Moreover, \eqref{E:Art} and \eqref{E:Apf4} imply that both $\Phi$ and $\Delta^+\backslash\Phi$ are closed.  Therefore, by Remark \ref{R:Dw}(c), $\Phi = \Delta(w')$ for some $w' \in W^\varphi$.  By \eqref{E:Wint}, $w' \in W^\varphi_\sI$.  It follows from the maximality of $\Delta(w)$ (cf. Section \ref{A:p0}) that $\Delta(w) = \Phi$.  That is,
\begin{equation} \label{E:Apf1}
  \Delta_{i_s}(w) \ = \ \{ \s_j + \cdots + \s_k \ | \ 
  j_s \le j \le i_s \le k \le k_s \} \, .
\end{equation}

\subsubsection*{Step 6:}  
Let $s < t$.  Steps 3 and 4 imply $\a = \s_{i_s} + \cdots + \s_{k_s+1} \in\Delta(\fm_1)\backslash\Delta(w)$.  Suppose that $k_s + 2 < j_{s+1}$.  Let $\Phi = \Delta(w) \cup \{ \a\}$.  I claim that $\Phi = \Delta(w')$ for some $w' \in W^\varphi_\sI$.  This will contradict the maximality of $\Delta(w)$, and taken with \eqref{E:Apf4} allows us to conclude that
\begin{equation} \label{E:Apf2}
  k_s + 2 \ = \ j_{s+1} \quad \hbox{ for all } \ 1 \le s \le t-1 \, .
\end{equation}

To prove the claim, first note that $\Phi \subset \Delta(\fm_1)$.  So, if $\Phi = \Delta(w')$ for some $w' \in W^\varphi$, then it follows immediately from \eqref{E:Wint} that $w' \in W^\varphi_\sI$.  By Remark \ref{R:Dw}(c), $\Phi = \Delta(w')$ for some $w' \in W^\varphi$ if and only if both $\Phi$ and $\Delta^+\backslash\Phi$ are closed.  Since $\Delta(w)$ is closed itself, $\Phi$ can fail to be closed only if $\a+\b \in \Delta$ for some $\b\in \Delta(w)$.  However, it follows from \eqref{E:Art} and the hypothesis $k_s + 2 < j_{s+1}$, that $\a+\b$ is not a root for any $\b \in \Delta(w)$.  Thus $\Phi$ is closed.  

Similarly, since $\Delta^+\backslash\Delta(w)$ is closed (Remark \ref{R:Dw}(c)), to see that $\Delta^+\backslash\Phi = \Delta^+\backslash ( \Delta(w) \cup \{\a\} )$ is closed, it suffices to show that there exist no $\b\in\Delta(\fm_1)\backslash\Phi$ and $\c\in\Delta^+(\fm_0)$ such that $\b+\c = \a$.  By \eqref{E:Art}, any such pair would necessarily be of the form 
$$
  \b \, = \, \s_{i_s} + \cdots + \s_b 
  \quad\hbox{and}\quad 
  \c \,=\, \s_{b+1} + \cdots + \s_{k_s+1} \,,
$$
for some $i_s \le b \le k_s$.  By \eqref{E:Apf1}, $\s_{i_s} + \cdots + \s_b \in \Delta(w)\subset\Phi$ for all $i_s \le b \le k_s$.  Thus, there exists no such $\b \in \Delta(\fm_1) \backslash \Phi$, and we may conclude that $\Delta^+\backslash\Phi$ is closed.

\subsubsection*{Step 7:}  
Define 
$$
  A \ = \ \{ a_{s-1} = j_s - 1 \ | \ 2 \le s \le t\} \ \stackrel{\eqref{E:Apf2}}{=} \ 
  \{ a_s = k_s + 1 \ | \ 1 \le s \le t-1 \} \, .
$$
From \eqref{E:Apf1} and \eqref{E:Apf2} we deduce that $\Delta(w) = \{ \a \in \Delta(\fg_1) \ | \ \a(\ttT_a)=0 \ \forall \ a \in A \} = \{ \a \in \Delta(\fg_1) \ | \ \a(\ttT_A)=0 \}$.\footnote{The claim requires our assumption that $\Delta_i(w)$ is nonempty for all $i \in I$, cf. Step 1.}  From this it follows that 
\begin{equation} \label{E:XwA}
\renewcommand{\arraystretch}{1.3}
\begin{array}{rcl}
  X_w & = & \tGr(i_1 , a_1) \ \times \tGr( i_2-a_1 , a_2 - a_1) \ \times \ \cdots\\
  & & \cdots \ \times \
  \tGr(i_{t-1} - a_{t-2} , a_{t-1}-a_{t-2}) \ \times \
  \tGr( i_t - a_{t-1} , r-a_{t-1} ) \,.
\end{array}
\end{equation}
This establishes the proposition in the case that $G_\bC = \tSL_{r+1}\bC$.

\begin{remark*}
It is can be checked that the maximal Schubert VHS $X_w$ satisfying the assumption of Step 1 (that the $\Delta_i(w) \not=\emptyset$ for all $i$) are indexed by subsets $A = \{ a_1 < \cdots < a_{t-1} \}$ satisfying $i_s < a_s < i_{s+1}$ for all $1 \le s \le t-1$.  
\end{remark*}

\subsection{The case 
\boldmath $G_\bC = \tSp_{2r}\bC$ \unboldmath } \label{A:psp}

We will use throughout the proof the standard representation theoretic result that the positive roots of $\fsp_{r+1}\bC$ are of the form
\begin{eqnarray*}
  \Delta^+ & = & 
  \{ \s_j + \cdots + \s_k \ | \ 1 \le j \le k \le r \} \\
  & & \quad \cup \ 
  \{ \s_j + \cdots + \s_{k-1} + 2(\s_k + \cdots +\s_{r-1}) + \s_r 
  \ | \ 1 \le j \le k \le r-1 \}\\
  & = & \{ \s_{j,k} \ | \ 1 \le j \le k \le r \} \ \cup \ 
  \{ \s_{j,r-1} + \s_{k,r} \ | \ 1 \le j \le k \le r-1 \} \, ,\\
  & & \hbox{where} \quad 
  \s_{j,k} \ = \ \s_j + \cdots + \s_k \,, \ 1 \le j \le k \le r \, ,
\end{eqnarray*}
and we employ the convention that $\s_{r,r-1}=0$; cf. \cite{MR1920389}.  In particular,  
\begin{equation} \label{E:Cpf3a}
\renewcommand{\arraystretch}{1.3}
\begin{array}{rrcl}
  \hbox{if $s<t$,} & \Delta_{i_s}(w) & \subset & \{ \s_{j,k} \ | \ 
  i_{s-1} \,<\, j \,\le\, i_s \,\le\, k \,<\, i_{s+1} \} \,;\\
  \hbox{if $i_t<r$,} & \Delta_{i_t}(w) & \subset & \{ \s_{j,k} \ | \ 
  i_{t-1} \,<\, j \,\le\, i_t \,\le\, k \,\le\, r \} \\
  & & & 
  \quad \cup \ \ \{ \s_{j,r-1} + \s_{k,r} \ | \ 
  i_{t-1} \,<\, j \,\le\, i_t \,<\, k \,\le \, r-1\} \,;\\
  \hbox{if $i_t=r$,} & \Delta_{i_t}(w) & \subset & \{ \s_{j,r-1} + \s_{k,r} \ | \ 
  i_{t-1} \,<\, j , k \,\le \, r\} \,.
\end{array}
\end{equation}

\subsubsection*{Step 1:} 
Suppose $\Delta_i(w)$ is empty, for some $i \in I$.  Then $\Delta(w) \subset \Delta(\fa_{i-1} \times \fc_{r-i}) = \Delta(\fa_{i-1}) \cup \Delta( \fc_{r-i})$, where $\fa_{i-1} \times \fc_{r-i} = \fsl_i\bC \times \fsp_{2(r-i)}\bC$ is the semisimple subalgebra of $\fm_\bC$ generated by the simple roots $\Sigma \backslash\{\s_i\}$.  Therefore, $X_w = X_{w^1} \times X_{w^2} \subset (\tSL_i\bC/P_{I^1}) \times (\tSp_{2(r-i)}/P_{I^2})$, where $\Delta(w^1) = \Delta(w) \cap \Delta(\fa_{i-1})$ and $\Delta(w^2) = \Delta(w) \cap \Delta(\fc_{r-i})$, and $I^1 = \{ j \in I \ | \ j < i \}$ and $I^2 = \{ j \in I \ | \ i < j \}$.  The Schubert variety $X_w$ is homogeneous if and only if $X_{w^1}$ and $X_{w^2}$ are homogeneous.  So, without loss of generality, \emph{we may restrict to the case that $\Delta_i(w)$ is nonempty for all $i \in I$.}

\subsubsection*{Step 2:}  
Suppose that $i_s+1 = i_{s+1}$ for some $1 \le s \le t-1$.  Then \eqref{E:Cpf3a} implies that, for any pair of roots $\a \in \Delta_{i_s}(w)$ and $\b \in \Delta_{i_{s+1}}(w)$, we have $\a+\b\in\Delta$.  This contradicts \eqref{E:abrt}.  Therefore, 
$$
  i_s + 1 \ < \ i_{s+1} \,.
$$

\subsubsection*{Step 3:}  
Let $\a \in \Delta_{i_s}(w)$ be as given in \eqref{E:Cpf3a}.  Suppose that $j < i_s$ and $\b = \a-\s_j \not \in \Delta(w)$.  Then $\s_j \in \Delta^+(\fm_0)$, so that $\b,\s_j \in \Delta^+\backslash\Delta(w)$ while $\b+\s_j \in \Delta(w)$.  This contradicts the closure of $\Delta^+\backslash\Delta(w)$, cf. Remark \ref{R:Dw}(c).  Therefore, $\b\in\Delta(w)$.  Similarly, if $i_s < k$, then $\b =\a-\s_k \in \Delta(w)$.  It follows by induction that, 
\begin{a_list_nem}
\item if $\s_{j,k} \in \Delta_{i_s}(w)$, then $\s_{j',k'} \in \Delta_{i_s}(w)$ for all $j \le j' \le i_s \le k' \le k$;
\item if $i_t < r$ and $\s_{j,r-1} + \s_{k,r} \in \Delta_{i_t}(w)$, then $\s_{j',r-1} + \s_{k',r}\in \Delta_{i_t}(w)$, for all $j \le j' < i_t$ and $k \le k' \le r$;
\item if $i_t = r$ and $\s_{j,r-1} + \s_{k,r} \in \Delta_r(w)$, then $\s_{j',r-1} + \s_{k',r}$ for all $j \le j'$ and $k \le k'$.
\end{a_list_nem}
Consider the case that $i_t < r$.   If $\a=\s_{j,r}\in \Delta_{i_t}(w)$, then (a) implies $\b=\s_{j,r-1} \in \Delta_{i_t}(w)$.  But $\a+\b \in \Delta$, contradicting \eqref{E:abrt}.  Therefore, $\s_{j,r}\not\in \Delta_{i_t}(w)$, if $i_t < r$.

Continuing with $i_t < r$, suppose that $\a=\s_{j,r-1} + \s_{k,r} \in \Delta_{i_t}(w)$.  Then (b) implies $\s_{j,r-1} + \s_{r,r} = \s_{j,r}\in \Delta_{i_t}(w)$; we have just seen that this is not possible.  Therefore, $\s_{j,r-1} + \s_{k,r} \not\in \Delta_{i_t}(w)$, if $i_t < r$.  

These two observations allow us to update \eqref{E:Cpf3a} to 
$$
  \hbox{if $i_t<r$,} \quad\hbox{then}\quad
  \Delta_{i_t}(w) \ \subset \ \{ \s_{j,k} \ | \ 
  i_{t-1} \,<\, j \,\le\, i_t \,\le\, k \,< \, r\} \,.
$$
In particular, if $i_t < r$, then $\Delta(w) \subset \{ \a\in\Delta^+ \ | \ \a(\ttT_r) = 0\} = \Delta^+(\fa_{r-1})$, where $\fa_{r-1} = \fsl_r\bC$ is the simple subalgebra of $\fm_\bC = \fsp_{2r}\bC$ generated by the simple roots $\Sigma \backslash\{\s_r\}$.  That is, we are reduced to the case that $X_w$ is a maximal Schubert VHS of a homogeneously embedded $\tSL_{r}\bC/P \subset \check D = G_\tad/P_\varphi$.  This is precisely the case addressed in Section \ref{A:psl}, and $X_w$ is necessarily of the form \eqref{E:XwA}.
\begin{center}
\emph{For the remainder of the proof, assume that $i_t = r$.}
\end{center}
In particular, the roots of $\Delta(w) = \sqcup_{i\in I} \Delta_i(w)$ are of the form
\begin{equation} \label{E:Cpf3b}
\renewcommand{\arraystretch}{1.3}
\begin{array}{rrcl}
  \hbox{if $s<t$,} & \Delta_{i_s}(w) & \subset & \{ \s_{j,k} \ | \ 
  i_{s-1} \,<\, j \,\le\, i_s \,\le\, k \,<\, i_{s+1} \} \,;\\
  \hbox{and} & \Delta_{r}(w) & \subset & \{ \s_{j,r-1} + \s_{k,r} \ | \ 
  i_{t-1} \,<\, j , k \,\le \, r\} \,.
\end{array}
\end{equation}

\subsubsection*{Step 4:}
Fix $s < t$.  With respect to \eqref{E:Cpf3b}, let $j_s$ be the minimal $j$, and $k_s$ the maximal $k$, amongst all $\a \in \Delta_{i_s}(w)$.  If $s < t-1$, then the argument of Section \ref{A:psl}, Step 4, yields $k_s + 1 < j_{s+1}$.  It remains to consider the case $s=t-1$.

With respect to \eqref{E:Cpf3b}, let $j_t$ be the minimal $j$ over all $\a\in\Delta_r(w)$.  Suppose that $k_{t-1}+1 \ge j_t$.  Let $\a \in \Delta_{i_{t-1}}(w)$ be a root realizing the maximum value $k_{t-1}$, and $\b\in\Delta_r(w)$ a root realizing the minimum value $j_t$.  Then $\a = \s_{a,k_{t-1}}$ for some $a \le i_{t-1}$, and $\b = \s_{j_t,r-1} + \s_{k,r}$ for some $j_t \le k \le r$.  Step 3(c) applied to $\b$ yields $\c = \s_{k_{t-1}+1,r-1} + \s_{k,r} \in \Delta_r(w)$.  Then $\a+\c = \s_{a,r-1} + \s_{k,r} \in \Delta$, contradicting \eqref{E:abrt}.  Thus,
$$
  k_s + 1 \ < \ j_{s+1} \quad\hbox{for all}\quad 1\le s \le r-1 \, .
$$

\subsubsection*{Step 5:}
Given $s < t$, define $\Phi_s = \{ \s_{j,k} \ | \ j_s \le j \le i_s \le k \le k_s\}$.  Set $\Phi_t = \{ \s_{j,r-1} + \s_{k,r} \ | \ j_t \le j,k \le r \}$.  Then $\Delta_{i_s}(w) \subset \Phi_s$ for all $1 \le s \le t$.  Thus $\Delta(w) \subset \Phi = \cup_s \Phi_s$.  An argument analogous to that of Section \ref{A:psl}, Step 5, yields $\Delta(w) = \Phi$.  That is,
\begin{equation} \label{E:Cpf1}
\renewcommand{\arraystretch}{1.3}
\begin{array}{rcl}
  \Delta_{i_s} & = & \{ \s_{j,k} \ | \ j_s \le j \le i_s \le k \le k_s \} \,, 
  \quad\hbox{for all} \quad s < t \,,\\
  \Delta_r(w) & = & \{ \s_{j,r-1} + \s_{k,r} \ | \ j_t \le j , k \le r \} \, .
\end{array}
\end{equation}
Details are left to the reader.

\subsubsection*{Step 6:}
Let $s < t$.  Arguing as in Step 6 of Section \ref{A:psl}, we may show that 
\begin{equation} \label{E:Cpf2}
  k_s+2 \ = \ j_{s+1} \quad\hbox{for all}\quad 1\le s \le t-1 \, .
\end{equation}

\subsubsection*{Step 7:}  
Define 
$$
  A \ = \ \{ a_{s-1} = j_s - 1 \ | \ 2 \le s \le t\} \ \stackrel{\eqref{E:Cpf2}}{=} \ 
  \{ a_s = k_s + 1 \ | \ 1 \le s \le t-1 \} \, .
$$
Then \eqref{E:Cpf1} and \eqref{E:Cpf2} are equivalent to $\Delta(w) = \{ \a \in \Delta(\fg_1) \ | \ \a(\ttT_a)=0 \ \forall \ a \in A \} = \{ \a \in \Delta(\fg_1) \ | \ \a(\ttT_A)=0 \}$.\footnote{The claim requires our assumption that $\Delta_i(w)$ is nonempty for all $i \in I$, cf. Step 1.}  Thus,
\begin{equation} \label{E:XwC}
\renewcommand{\arraystretch}{1.3}
\begin{array}{rcl}
  X_w & = & \tGr(i_1 , a_1) \ \times \tGr( i_2-a_1 , a_2 - a_1) \ \times \ \cdots\\
  & & \cdots \ \times \
  \tGr(i_{t-1} - a_{t-2} , a_{t-1}-a_{t-2}) \ \times \
  \tLG( r - a_{t-1} , 2( r-a_{t-1} ) ) \,.
\end{array}
\end{equation}
Above, $\tLG(d,2d) \simeq C_d/P_d$ is the Lagrangian grassmannian of $d$--planes in $\bC^{2d}$ that are isotropic with respect to a nondegenerate, skew-symmetric bilinear form $\bvsigma$.  

This establishes the proposition in the case that $G_\bC = \tSp_{2r}\bC$.

\section{Bruhat order}\label{S:bo}

Let $w,w'\in W$.  Given a root $\a\in\Delta^+$, let $r_\a \in W$ denote the associated reflection.  Write $w \stackrel{\a}{\to} w'$ if $|w'| = |w|+1$ and $w' = r_\a w$.  The Bruhat order is a partial order on $W$ defined by $w \le w'$ if either $w = w'$ or there is a chain $w \stackrel{\a_1}{\to} w_1  \stackrel{\a_2}{\to} \cdots  \stackrel{\a_n}{\to} w'$.  The following is well-known; see, for example, \cite{MR0429933} and the references therein.

\begin{lemma*}
$w \le w'$ if and only if $X_w \subset X_{w'}$.
\end{lemma*}

\begin{lemma} \label{L:ww'}
Let $w,w' \in W^\varphi_\sI$.  Then $w \stackrel{\a}{\to} w'$ if and only if $\Delta(w') =\Delta(w) \cup \{\a\}$.
\end{lemma}

Together \eqref{E:nw} and Lemma \ref{L:ww'} yield Corollary \ref{C:ww'}.

\begin{corollary} \label{C:ww'}
Let $w,w' \in W^\varphi_\sI$.  Then $w \le w'$ if and only if $\Delta(w) \subset \Delta(w')$; equivalently, $\fn_w \subset \fn_{w'}$.
\end{corollary}

\begin{proof}[Proof of Lemma \ref{L:ww'}]
Suppose that $\Delta(w') =\Delta(w) \cup \{\a\}$.  So $\varrho_{w'} = \varrho_w + \a$, cf. \eqref{E:rhow}.  Noting that the $\Phi_w$ and $\langle\Phi_w\rangle$ of \cite{MR2532439} are our $\Delta(w)$ and $\varrho_w$, respectively, we see that $w \stackrel{\a}{\to} w'$ by \cite[Proposition 3.2.14(5)]{MR2532439}.

Conversely, suppose that $w \stackrel{\a}{\to} w'$.  Then $|w'| = |w| + 1$.  By \cite[(3.9)]{MR2532439}, we have 
\begin{subequations} \label{SE:ww'}
\begin{equation} \label{E:ww'1}
  \Delta(w') \ = \ \{\a\} \,\cup\, \left( \Delta(w) \cap \Delta(r_\a)\right) \,\cup\,
  r_\a \left( \Delta(w) \backslash \Delta(r_\a) \right) \, .
\end{equation}
Therefore, to establish $\Delta(w') = \Delta(w) \cup \{\a\}$, it suffices to show that 
\begin{equation} \label{E:ww'2}
  r_\a \left( \Delta(w) \backslash \Delta(r_\a) \right) \ = \ \Delta(w) \backslash \Delta(r_\a) \,.
\end{equation}
\end{subequations}
To that end, suppose that $\b\in \Delta(w) \backslash \Delta(r_\a)$, so that $r_\a\b = \b - n\a \in\Delta(w')$, for some $n\in\bZ$.  Since $\a,\b$ and $r_\a\b$ are elements of $\Delta(w) \cup \Delta(w') \subset \Delta(\fm_1)$, we have $1 = \a(\ttT_\varphi) = \b(\ttT_\varphi)$ and $1 = (r_\a\b)(\ttT_\varphi) = 1 - n$, by \eqref{E:grm}.  Thus, $n=0$ and $r_\a\b = \b$. This establishes \eqref{E:ww'2}, whence \eqref{E:ww'1} yields $\Delta(w') = \Delta(w) \cup \{\a\}$.
\end{proof}

\begin{remark*}
In general, $w \le w'$ is not equivalent to $\Delta(w) \subset \Delta(w')$.
\end{remark*}

\noindent Proposition \ref{P:SchubMax} is an immediate corollary to Lemmas \ref{L:ww'} and \ref{L:nwmax}.

\begin{lemma} \label{L:nwmax}
Assume $w \in W^\varphi_\sI$ is maximal with respect to the Bruhat order.  Then $\fn_w$ is a maximal IVHS.
\end{lemma}

\begin{proof}
By Lemma \ref{L:abel}, IVHS are subspaces $\fe \subset \fm_{-1}$ such that $[\fe,\fe]=0$.  Suppose that $\fn_w \subset \fe$.  Let $\z \in \fe$, and write $\z = \sum_{\a\in\Delta(\fm_1)} \z_{-\a}$ with $\z_{-\a} \in \fg_{-\a}$.  Since $\fe$ is abelian, we have $[\z,\fn_w]=0$.  Since $\fn_w$ is a direct sum of root spaces, this is possible if and only if $[\z_{-\a},\fn_w] = 0$ for all $\a$.  Equivalently, if $\z_{-\a}\not=0$, then the set $\Phi_\a = \{\a\} \cup \Delta(w) \subset \Delta(\fm_1)$ is closed, c.f. Remark \ref{R:Dw}(c).  Therefore, the lemma holds if and only if the set $A = \{ \a \in\Delta(\fm_1)\backslash\Delta(w) \ | \ \Phi_\a \hbox{ is closed}\}$ is empty.

Define a partial order on $A$ by declaring $\a_1 < \a_2$ if there exists $\b\in\Delta^+(\fg_0)$ such that $\a_1+\b=\a_2$.  If $A$ is nonempty, then there exists an element $\a\in A$ that is minimal with respect to this partial order.  By Remark \ref{R:Dw}(c), $\Delta^+\backslash\Delta(w)$ is closed.  This, along with the minimality of $\a$, implies that $\Delta^+\backslash\Phi_\a$ is also closed.  By Remark \ref{R:Dw}(c), there exists $w' \in W^\varphi$ such that $\Phi_\a = \Delta(w')$.  It follows from \eqref{E:nw} that $\fn_w \subset \fn_{w'}$, contradicting the maximality of $\fn_w$.  
\end{proof}

\bibliography{refs.bib}
\bibliographystyle{plain}
\end{document}